\documentclass[12pt, reqno]{amsart}
\usepackage{amsmath, amstext, amsbsy, amssymb, amscd}

\usepackage{amsxtra}
\usepackage{amscd}
\usepackage{amsthm}
\usepackage{amsfonts}
\usepackage{eucal}
\usepackage{youngtab}
\usepackage[table]{xcolor}
\usepackage{eucal}
\usepackage{tikz-cd}
\usepackage{mathrsfs}
\usepackage{cases}
\usepackage{epic}
\usepackage{tikz}
\usepackage{graphicx}
\usepackage{upgreek}
\usepackage{bm}
\usepackage{latexsym,todonotes}
\usepackage{pdflscape}
\usepackage[all]{xypic}
\usepackage[all]{xy}
\usepackage{color}
\usepackage{colordvi}
\usepackage{multicol}
\usepackage[normalem]{ulem}
\usepackage[linktocpage=true]{hyperref}
\hypersetup{colorlinks,linkcolor=blue,urlcolor=cyan,citecolor=blue}

\input prepictex
\input pictex
\input postpictex

\newtheorem{theorem}{Theorem}[section]

\newtheorem{lemma}[theorem]{Lemma}
\newtheorem{proposition}[theorem]{Proposition}
\newtheorem{corollary}[theorem]{Corollary}
\newtheorem{definition}[theorem]{Definition}

\newtheorem{remark}[theorem]{Remark}
\definecolor{A}{rgb}{.75,1,.75}

\numberwithin{equation}{section}

\newcommand{\ad}{\operatorname{ad}}
\newcommand{\C}{ \mathbb C }

\newcommand{\Daij}{D_{a;i,j}^{(r)}}
\newcommand{\Dpaij}{D_{a;i,j}^{\prime(r)}}
\newcommand{\Ebhk}{E_{b;h,k}^{(r)}}
\newcommand{\Fbkh}{F_{b;k,h}^{(r)}}
\newcommand{\del}{\delta}
\newcommand{\Del}{\Delta}
\newcommand{\End}{\operatorname{End}}
\newcommand{\W}{\mathcal{W}}

\newcommand{\ev}{\operatorname{ev}}
\newcommand{\str}{\operatorname{str}}

\newcommand{\bo}{\Upsilon}

\newcommand{\gl}{\mathfrak{gl}}

\newcommand{\g}{\mathfrak{g}}

\newcommand{\gr}{\operatorname{gr}}
\newcommand{\glMN}{\mathfrak{gl}_{M|N}}

\newcommand{\id}{\operatorname{id}}

\newcommand{\pa}[1]{|{#1}|}
\newcommand{\ovl}[1]{\overline{#1}}
\newcommand{\tp}{\operatorname{pa}}
\newcommand{\pr}{\operatorname{pr}}
\newcommand{\col}{\text{col}}
\newcommand{\row}{\text{row}}

\newcommand{\Z}{ \mathbb Z }

\begin{document}
\title[Finite $W$-superalgebras via super Yangians] {Finite $W$-superalgebras via super Yangians}

\author[Yung-Ning Peng]{Yung-Ning Peng}
\address{Department of Mathematics, National Central University, Chung-Li, Taiwan, 32054} \email{ynp@math.ncu.edu.tw}

\begin{abstract}
Let $e$ be an arbitrary even nilpotent element in the general linear Lie superalgebra $\glMN$ and let $\W_e$ be the associated finite $W$-superalgebra. 
Let $Y_{m|n}$ be the super Yangian associated to the Lie superalgebra $\gl_{m|n}$. A subalgebra of $Y_{m|n}$, called the shifted super Yangian and denoted by $Y_{m|n}(\sigma)$, is defined and studied.
Moreover, an explicit isomorphism between $\W_e$ and a quotient of $Y_{m|n}(\sigma)$ is established. 
\end{abstract}

\maketitle

\setcounter{tocdepth}{1}
\tableofcontents

\section{Introduction}\label{introd}
A finite $W$-algebra is an associative algebra determined by a pair $(\g,e)$, where $\g$ is a finite dimensional semisimple or reductive Lie algebra and $e$ is a nilpotent element in $\g$. In the extreme case when $e=0$, the corresponding finite $W$-algebra is the universal enveloping algebra $U(\g)$. In the other extreme case when $e$ is the {\em principal} (also called {\em regular}) nilpotent element, Kostant \cite{Ko} proved that the associated finite $W$-algebra is isomorphic to the center of the universal enveloping algebra. 

The study of finite $W$-algebra for a general $e$ was firstly developed systematically by Premet \cite{Pr1}, in which the modern terminologies were given and a proof of the long-standing Kac-Weisfeiler conjecture \cite{WK} was established.
Moreover, finite $W$-algebras can be understood as quantizations of Slodowy slices \cite{GG, Pr2}.
Since then, finite $W$-algebras have appeared in many branches of mathematics 
so that their behavior and properties can be explained from different viewpoints.
In recent years, the finite $W$-algebras have been intensively studied by various approaches; see the survey articles \cite{Ar, Lo, Wa} for details.

On the other hand, Yangians are certain non-commutative Hopf algebras that are important examples of quantum groups. 
They first appeared in physics in the work of Faddeev and his school around 80's concerning the quantum inverse scattering method. The term Yangian was given by Drinfeld \cite{Dr1} in honor of C.N. Yang and had been commonly used since then.
They were used to provide rational solutions of the Yang-Baxter equation; see the book \cite{Mo} for related topics and further applications of Yangians.

The connection between Yangians and finite $W$-algebras was firstly noticed by Ragoucy and Sorba \cite{RS} for type A Lie algebras. 
Suppose that the nilpotent element $e$ is {\em rectangular}, which means that all the Jordan blocks of $e$ are of the same size, say $\ell$. They showed that the associated finite $W$-algebra is isomorphic to the {\em Yangian of level } $\ell$, which is a certain quotient of the Yangian, considered by Cherednik \cite{C1, C2}.

This observation is further generalized by Brundan and Kleshchev \cite{BK2} to an arbitrary nilpotent $e\in\gl_N$ . The main result \cite[Theorem 10.1]{BK2} can be shortly described as follows: the finite $W$-algebra associated to a nilpotent $e\in\gl_{N}$ is isomorphic to a quotient of some subalgebra of the Yangian (called the {\em shifted Yangian}) associated to $\gl_n$, where $n$ is the number of Jordan blocks of $e$.
Moreover, an explicit realization of type A finite $W$-algebra by generators and relations is obtained. 
This provides a powerful tool for the study of finite $W$-algebras, their representations and further applications \cite{BGK, BK3, BK4}. 
It is also observed recently that the shifted Yangian can also be defined by different approaches together with generalizations and applications; see \cite{BFN, FKPRW, FPT, KWWY}.

The finite $W$-superalgebras are defined in a very similar way as the Lie algebra case except that the nilpotent element $e\in\mathfrak{g}$ is assumed to be {\em even} (with respect to the $\mathbb{Z}_2$-grading of the Lie superalgebra) with other modifications. 
In recent years, finite $W$-superalgebras and their representations have been 
extensively studied \cite{BBG, BGK, WZ1, WZ2, ZS1, ZS2, Zh} 
with different emphases.

The super Yangian associated to $\gl_{m|n}$, denoted by $Y_{m|n}$, was defined by Nazarov \cite{Na1} in terms of the {\em RTT presentation}.  
It is natural to seek for connections between finite $W$-superalgebras and super Yangians.
The very first result is obtained by Briot and Ragoucy \cite{BR}, saying that if the nilpotent element $e\in\gl_{M|N}$ is rectangular, then the associated finite $W$-superalgebra is isomorphic to a certain quotient of $Y_{m|n}$ called the {\em truncated super Yangian}, where $m$ and $n$ are the numbers of Jordan blocks of $e$ restricted to the even and odd spaces, respectively.
In recent years, there have been some results \cite{BBG, Pe2, Pe3} generalizing the above observation when the nilpotent element $e$ satisfies some assumptions, but for a general $e$ the problem remains to be open.

The goal of this article is to give a solution to this open problem, generally establishing the connection between the finite $W$-superalgebras and super Yangians for type A. 
That is, we explicitly 
give a superalgebra isomorphism between the finite $W$-superalgebra associated to an {\em arbitrary} even nilpotent element $e\in\gl_{M|N}$ and a quotient of a certain subalgebra of $Y_{m|n}$, obtaining a super analogue of the main result of \cite{BK2} for type A Lie superalgebras in full generality.

We shortly explain our approach, which is basically generalizing the arguments in \cite{BK2} to the general linear Lie superalgebras with suitable modifications and try to overcome all of the difficulties along the way. 
Although there are similarities between $\gl_N$ and $\gl_{M|N}$ and similarities between the associated super Yangians, some of the earlier approaches are no longer available in the case of Lie superalgebras. 
Moreover, there are other technical or conceptual obstacles that did not appear in the Lie algebra case.

Our first step is to define a subalgebra of $Y_{m|n}$ which we call the {\em shifted super Yangian} and denote by $Y_{m|n}(\sigma)$. 
To obtain this subalgebra, we need to use certain presentations of $Y_{m|n}$ called the {\em parabolic presentations}. 
Similar to the Lie algebra case \cite{BK1, Dr2}, the RTT presentation and the Drinfeld's presentation can be treated as special cases of the parabolic presentations.
There have been some results \cite{Go, Pe1} giving suitable presentations of $Y_{m|n}$, where the results \cite{BBG, Pe3} are in fact based on them. However, as noticed in \cite{BBG, Pe3}, they are no longer suitable presentations for the general case.
What we need is a kind of ``more" generalized parabolic presentation which works for {\em any} 01-sequence \cite{CW, FSS}, which is a parametrizing set controlling the parities of elements in $Y_{m|n}$. Such a presentation was recently obtained by the author in \cite{Pe4}. As a consequence, the shifted super Yangian $Y_{m|n}(\sigma)$ can be defined as a subalgebra of $Y_{m|n}$ generated by a certain subset of the generating set for the whole $Y_{m|n}$.

However, to establish the desired connection, we need not only the subalgebra but also its {\em presentation}.
By suitably modifying the defining relations for $Y_{m|n}$ found in \cite{Pe4}, we obtain a set of defining relations and hence a presentation of the shifted super Yangian $Y_{m|n}(\sigma)$. It should be emphasized that there are a few extra series of defining relations for $Y_{m|n}$ that did not appear in \cite{BK2}. Although we are able to guess the suitable modifications, it is highly non-trivial to check that our proposed relations actually hold in $Y_{m|n}(\sigma)$.
With some effort, one can eventually overcome this difficulty and a presentation of $Y_{m|n}(\sigma)$ is obtained, which allows one to define some homomorphisms called {\em baby comultiplications}, see \textsection \ref{babyco}, that will play important roles in the desired connection.

We further define the {\em shifted super Yangian of level} $\ell$, denoted by $Y_{m|n}^\ell(\sigma)$, as a quotient of $Y_{m|n}(\sigma)$ over some 2-sided ideal. Roughly speaking, $\sigma$ is a matrix recording the generating set for $Y_{m|n}(\sigma)$, while $\ell$ is an integer recording the size of the ideal in the quotient. 
It turns out that the data $\sigma$ and $\ell$ can be recorded by a diagram called {\em pyramid} \cite{EK, Ho}, which we denote by $\pi$, and it makes sense to set the notation $Y_\pi:=Y_{m|n}^\ell(\sigma)$. On the other hand, the diagram $\pi$ also determines a finite $W$-superalgebra which we denote by $\W_\pi$.

In \textsection \ref{Inv}, we introduce the notion of {\em super column height} so that one may explicitly write down some distinguished elements in $\W_\pi$ according to the diagram $\pi$ by modifying the description in \cite[\textsection 9]{BK2}. Our main result, Theorem \ref{main}, shows that the map sending the generators of $Y_\pi$ into these distinguished elements in $\W_\pi$ is an isomorphism of (filtered) superalgebras, obtaining a presentation of the finite $W$-superalgebra $\W_\pi$.

It is an interesting question to generalize the results in this article to other types of Lie superalgebras. 
In particular, there have been some results in the case of {\em queer Lie superalgebras} and their associated Yangians \cite{Na2} when the even nilpotent element is regular \cite{PS1} or rectangular \cite{PS2}, but it is still open in general.
We expect that the approaches in this article can be suitably modified to deal with the queer Lie superalgebra case for a general nilpotent element.


This article is organized as follows. In \textsection \ref{preW}, we set up our notations and  
recall some necessary background knowledge about finite $W$-superalgebras. 
In particular, the notion of pyramid with respect to a 01-sequence is recalled. 
In \textsection \ref{preY}, we recall some well-known facts about $Y_{m|n}$.

The shifted super Yangian $Y_{m|n}(\sigma)$ is defined in \textsection \ref{shiftY1}
by generators and relations, with the use of Drinfeld's presentation for $Y_{m|n}$, 
where some computations are relatively easier in this setting.
Then we show that $Y_{m|n}(\sigma)$ can be identified as a subalgebra of $Y_{m|n}$. Some basic properties of $Y_{m|n}(\sigma)$ are also provided. 

In \textsection \ref{shiftY2} we provide a more general approach, using the parabolic presentations for $Y_{m|n}$, to define $Y_{m|n}(\sigma)$ and establish the corresponding properties obtained in \textsection \ref{shiftY1} to parabolic case. In particular, the results in \textsection \ref{shiftY1} serve as initial steps of some induction arguments in the parabolic case.

\textsection \ref{babyco} is devoted to define the baby comultiplications that will help us establish the main result later. We explicitly write down their formulas and show that they are injective whenever they are defined. 

In \textsection \ref{canfil}, we introduce the canonical filtration of $Y_{m|n}(\sigma)$, which eventually corresponds to the Kazhdan filtration of finite $W$-superalgebras. 
The shifted super Yangian of level $\ell$ is defined in \textsection \ref{truncation} as a quotient of $Y_{m|n}(\sigma)$.

In \textsection \ref{Inv}, we explicitly define some distinguished elements in the universal enveloping algebra $U(\mathfrak{gl}_{M|N})$ that will eventually be identified as generators of our finite $W$-superalgebra. Our main result is stated and proved in \textsection \ref{mainsec}.

In this article, our field is the field of complex numbers $\mathbb{C}$, which can be replaced by any algebraically closed field of characteristic zero.
The term {\em subalgebra} always means a {\em sub-superalgebra}.
For homogeneous elements $x$ and $y$ in an associated superalgebra $L$, the {\em supercommutator} of $x$ and $y$ is defined by 
\[
\big[ x,y \big] = xy-(-1)^{\pa{x}\pa{y}}yx,
\]
where $\pa{x}$ is the $\Z_2$-grading of $x$ in $L$, called the {\em parity} of $x$. By convention, a homogeneous element $x$ is called $even$ (resp. $odd$) if $\pa{x}=\overline{0}$ (resp. $\overline{1}$). $L_{\overline{0}}$ and $L_{\overline{1}}$ denote the set of even and odd elements in $L$, respectively.


\subsection*{Acknowledgements}
The author is grateful to Shun-Jen Cheng and Weiqiang Wang for countless discussions and encouragement. 
A part of this article was finished during the author's visit to RIMS (Kyoto, Japan) in 2016. 
The author would like to thank the RIMS for providing an excellent working environment, 
and also thank Naoki Genra, Ryosuke Kodera and Hiraku Nakajima for stimulating discussions during the visit.
The visit is supported by the NCTS (Taipei, Taiwan), which is greatly acknowledged.
The author would also like to thank Lucy Gow, Alexander Molev and Alexander Tsymbaliuk for communication.
This work is partially supported by MOST grant 105-2628-M-008-004-MY4.


\section{Finite $W$-superalgebras and pyramids}\label{preW}

In this section, we recall the definition of a finite $W$-superalgebra, which is determined by an even nilpotent element $e$ and a semisimple element $h$ of $\glMN$. Also, a combinatorial object called {\em pyramid} is introduced so that we may encode $e$ and $h$ simultaneously by a diagram $\pi$.

Throughout this section, $\mathfrak{g}=\glMN$ is identified with the set of $(M+N)\times(M+N)$ matrices with the standard $\mathbb{Z}_2$-grading $\g=\g_{\ovl 0}\oplus \g_{\ovl 1}$ and $(\, \cdot \, ,\,\cdot\,)$ means the non-degenerate even supersymmetric $\g$-invariant bilinear form on $\g$ defined by 
$$(x,y):=\str(x y)$$
for all $x,y \in \g$, where $x y$ stands for the usual matrix product and $\str$ means the supertrace. Every elements of $\g$ appearing in any equations are considered homogeneous with respect to the $\mathbb{Z}_2$-grading unless specifically mentioned.

\subsection{Finite $W$-superalgebras of $\glMN$}

Let $e$ be an even nilpotent element in $\g$. It is well-known \cite{Ho,Wa} that there exists (not uniquely in general) a semisimple element $h\in\g$ such that $\ad h:\g\rightarrow\g$ gives a {\em good $\mathbb{Z}$-grading of $\g$ for $e$}, which means the following conditions are satisfied:
\begin{enumerate}
\item[(1)] $\ad h(e)=2e$,
\item[(2)] $\g=\bigoplus_{j\in\mathbb{Z}} \g(j)$, where $\g(j):=\{x\in\g|\ad h(x)=jx\}$,
\item[(3)] the center of $\g$ is contained in $\g(0)$,
\item[(4)] $\ad e:\g(j)\rightarrow\g(j+2)$ is injective for all $j\leq -1$,
\item[(5)] $\ad e:\g(j)\rightarrow\g(j+2)$ is surjective for all $j\geq -1$.
\end{enumerate}

In order to simplify the definition of finite $W$-superalgebras, throughout this article, we assume in addition that the $\mathbb{Z}$-grading is {\em even}; that is, $\g(i)=0$ for all $i\notin 2\mathbb{Z}$. We say $\langle e,h\rangle$ is a {\em good pair} if $\ad h$ gives an even good $\mathbb{Z}$-grading of $\g$ for $e$.

\begin{remark}
In general, a good pair may fail to exist in other types of classical Lie superalgebras \cite{Ho}. But for any even nilpotent $e\in\glMN$ we can always find some $h$ such that $\langle e,h\rangle$ is a good pair; see Theorem {\em \ref{Hoyt2}}. 
\end{remark}

Fix a good pair $\langle e,h\rangle$ in $\g$. Define the following subalgebras of $\g$ by
\begin{equation}\label{mpdef}
\mathfrak{p}:=\bigoplus_{j\geq 0}\g(j), \quad \mathfrak{m}:=\bigoplus_{j<0}\g(j).
\end{equation}
Define $\chi\in\g^*$ by 
$$\chi(y):=(y,e)\qquad   \forall y\in\g.$$ 
The restriction of $\chi$ on $\mathfrak{m}$ extends to a one dimensional $U(\mathfrak{m})$-module. Let $I_\chi$ be the left ideal of $U(\g)$ generated by 
$$\{a-\chi(a)  \,  |  \,  a\in\mathfrak{m}\}.$$ 

As a consequence of the PBW theorem for $U(\g)$, we have $U(\g)=I_\chi\oplus U(\mathfrak{p})$ together with the following identification 
$$U(\g)/I_\chi \cong U(\mathfrak{p})$$ 
by the natural projection $\pr_\chi:U(\g)\rightarrow U(\mathfrak{p})$. One defines the following $\chi$-twisted action of $\mathfrak{m}$ on $U(\mathfrak{p})$ by  
$$a\cdot y := \pr_\chi([a,y]),$$ 
for all $a\in\mathfrak{m}, y\in U(\mathfrak{p})$. 

The {\em finite W-superalgebra}, which we will usually omit the prefix ``finite" from now on, is defined to be the space of $\mathfrak{m}$-invariants in $U(\mathfrak{p})$ under the $\chi$-twisted action; to be explicit,
\begin{align*}
\mathcal{W}_{e,h}:=U(\mathfrak{p})^\mathfrak{m}=&\{y\in U(\mathfrak{p})\,\,|\,\, \pr_\chi ([a,y])=0, \forall a\in\mathfrak{m}\}\\
=&\{y\in U(\mathfrak{p})\,\,|\,\, \big(a-\chi(a)\big)y\in I_\chi, \forall a\in\mathfrak{m}\}.
\end{align*}
For example, if $e=0$, then $\chi=0$, $\g=\g(0)=\mathfrak{p}$ and $\mathfrak{m}=0$. Thus the associated $W$-superalgebra is exactly $U(\g)$.

At this point, it seems that the definition of a $W$-superalgebra depends on both of $e$ and $h$ in the good pair. In fact, the definition is independent of the choices of $h$ up to isomorphisms; see Remark \ref{912}.

\subsection{Pyramids and $W$-superalgebras}
We recall the notion of {\em pyramid} \cite{EK, Ho} as a convenient tool to present a good pair $\langle e,h\rangle$. We will identify a partition $\lambda=(\lambda_1,\lambda_2,\ldots)$ with its corresponding Young diagram in French style, which means that the diagrams are left-justified and the longest row is located in the bottom.

\begin{definition}
Let $\lambda$ be a Young diagram. A $pyramid$ is a diagram obtained by horizontally shifting the rows of $\lambda$ such that 
for each box not in the bottom row, there is exactly one box below it.
\end{definition}
For example, only the left-most diagram is a pyramid obtained from $\lambda=(3,2,1)$:
$$
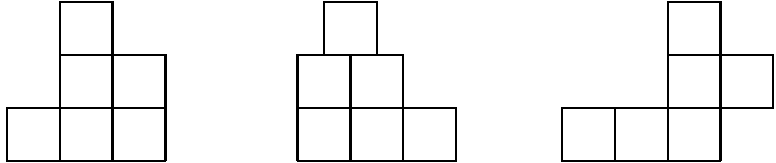
\begin{picture}(150,70)
\put(-70,10){\line(1,0){60}}
\put(-70,30){\line(1,0){60}}
\put(-50,50){\line(1,0){40}}
\put(-50,70){\line(1,0){20}}
\put(-10,10){\line(0,1){40}}
\put(-70,10){\line(0,1){20}}
\put(-50,10){\line(0,1){40}}
\put(-30,10){\line(0,1){60}}
\put(-50,50){\line(0,1){20}}

\put(40,10){\line(1,0){60}}
\put(40,30){\line(1,0){60}}
\put(40,50){\line(1,0){40}}
\put(50,70){\line(1,0){20}}
\put(50,50){\line(0,1){20}}
\put(70,50){\line(0,1){20}}

\put(100,10){\line(0,1){20}}
\put(40,10){\line(0,1){40}}
\put(60,10){\line(0,1){40}}
\put(80,10){\line(0,1){40}}

\put(140,10){\line(1,0){60}}
\put(140,30){\line(1,0){80}}
\put(180,70){\line(1,0){20}}
\put(180,50){\line(1,0){40}}

\put(200,10){\line(0,1){60}}
\put(140,10){\line(0,1){20}}
\put(160,10){\line(0,1){20}}
\put(180,10){\line(0,1){60}}
\put(220,30){\line(0,1){20}}
\end{picture}
$$

Let $V=V_{\ovl{0}}\oplus V_{\ovl{1}}$ be a $\Z_2$-graded vector space with $\dim V_{\ovl{0}}=M$ and $\dim V_{\ovl{1}}=N$. 
We identify $\g=\glMN$ with $\End V$ and one has the following identification for $\g_{\ovl 0}$
$$\g_{\ovl{0}}\cong \End (V_{\ovl{0}})\oplus \End (V_{\ovl{1}}).$$
As a result, an even nilpotent element $e\in\glMN$ can be thought as a sum of two nilpotent element
$e=e_{\ovl{0}}+e_{\ovl{1}}$, where $e_{i}\in \End V_{i}$ for $i\in\{\ovl{0},\ovl{1}\}$. Thus we may describe $e$ by two Young diagrams $\mu$ and $\nu$ corresponding to the Jordan types of $e_{\ovl{0}}$ and $e_{\ovl{1}}$, respectively. 

For example, the diagram
$$\Yvcentermath1 \young(++,+++)\oplus \young(--,----)$$
represents an even nilpotent element in $\gl_{5|6}$, which is a sum of a nilpotent element in $\End \C^5$ with Jordan type $\mu=(3,2)$ and a nilpotent element in $\End \C^6$ with Jordan type $\nu=(4,2)$. We put $+$ and $-$ in the boxes because we now stack the two diagrams together to obtain a new Young diagram, and we need to track from which diagram the boxes originally are.

For example, there are two possibilities if we stack the above two Young diagrams together to obtain one Young diagram:
\begin{equation}\label{2.2}
\Yvcentermath1 \young(++,--,+++,----)\qquad\qquad \young(--,++,+++,----)
\end{equation}

\begin{remark}\label{2.4}
The pyramids in this article correspond to certain even nilpotent elements in $\gl_{M|N}$, hence the following condition always holds:
\begin{center}
every boxes in a row have the same $+$ or $-$ labeling.
\end{center}
\end{remark}

As one may expect, we shift the rows of the stacked Young diagram to obtain a pyramid. For example, we take the right diagram in (\ref{2.2}) and list all possibilities below:
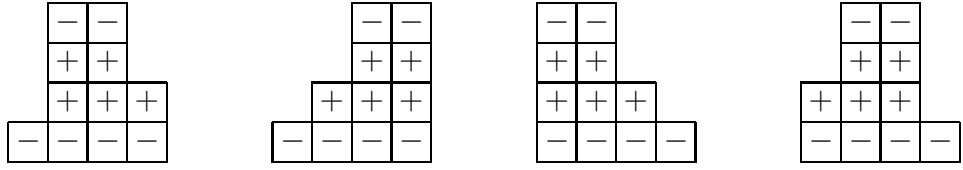
\[
{\begin{picture}(90, 60)%
\put(-115,-10){\line(1,0){60}}
\put(-115,5){\line(1,0){60}}
\put(-100,20){\line(1,0){45}}
\put(-100,35){\line(1,0){30}}
\put(-100,50){\line(1,0){30}}
\put(-115,-10){\line(0,1){15}}
\put(-100,-10){\line(0,1){60}}
\put(-85,-10){\line(0,1){60}}
\put(-70,-10){\line(0,1){60}}
\put(-55,-10){\line(0,1){30}}
\put(-112,-5){$-$}\put(-97,-5){$-$}\put(-82,-5){$-$}\put(-67,-5){$-$}
\put(-97,10){$+$}\put(-82,10){$+$}\put(-67,10){$+$}
\put(-97,25){$+$}\put(-82,25){$+$}
\put(-97,40){$-$}\put(-82,40){$-$}

\put(-15,-10){\line(1,0){60}}
\put(-15,5){\line(1,0){60}}
\put(0,20){\line(1,0){45}}
\put(15,35){\line(1,0){30}}
\put(15,50){\line(1,0){30}}
\put(-15,-10){\line(0,1){15}}
\put(0,-10){\line(0,1){30}}
\put(15,-10){\line(0,1){60}}
\put(30,-10){\line(0,1){60}}
\put(45,-10){\line(0,1){60}}
\put(-12,-5){$-$}\put(3,-5){$-$}\put(18,-5){$-$}\put(33,-5){$-$}
\put(3,10){$+$}\put(18,10){$+$}\put(33,10){$+$}
\put(18,25){$+$}\put(33,25){$+$}
\put(18,40){$-$}\put(33,40){$-$}

\put(85,-10){\line(1,0){60}}
\put(85,5){\line(1,0){60}}
\put(85,20){\line(1,0){45}}
\put(85,35){\line(1,0){30}}
\put(85,50){\line(1,0){30}}
\put(85,-10){\line(0,1){60}}
\put(100,-10){\line(0,1){60}}
\put(115,-10){\line(0,1){60}}
\put(130,-10){\line(0,1){30}}
\put(145,-10){\line(0,1){15}}
\put(88,-5){$-$}\put(103,-5){$-$}\put(118,-5){$-$}\put(133,-5){$-$}
\put(88,10){$+$}\put(103,10){$+$}\put(118,10){$+$}
\put(88,25){$+$}\put(103,25){$+$}
\put(88,40){$-$}\put(103,40){$-$}

\put(185,-10){\line(1,0){60}}
\put(185,5){\line(1,0){60}}
\put(185,20){\line(1,0){45}}
\put(200,35){\line(1,0){30}}
\put(200,50){\line(1,0){30}}
\put(185,-10){\line(0,1){30}}
\put(200,-10){\line(0,1){60}}
\put(215,-10){\line(0,1){60}}
\put(230,-10){\line(0,1){60}}
\put(245,-10){\line(0,1){15}}
\put(188,-5){$-$}\put(203,-5){$-$}\put(218,-5){$-$}\put(233,-5){$-$}
\put(188,10){$+$}\put(203,10){$+$}\put(218,10){$+$}
\put(203,25){$+$}\put(218,25){$+$}
\put(203,40){$-$}\put(218,40){$-$}

\end{picture}}\\[4mm]
\]

Soon we will see (Theorem \ref{Hoyt2}) that each of these pyramids represents a good pair $\langle e,h\rangle$ in $\gl_{5|6}$. Moreover, these are {\em all} good pairs we could have for that given $e\in\gl_{5|6}$.

Now we do the other way around: obtaining a good pair $\langle e,h\rangle $ from a given pyramid $\pi$ satisfying the condition described in Remark \ref{2.4}.
Assume that we have $M$ (resp. $N$) boxes labeled with $+$ (resp. $-$) in $\pi$, where they came from the Young diagram of $e_{\ovl{0}}\in\gl_{M|0}$ (resp. $e_{\ovl{1}}\in\gl_{0|N}$).
We enumerate those $``+"$ boxes by $1,2,\ldots, M$ down columns from left to right, and enumerate those $``-"$ boxes by $\ovl{1},\ovl{2},\ldots,\ovl{N}$ by the same rule.

Next we imagine that each box of $\pi$ is of size $2\times 2$ and our pyramid is built on the $x$-axis, where the center of $\pi$ is exactly located above the origin. For instance:

\begin{equation}\label{exp}
\pi={\begin{picture}(90, 80)%
\put(-80,-35){$x$-coordinates:}
\put(15,-10){\line(1,0){80}}
\put(15,10){\line(1,0){80}}
\put(35,30){\line(1,0){60}}
\put(35,50){\line(1,0){40}}
\put(35,70){\line(1,0){40}}
\put(15,-10){\line(0,1){20}}
\put(35,-10){\line(0,1){80}}
\put(55,-10){\line(0,1){80}}
\put(75,-10){\line(0,1){80}}
\put(95,-10){\line(0,1){40}}
\put(23,-3){$\ovl{1}$}\put(43,-3){$\ovl{3}$}\put(63,-3){$\ovl{5}$}\put(83,-3){$\ovl{6}$}
\put(43,16){$2$}\put(63,16){$4$}\put(83,16){$5$}
\put(43,35){$1$}\put(63,35){$3$}
\put(43,55){$\ovl{2}$}\put(63,55){$\ovl{4}$}
\put(0,-20){\line(1,0){110}}
\put(53,-23){$\bullet$}
\put(63,-35){$1$}\put(83,-35){$3$}
\put(35,-35){$-1$}\put(15,-35){$-3$}
\end{picture}}\\[15mm]
\end{equation}

Let $I=\{1<\ldots<M<\ovl{1}<\ldots<\ovl{N}\}$ be an ordered index set and let $\{v_i|i\in I\}$ be the standard basis of $\mathbb{C}^{M|N}$ with respect to the following order 
\[v_i<v_j \text{ if } i<j \text{ in } I.\] 
Let $\{e_{i,j}\,|\, i,j\in I\}$ denote the elementary matrices in $\gl_{M|N}$.
Define the element 
\begin{equation}\label{edef}
e_\pi:=\sum_{\substack{  {\tiny \young(ij)} \in \pi}}e_{i,j}\in \g_{\ovl{0}},
\end{equation}
where the sum is taken over all adjacent pairs $\young(ij)$ appeared in $\pi.$ 

Let $\text{col}_x(i)$ denote the $x$-coordinate of the center of the box numbered with $i\in I$, which must be an integer by our construction. Define the following diagonal matrix
\begin{equation}\label{hdef}
h_\pi:=-\text{diag}\big(\text{col}_x(1), \ldots, \text{col}_x(M), \text{col}_x(\ovl{1}),\ldots, \text{col}_x(\ovl{N})\big)
\end{equation}
For example, the elements $e_\pi$ and $h_\pi$ associated to the pyramid $\pi$ in (\ref{exp}) are
\begin{align*}
e_\pi&=e_{13}+e_{24}+e_{45}+e_{\ovl{2} \,\ovl{4}}+e_{\ovl{1}  \, \ovl{3}}+e_{\ovl{3} \, \ovl{5}}+e_{ \ovl{5} \, \ovl{6}},\\
h_\pi&=\text{diag}(1,1,-1,-1,-3,3,1,1,-1,-1,-3).
\end{align*}
It is easy to check that $\langle e_\pi,h_\pi\rangle$ forms a good pair.

Note that if we horizontally shift the rows of $\pi$ to obtain another pyramid $\vec\pi$, then $e_\pi=e_{\vec\pi}$ but $h_\pi\neq h_{\vec\pi}$. The following theorem implies that every even good $\mathbb{Z}$-gradings for $e_\pi$ can be obtained by shifting the rows of $\pi$.

\begin{theorem}\cite[Theorem 7.2]{Ho}\label{Hoyt2}
Let $\pi$ be a pyramid. Let $e=e_\pi$ and $h=h_\pi$ be the elements in $\glMN$ defined by {\em (\ref{edef})} and {\em (\ref{hdef})}, respectively. Then $\langle e,h\rangle$ forms a good pair for $e$. Moreover, any good pair for $e$ is of the form $\langle e,h_{\vec\pi}\rangle$ where $\vec\pi$ is some pyramid obtained by shifting rows of $\pi$ horizontally.
\end{theorem}

In other words, Theorem~\ref{Hoyt2} {\em classifies} all of the even good $\mathbb{Z}$-gradings of $\gl_{M|N}$ for {\em any} even nilpotent $e$. 
(In fact, \cite[Theorem 7.2]{Ho}\label{Hoyt} classifies {\em all} good $\mathbb{Z}$-gradings, not just those even good $\mathbb{Z}$-gradings considered in this article.)
As a consequence, for a given pyramid $\pi$, it makes sense to denote the $W$-superalgebra associated to the good pair $\langle e_\pi$, $h_\pi\rangle$ simply by $\W_\pi:=\W_{e_\pi,h_\pi}$. 

\begin{remark}\label{swap}
If we permute the rows with the same length of $\pi$ to obtain a new pyramid $\pi^\prime$, then we have $e_\pi=e_{\pi^{\prime}}$ and $h_\pi=h_{\pi^{\prime}}$. For example, the two Young diagrams in {\em (\ref{2.2})} give us exactly the same list of good pairs by shifting their rows.
\end{remark}

We label the columns of $\pi$ from left to right by $1,\ldots, \ell$. For any $i\in I$, let col($i$) denote the column where $i$ appear. The {\em Kazhdan filtration} of $U(\g)$
\[
\cdots\subseteq F_dU(\g) \subseteq F_{d+1}U(\g)\subseteq \cdots
\] 
is defined by setting
\begin{equation}\label{degdef}
\deg (e_{i,j}):= \text{col}(j)-\text{col}(i)+1
\end{equation}
for each $i,j \in I$, where $F_dU(\g)$ denotes the span of all supermonomials $e_{i_1,j_1}\cdots e_{i_s,j_s}$ for $s\geq 0$ with $\sum_{k=1}^s$ deg $(e_{i_k,j_k})\leq d$. Let $\gr U(\g)$ denote the graded superalgebra associated to the Kazhdan filtration. A natural grading on $\W_\pi$ is induced from the  projection $\g\twoheadrightarrow\mathfrak{p}$ and we denote by $\gr \W_\pi$ the associated graded superalgebra. 

Let $\g^e$ denote the centralizer of $e$ in $\g$ and let $S(\g^e)$ denote the associated supersymmetric superalgebra. 
The same setting (\ref{degdef}) defines the Kazhdan filtration on $S(\g^e)$. 
The following result still holds in our case since our pyramid $\pi$ satisfies the condition in Remark \ref{2.4}.

\begin{proposition}\cite[Remark 3.11]{Zh}\label{dimprop}
$S(\g^e)$ and $\gr \W_\pi$ are isomorphic as graded superalgebras.
\end{proposition}

\subsection{Shift matrix}
We give an alternative way to describe a pyramid.
An $(m+n)\times(m+n)$ matrix $\sigma=(s_{i,j})_{1\leq i,j\leq m+n}$ is called a {\em shift matrix} if its entries are non-negative integers satisfying the following condition
\begin{equation}\label{sijk}
s_{i,j} + s_{j,k} = s_{i,k},
\end{equation}
whenever $|i-j|$+$|j-k|$=$|i-k|$. 
For example, the following matrix is a shift matrix:
\begin{equation}\label{sigex}
\sigma = \left(\begin{array}{llllll}
0&1&2&2&3&3\\
0&0&1&1&2&2\\
1&1&0&0&1&1\\
1&1&0&0&1&1\\
3&3&2&2&0&0\\
4&4&3&3&1&0
\end{array}\right)
\end{equation}

\begin{lemma}\label{shiftdet}
The follow facts hold for a shift matrix $\sigma=(s_{i,j})_{1\leq i,j\leq m+n}$.
\begin{enumerate}
\item If the entries in the last column $\{ s_{i,m+n} \,|\, 1\leq i\leq m+n\}$ are known, then the whole upper-triangular part of $\sigma$ is determined.
\item If the entries in the upper-diagonal $\{ s_{i,i+1} \,|\, 1\leq i <m+n\}$ are known, then the whole upper-triangular part of $\sigma$ is determined.
\item If the entries in the last row $\{ s_{m+n,i} \,|\, 1\leq i\leq m+n\}$ are known, then the whole lower-triangular part of $\sigma$ is determined.
\item If the entries in the lower-diagonal $\{ s_{i+1,i} \,|\, 1\leq i <m+n \}$ are known, then the whole lower-triangular part of $\sigma$ is determined.
\end{enumerate}
\end{lemma}
\begin{proof}
By $(\ref{sijk})$.
\end{proof}

In our superalgebra setting, we need to record the $\pm$-labeling of each row in our pyramid, so we introduce the following terminology.
Let $m,n\in\mathbb{Z}_{\geq 0}$. A {\em $0^m1^n$-sequence}, or {\em $01$-sequence} for short, is an ordered sequence $\bo$ consisting of $m$ $0$'s and $n$ $1$'s. For $1\leq i\leq m+n$, the i-th digit of $\bo$ is denoted by $|i|$.

Suppose that $\sigma\in M_{m+n}(\mathbb{Z}_{\geq 0})$ is a shift matrix. Let $\ell$ be an integer such that $\ell>s_{1,m+n}+s_{m+n,1}$ and let $\bo$ be a fixed $0^m1^n$-sequence. Then one can obtain a pyramid $\pi$, with $m$ (resp. $n$) rows labeled by ``+" (resp. ``$-$") and the bottom row consisting of $\ell$ boxes, from the triple $(\sigma,\ell,\bo)$ by the following fashion.

Start with a rectangular Young diagram consisting of $m+n$ rows and $\ell$ columns, which we denote by $\Xi$. 
We number the rows of $\Xi$ from top to bottom by $1,2,\ldots, m+n$. For each $1\leq i\leq m+n$, we label every boxes in the $i$-th row of $\Xi$ by $``+"$ if $\pa{i}=0$, and by $``-"$ if $\pa{i}=1$.

Next we obtain our pyramid from this rectangle. Consider the entries in the last row and the last column of $\sigma$: $\{ s_{m+n,i} \,|\, 1\leq i\leq m+n \}$ and $\{ s_{i,m+n} \,|\, 1\leq i\leq m+n \}$.
For each $1\leq j\leq m+n$, we erase the leftmost $s_{m+n,j}$ boxes and the rightmost $s_{j,m+n}$ boxes in the $j$-th row of $\Xi$. By (\ref{sijk}), the resulted diagram is a pyramid which has $\ell$ boxes in the bottom row and $\ell-s_{m+n,1}-s_{1,m+n}$ boxes in the top row.
For example, take $\ell=8$ and let $\sigma$ be the one given in (\ref{sigex}) with $\bo=100010$, the resulted pyramid $\pi$ is
$$
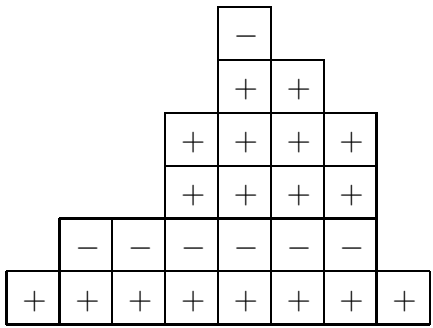
\begin{picture}(200, 125)%
\put(20,0){\line(1,0){160}}
\put(20,20){\line(1,0){160}}
\put(40,40){\line(1,0){120}}
\put(80,60){\line(1,0){80}}
\put(80,80){\line(1,0){80}}
\put(100,100){\line(1,0){40}}
\put(100,120){\line(1,0){20}}
\put(20,0){\line(0,1){20}}
\put(40,0){\line(0,1){40}}
\put(60,0){\line(0,1){40}}
\put(80,0){\line(0,1){80}}
\put(100,0){\line(0,1){120}}
\put(120,0){\line(0,1){120}}
\put(140,0){\line(0,1){100}}
\put(160,0){\line(0,1){80}}
\put(180,0){\line(0,1){20}}
\put(106,106){$-$}

\put(106,86){$+$}\put(126,86){$+$} 

\put(106,66){$+$}\put(126,66){$+$} 
\put(86,66){$+$}\put(146,66){$+$} 

\put(106,46){$+$}\put(126,46){$+$} 
\put(86,46){$+$}\put(146,46){$+$} 

\put(106,26){$-$}\put(126,26){$-$} 
\put(86,26){$-$}\put(146,26){$-$} 
\put(46,26){$-$}\put(66,26){$-$} 

\put(106,6){$+$}\put(126,6){$+$} 
\put(86,6){$+$}\put(146,6){$+$} 
\put(46,6){$+$}\put(66,6){$+$} 
\put(26,6){$+$}\put(166,6){$+$} 

\end{picture}
$$

Conversely, given a pyramid $\pi$ which represents a good pair. Let $\ell$ be the number of boxes in the bottom of $\pi$ and let $m$ and $n$ be the numbers of rows of $\pi$ labeled by $+$ and $-$, respectively. We number the rows of $\pi$ from top to bottom by $1,2,\ldots, m+n$ as before. 
Since $\pi$ satisfies the condition in Remark \ref{2.4}, we may obtain a $0^m1^n$-sequence $\bo$ by assigning the $i$-th digit of $\bo$ to be 0 (resp. 1) if the boxes in the $i$-th row are labeled by $``+"$ (resp. $``-"$).

For each $1\leq i\leq m+n$, define the number $s_{m+n,i}$ (resp. $s_{i,m+n}$) to be the number of missing boxes on the left-hand side (resp. right-hand side) of the $i$-th row of $\pi$ in a rectangular diagram $\Xi$ of size $(m+n)\times \ell$. This gives us the entries of the last row and the last column of $\sigma$ and hence we are able to recover the whole $\sigma$ by Lemma \ref{shiftdet}. The discussion above is summarized in the following proposition.

\begin{proposition}\label{SPbij}
Let $S$ be the set of triples $(\sigma, \ell, \bo)$ where $\sigma$ is a shift matrix of size $m+n$, $\ell>s_{m+n,1}+s_{1,m+n}$ is an integer and $\bo$ is a $0^m1^n$-sequence. Let $P$ be the set of all pyramids $\pi$ such that $\pi$ has $m$ (resp. $n$) rows labeled by $+$ (resp. $-$) and $\ell$ columns. Then there exists a bijection between $S$ and $P$.
\end{proposition}

Roughly speaking, $\sigma$ determines the {\em shape and height}, $\ell$ determines the {\em width} and $\bo$ determines the {\em $\pm$-labeling} of $\pi$ and vise versa.

The following proposition is a super analogue of a well-known result about $\g^e$. 
Since our pyramid $\pi$ satisfies the condition described in Remark \ref{2.4}, it is similar to the Lie algebra case as remarked in \cite{BBG}.
\begin{proposition}\label{counting2}
Let $\pi$ be a pyramid with row lengths $\lbrace p_i \, | \, 1\leq i\leq m+n\rbrace$, where the rows are labeled from top to bottom.
Let $\sigma=(s_{i,j})_{1\leq i,j\leq m+n}$ be the associated shift matrix of $\pi$ in the triple $(\sigma,\ell,\bo)$. Let $e=e_\pi$ be the nilpotent element defined by {\em (\ref{edef})}. 
Let $M$ (resp. $N$) be the number of boxes of $\pi$ labeled in $+$ (resp. $-$).
For all $1\leq i,j\leq m+n$ and $r>0$, define
\[
c_{i,j}^{(r)}:=
\sum_{\substack{h,k\in I \\ row(h)=i, \,\, row(k)=j\\ \col(k)-\col(h)=r-1}}e_{h,k}\in \g=\glMN.
\]
Then $\lbrace c_{i,j}^{(r)} \, | \,  1\leq i,j\leq m+n, \, s_{i,j}<r\leq s_{i,j}+p_{min (i,j)}\rbrace$ forms a linear basis for $\g^e$.
\end{proposition}

\section{The super Yangian $Y_{m|n}$}\label{preY}
In this section, we recall some well-known facts about the super Yangian associated to the general linear Lie superalgebra.
\subsection{RTT presentations of $Y_{m|n}$}
\begin{definition}\cite{Na1}
For a given 01-sequence $\bo$, the Yangian associated to the general linear Lie superalgebra $\gl_{m|n}$, denoted by $Y_{m|n}$, is the associative $\mathbb{Z}_2$-graded algebra with unity generated over $\mathbb{C}$ by the RTT generators 
\begin{equation}\label{RTTgen}
\left\lbrace t_{i,j}^{(r)}\,| \; 1\le i,j \le m+n; r\ge 1\right\rbrace,
\end{equation}
subject to following RTT relations:
\begin{equation}\label{RTT}
\big[ t_{i,j}^{(r)}, t_{h,k}^{(s)} \big] = (-1)^{\pa{i}\,\pa{j} + \pa{i}\,\pa{h} + \pa{j}\,\pa{h}}
\sum_{g=0}^{\mathrm{min}(r,s) -1} \Big( t_{h,j}^{(g)}\, t_{i,k}^{(r+s-1-g)} -   t_{h,j}^{(r+s-1-g)}\, t_{i,k}^{(g)} \Big),
\end{equation}
where the parity of $t_{i,j}^{(r)}$ is defined by $\pa{i}+\pa{j}$ {\em(mod 2)}. By convention, we set $t_{i,j}^{(0)}:=\del_{ij}$.
\end{definition}

The original definition in \cite{Na1} corresponds to the case when $\bo$ is the {\em standard} 01-sequence, which is defined as
\[
\bo^{st}:=\stackrel{m}{\overbrace{0\ldots0}}\,\stackrel{n}{\overbrace{1\ldots1}}.
\]
As observed in \cite{Pe2, Ts}, up to isomorphism, the definition of $Y_{m|n}$ is independent of the choices of $\bo$ so we often omit it in our notation when appropriate.

For each $1\leq i,j\leq m+n$, define the formal series 
\[
t_{i,j}(u):= \sum_{r\geq 0} t_{i,j}^{(r)}u^{-r} \in Y_{m|n}[[u^{-1}]].
\]
It is well-known \cite{Na1} that $Y_{m|n}$ is a Hopf-superalgebra. In particular, the comultiplication 
$\Del:Y_{m|n}\rightarrow Y_{m|n}\otimes Y_{m|n}$ can be nicely described as 
\begin{equation}\label{Del}
\Del(t_{i,j}^{(r)})=\sum_{s=0}^r \sum_{k=1}^{m+n} t_{i,k}^{(r-s)}\otimes t_{k,j}^{(s)}.
\end{equation}
Moreover, there exists a surjective homomorphism $$\ev:Y_{m|n}\rightarrow U(\gl_{m|n})$$ called the {\em evaluation homomorphism}, defined by
\begin{equation}\label{ev}
\ev\big(t_{i,j}(u)\big):= \del_{ij} + (-1)^{|i|} e_{ij}u^{-1},
\end{equation}
where $e_{ij}\in\gl_{m|n}$ means the elementary matrix. 

The following proposition gives a PBW basis for $Y_{m|n}$ in terms of the RTT generators, where the proof in \cite{Go} works perfectly for any fixed $\bo$.
\begin{proposition}\cite[Theorem 1]{Go}\label{PBWSY}
The set of supermonomials in the following elements 
\[
\left\lbrace t_{i,j}^{(r)}\, |\, 1\leq i,j\leq m+n,  r\geq 1 \right\rbrace
\] taken in some fixed order
forms a linear basis for $Y_{m|n}$.
\end{proposition}

Define the $loop$ $filtration$ on $Y_{m|n}$
\begin{equation}\label{filt2}\notag
L_0 Y_{m|n} \subseteq L_1 Y_{m|n} \subseteq L_2 Y_{m|n} \subseteq \cdots
\end{equation}
by setting $\deg t_{ij}^{(r)}=r-1$ for each $r\geq 1$ and letting $L_kY_{m|n}$ be the span of all supermonomials of the form 
$$t_{i_1j_1}^{(r_1)}t_{i_2j_2}^{(r_2)}\cdots t_{i_sj_s}^{(r_s)}$$
with total degree not greater than $k$. We denote by $\gr^LY_{m|n}$ the associated graded superalgebra.

Let $\gl_{m|n}[x]$ denote the {\em loop superalgebra} $\gl_{m|n}\otimes \mathbb{C}[x]$, where a basis is given by 
$$\lbrace e_{ij}x^r \,|\, 1\leq i,j\leq m+n, r\geq 0\rbrace.$$ 
Let $U(\gl_{m|n}[x])$ denote its universal enveloping algebra with the natural filtration and grading given by 
$$\deg e_{ij}x^r:=r.$$
The following corollary is a consequence of Proposition~\ref{PBWSY}.
\begin{corollary}\cite[Corollary 1]{Go}\label{Yloop}
The function $Y_{m|n}\rightarrow U(\gl_{m|n}[x])$ given by 
$$t_{ij}^{(r)}\mapsto (-1)^{\pa{i}}e_{ij}x^{r-1}$$ 
induces an isomorphism $\gr^LY_{m|n}\cong U(\gl_{m|n}[x])$ of graded superalgebras.

\end{corollary}

\subsection{Parabolic generators of $Y_{m|n}$}
In this subsection, we give another generating set for $Y_{m|n}$. 
Eventually it will allow us to define a certain subalgebra of $Y_{m|n}$ which can not be observed by the earlier RTT-presentation except for some special cases.

Firstly we introduce a convenient shorthand notation. Let $\mu=(\mu_1,\ldots,\mu_z)$ be a given composition of $m+n$ with length $z$ and let $\bo$ be a fixed $0^m1^n$-sequence. We break $\bo$ into $z$ subsequences according to $\mu$; that is, 
\[\bo=\bo_1\bo_2\ldots\bo_z,\] 
where $\bo_1$ is the subsequence consisting of the first $\mu_1$ digits of $\bo$, $\bo_2$ is the subsequence consisting of the next $\mu_2$ digits of $\bo$, and so on.
For example, if we have $\bo=011100011$ and $\mu=(2,4,3)$, then
\[
\bo=\overbrace{01}^{\bo_1} \, \overbrace{1100}^{\bo_2} \, \overbrace{011}^{\bo_3}.
\]

For each $1\leq a\leq z$, let $p_a$ and $q_a$ denote the number of $0$'s and $1$'s in $\bo_a$, respectively.
For a fixed $1\leq a\leq z$ and each value of $i=1,2,\ldots,\mu_a$, we define the {\em restricted parity} $\pa{i}_a$ by
\begin{center} $\pa{i}_a$:= the $i$-th digits of $\bo_a$, \end{center}
or equivalently
\begin{equation}\label{respa}
\pa{i}_a=\pa{\sum_{j=1}^{a-1}\mu_j+i}.
\end{equation}

Define the $(m+n)\times (m+n)$ matrix with entries in $Y_{m|n}[[u^{-1}]]$ by
\[
T(u):=\Big( t_{i,j}(u) \Big)_{1\leq i,j\leq m+n}
\]
Note that the leading minors of the matrix $T(u)$ are always invertible and hence the matrix $T(u)$ possesses a Gauss decomposition with respect to $\mu$; that is,
\begin{equation}\label{T=FDE}
T(u) = F(u) D(u) E(u)
\end{equation}
for unique {\em block matrices} $D(u)$, $E(u)$ and $F(u)$ of the form
$$
D(u) = \left(
\begin{array}{cccc}
D_{1}(u) & 0&\cdots&0\\
0 & D_{2}(u) &\cdots&0\\
\vdots&\vdots&\ddots&\vdots\\
0&0 &\cdots&D_{z}(u)
\end{array}
\right),
$$

$$
E(u) =
\left(
\begin{array}{cccc}
I_{\mu_1} & E_{1,2}(u) &\cdots&E_{1,z}(u)\\
0 & I_{\mu_2} &\cdots&E_{2,z}(u)\\
\vdots&\vdots&\ddots&\vdots\\
0&0 &\cdots&I_{\mu_{z}}
\end{array}
\right),\:
$$

$$
F(u) = \left(
\begin{array}{cccc}
I_{\mu_1} & 0 &\cdots&0\\
F_{2,1}(u) & I_{\mu_2} &\cdots&0\\
\vdots&\vdots&\ddots&\vdots\\
F_{z,1}(u)&F_{z,2}(u) &\cdots&I_{\mu_{z}}
\end{array}
\right),
$$
where
\begin{align}
\label{3.7} D_a(u) &=\big(D_{a;i,j}(u)\big)_{1 \leq i,j \leq \mu_a},\\
\label{3.8} E_{a,b}(u)&=\big(E_{a,b;h,k}(u)\big)_{1 \leq h \leq \mu_a, 1 \leq k \leq \mu_b},\\
\label{3.9} F_{b,a}(u)&=\big(F_{b,a;k,h}(u)\big)_{1 \leq k \leq \mu_b, 1 \leq h \leq \mu_a},
\end{align}
are $\mu_a \times \mu_a$,
$\mu_a \times \mu_b$
and  $\mu_b \times\mu_a$ matrices, respectively, for all $1\le a\le z$ in (\ref{3.7})
and all $1\le a<b\le z$ in (\ref{3.8}) and (\ref{3.9}).
In fact, these matrices can be explicitly obtained by {\em quasideterminants} (cf. \cite{GR}).
 
Since all of the submatrices $D_{a}(u)$'s are invertible, it allows one to define the $\mu_a\times\mu_a$ matrix
$D_a^{\prime}(u)=\big(D_{a;i,j}^{\prime}(u)\big)_{1\leq i,j\leq \mu_a}$ by
\begin{equation*}
D_a^{\prime}(u):=\big(D_a(u)\big)^{-1}.
\end{equation*}
The entries of these matrices give us some formal series with coefficients in $Y_{m|n}$:
\begin{eqnarray}
\label{goodd}&D_{a;i,j}(u) = \sum_{r \geq 0} D_{a;i,j}^{(r)} u^{-r}, \qquad &D_{a;i,j}^{\prime}(u) = \sum_{r \geq 0}D^{\prime(r)}_{a;i,j} u^{-r},\\
\label{badef}&E_{a,b;h,k}(u) = \sum_{r \geq 1} E_{a,b;h,k}^{(r)} u^{-r}, \qquad &F_{b,a;k,h}(u) = \sum_{r \geq 1} F_{b,a;k,h}^{(r)} u^{-r}.
\end{eqnarray}
Actually we only need the diagonal, upper-diagonal and lower-diagonal blocks.
Hence we set
\begin{equation}\label{paraef}
E_{b;i,j}(u) := E_{b,b+1;h,k}(u)=\sum_{r \geq 1}  E_{b;h,k}^{(r)} u^{-r},\qquad
F_{b;i,j}(u) := F_{b+1,b;k,h}(u)=\sum_{r \geq 1} F_{b;k,h}^{(r)} u^{-r},
\end{equation}
for $1\leq b\leq z-1$.
As proved in \cite{Pe4}, these coefficients
\begin{align*}
&\lbrace D_{a;i,j}^{(r)}, D_{a;i,j}^{\prime(r)} \,|\, 1\leq a\leq z, 1\leq i,j\leq \mu_a, r\geq 0\rbrace\\
&\lbrace E_{b;h,k}^{(r)} \,|\, 1\leq b< z, 1\leq h\leq \mu_b, 1\leq k\leq\mu_{b+1}, r\geq 1\rbrace\\
&\lbrace F_{b;k,h}^{(r)} \,|\, 1\leq b< z, 1\leq k\leq\mu_{b+1}, 1\leq h\leq \mu_b, r\geq 1\rbrace
\end{align*}
form a generating set for $Y_{m|n}$, called the {\em parabolic generators of $Y_{m|n}$}, which will be denoted by $\mathcal{P}_\mu$. 
Moreover, by \cite[Lemma~4.2]{Pe4}, their parities can be explicitly determined by the following rule:
\begin{eqnarray}
    \label{pad}\text{ parity of } D_{a;i,j}^{(r)}&=&\pa{i}_a+\pa{j}_a \,\,\text{(mod 2)},\\
    \label{pae}\text{ parity of } E_{b;h,k}^{(r)}&=&\pa{h}_{b}+\pa{k}_{b+1} \,\,\text{(mod 2)},\\
    \label{paf}\text{ parity of } F_{b;k,h}^{(r)}&=&\pa{k}_{b+1}+\pa{h}_{b} \,\,\text{(mod 2)}.
\end{eqnarray}

In the special case when $\mu=(1^{m+n}):=(\, \overbrace{1,\ldots,1}^{m+n} \,)$, the generating set, which will be denoted by $\mathcal{P}_{D}$, appeared in an analogue of the Drinfeld presentation for $Y_{m|n}$ \cite{BK1,Dr2,Go,Pe4, St, Ts}.
We list $\mathcal{P}_{D}$ explicitly here since it will be used right away:
\begin{align}
\label{drd}&\lbrace D_{a}^{(r)}, D_{a}^{\prime(r)} \,|\, 1\leq a\leq m+n, r\geq 0\rbrace,\\
\label{dre}&\lbrace E_{b}^{(r)} \,|\, 1\leq b< m+n, r\geq 1\rbrace,\\
\label{drf}&\lbrace F_{b}^{(r)} \,|\, 1\leq b< m+n, r\geq 1\rbrace,
\end{align}
and their parities are given by
\begin{equation}
\label{drpa}\pa{D_{a}^{(r)}}=\pa{D_{a}^{\prime(r)}}=0, \qquad \pa{E_{b}^{(r)}}=\pa{F_{b}^{(r)}}=\pa{b}+\pa{b+1}  \,\,\text{(mod 2)}.
\end{equation}

\section{Shifted super Yangian: Drinfeld's presentation}\label{shiftY1}
Recall from \textsection \ref{preW} that a pyramid $\pi$ can be uniquely recorded by a triple $(\sigma,\ell,\bo)$ where $\sigma$ is a shift matrix of size $m+n$, $\ell$ is a positive integer and $\bo$ is a $01$-sequence. Following \cite[\textsection 2]{BK2}, we use $\sigma$ and $\bo$ to define the following structure, which is one of the main objects studied in this article. 
\begin{definition}\label{drshift}
Let $m,n\in\mathbb{Z}_{\geq 0}$, $\sigma=(s_{i,j})$ be a shift matrix of size $m+n$ with a fixed $0^m1^n$-sequence $\bo$.
The shifted super Yangian of $\gl_{m|n}$ associated to $\sigma$, denoted by $Y_{m|n}(\sigma)$, is the superalgebra over $\mathbb{C}$ generated by following symbols
\begin{align*}
&\big\lbrace D_{a}^{(r)}, D_{a}^{\prime(r)} \,|\, {1\leq a\leq m+n,\;  r\geq 0}\big\rbrace,\\
&\big\lbrace E_{b}^{(r)} \,|\, {1\leq b< m+n,\; r> s_{b,b+1}}\big\rbrace,\\
&\big\lbrace F_{b}^{(r)} \,|\, {1\leq b< m+n,\; r> s_{b+1,b}}\big\rbrace,
\end{align*}
where their parities are defined by {\em (\ref{drpa})},
subject to the following relations:
\begin{eqnarray}
\label{d401} D_{a}^{(0)}=D_{a}^{\prime(0)}&=&1\,,\\
\label{d402} \sum_{t=0}^{r}D_{a}^{(t)}D_{a}^{\prime (r-t)}&=&\delta_{r0},\\
\label{d403} \big[D_{a}^{(r)},D_{b}^{(s)}\big]&=&0,
\end{eqnarray}
\begin{eqnarray}
\label{d404}  [D_{a}^{(r)}, E_{b}^{(s)}]
        &=&(-1)^{\pa{a}}\big( \delta_{a,b}-\delta_{a,b+1} \big) \sum_{t=0}^{r-1} D_{a}^{(t)} E_{b}^{(r+s-1-t)}, \\
\label{d405}  [D_{a}^{(r)}, F_{b}^{(s)}]
        &=&(-1)^{\pa{a}}\big( \delta_{a,b+1} -\delta_{a,b} \big) \sum_{t=0}^{r-1} F_{b}^{(r+s-1-t)}D_{a}^{(t)}, \\
\label{d406}  [E_{a}^{(r)} , F_{b}^{(s)}]
          &=&\delta_{a,b}(-1)^{\pa{a+1}+1}
          \sum_{t=0}^{r+s-1} D_{a}^{\prime (r+s-1-t)} D_{a+1}^{(t)},       
\end{eqnarray}          
\begin{multline}\label{d407}
 [E_{a}^{(r)} , E_{a}^{(s)}]
          =(-1)^{\pa{a+1}}
          \big( \sum_{t=s_{a,a+1}+1}^{s-1} E_{a}^{(r+s-1-t)} E_{a}^{(t)} 
          -\sum_{t=s_{a,a+1}+1}^{r-1} E_{a}^{(r+s-1-t)} E_{a}^{(t)}  \big),       
\end{multline} 
\begin{multline}\label{d408}
 [F_{a}^{(r)} , F_{a}^{(s)}]
          =(-1)^{\pa{a}}
          \big( \sum_{t=s_{a+1,a}+1}^{r-1} F_{a}^{(r+s-1-t)} F_{a}^{(t)} 
          -\sum_{t=s_{a+1,a}+1}^{s-1} F_{a}^{(r+s-1-t)} F_{a}^{(t)}  \big),         
 \end{multline}                          
{\allowdisplaybreaks
\begin{eqnarray}
\label{d409}[E_{a}^{(r+1)}, E_{a+1}^{(s)}]-[E_{a}^{(r)}, E_{a+1}^{(s+1)}]
&=&(-1)^{\pa{a+1}} E_{a}^{(r)}E_{a+1}^{(s)}\,,\\[3mm]
\label{d410}[F_{a}^{(r+1)}, F_{a+1}^{(s)}]-[F_{a}^{(r)}, F_{a+1}^{(s+1)}]&=&
(-1)^{1+\pa{a}\pa{a+1}+\pa{a+1}\pa{a+2}+\pa{a}\pa{a+2}} F_{a}^{(s)}F_{a}^{(r)}\,,
\end{eqnarray}
\begin{align}
\label{d411}&[E_{a}^{(r)}, E_{b}^{(s)}] = 0
\qquad\qquad\text{\;\;if\;\;} |b-a|>1,\\[3mm]
\label{d412}&[F_{a}^{(r)}, F_{b}^{(s)}] = 0
\qquad\qquad\text{\;\;if\;\;} |b-a|>1,\\[3mm]
\label{d413}&\big[E_{a}^{(r)},[E_{a}^{(s)},E_{b}^{(t)}]\big]+
\big[E_{a}^{(s)},[E_{a}^{(r)},E_{b}^{(t)}]\big]=0 \quad \text{if}\,\,\, |a-b|=1,\\[3mm]
\label{d414}&\big[F_{a}^{(r)},[F_{a}^{(s)},F_{b}^{(t)}]\big]+
\big[F_{a}^{(s)},[F_{a}^{(r)},F_{b}^{(t)}]\big]=0 \quad \text{if}\,\,\, |a-b|=1,\\[3mm]
\label{d415}&\big[\,[E_{a-1}^{(r)},E_{a}^{( s_{a,a+1}+1)}]\,,\,[E_{a}^{(s_{a,a+1}+1)},E_{a+1}^{(s)}]\,\big]=0 \;\;\text{when\;\;}  m+n\geq 4\, \text{and\;\;} \pa{a}+\pa{a+1}=1,\\[3mm]
\label{d416}&\big[\,[F_{a-1}^{(r)},F_{a}^{(s_{a+1,a}+1)}]\,,\,[F_{a}^{(s_{a+1,a}+1)},F_{a+1}^{(s)}]\,\big]=0 \;\;\text{when\;\;} m+n\geq 4\, \text{and\;\;} \pa{a}+\pa{a+1}=1,
\end{align}
}
for all admissible indices $a,b,r,s,t$. For example, {\em (\ref{d404})} is meant to hold for all $r\geq 0$, 
$s> s_{b,b+1}$, $1\leq a\leq m+n$ and $1\leq b<m+n$.
\end{definition}

Note that when $\sigma$ is the zero matrix, the presentation above coincides with the presentation of $Y_{m|n}$ given in \cite{Pe4} by taking $\mu=(1^{m+n})$ therein (this special case is also obtained in \cite{Ts}).
As a result, we may identify $Y_{m|n}(0)=Y_{m|n}.$

In the remaining part of this section, we will show that $Y_{m|n}(\sigma)$ can be identified as a subalgebra of $Y_{m|n}$ in general (Crorllary \ref{drsbpr}). Let $\mathcal{P}_{D,\sigma}$ be the generating set of $Y_{m|n}(\sigma)$ in Definition \ref{drshift}.
Let $\Gamma:Y_{m|n}(\sigma)\rightarrow Y_{m|n}$ be the map sending $\mathcal{P}_{D,\sigma}$ to the elements with the same name (\ref{drd})--(\ref{drf}) in $Y_{m|n}$ obtained by Gauss decomposition.

\begin{proposition}\label{drinj}
The canonical map $\Gamma:Y_{m|n}(\sigma)\rightarrow Y_{m|n}$ is a homomorphism.
\end{proposition}
\begin{proof}
By setting $\mu=(1^{m+n})$ in \cite[Proposition 7.1]{Pe4}, or simply by \cite[(2.2)--(2.10)]{Ts}, the relations (\ref{d401})--(\ref{d414}) are preserved by $\Gamma$. 
Setting $k=l$ in the {\em generalized quartic Serre relations} in \cite[(2.14), (2.15)]{Ts}, we see that (\ref{d415}) and (\ref{d416}) are preserved by $\Gamma$ as well.
\end{proof}

It remains to show that $\Gamma$ is injective.
We introduce the $loop$ $filtration$ on $Y_{m|n}(\sigma)$
\begin{equation}\notag
L_0 Y_{m|n}(\sigma) \subseteq L_1 Y_{m|n}(\sigma) \subseteq L_2 Y_{m|n}(\sigma) \subseteq \cdots
\end{equation}
by setting the degrees of the generators $D_{a}^{(r)}$, $E_{b}^{(r)}$, and $F_{b}^{(r)}$ to be $(r-1)$ and setting $L_k Y_{m|n}(\sigma)$ to be the span of all supermonomials in the generators of total degree not greater than $k$. 
Let $\gr^LY_{m|n}(\sigma)$ denote the associated graded superalgebra. 

For $1\leq a<b\leq m+n$, $r> s_{a,b}$ and $t>s_{b,a}$, define the following higher root elements $E_{a,b}^{(r)}, F_{b,a}^{(t)}\in Y_{m|n}(\sigma)$ recursively by
\begin{align}
\label{edrg}
E_{a,a+1}^{(r)}:= E_{a}^{(r)}, \qquad E_{a,b}^{(r)}:=(-1)^{\pa{b-1}}[E_{a,b-1}^{(r-s_{b-1,b})}, E_{b-1}^{(s_{b-1,b}+1)}],\\
\label{fdrg}
F_{a+1,a}^{(t)}:= F_{a}^{(t)}, \qquad F_{b,a}^{(t)}:=(-1)^{\pa{b-1}}[F_{b-1}^{(s_{b,b-1}+1)}, F_{b-1,a}^{(t-s_{b,b-1})}].
\end{align}
By definition, we have $E_{a,b}^{(r)}\in L_{r-1} Y_{m|n}(\sigma)$ and $F_{b,a}^{(t)}\in L_{t-1} Y_{m|n}(\sigma)$.

Define the elements $\{ e_{a,b}^{(r)} \,|\, 1\leq a,b\leq m+n,  r\geq s_{a,b} \} \subseteq \gr^LY_{m|n}(\sigma)$ by
\begin{equation}\label{drloop}
e_{a,b}^{(r)} := \left\{
\begin{array}{ll}
\gr^L_{r} D_{a}^{(r+1)} &\hbox{if $a = b$,}\\[3mm]
\gr^L_{r} E_{a,b}^{(r+1)} &\hbox{if $a < b$,}\\[3mm]
\gr^L_{r} F_{a,b}^{(r+1)} &\hbox{if $a > b$.}
\end{array}
\right.
\end{equation}
Using the same argument in \cite[Lemma~7.5]{Pe4}, except that one uses the defining relations of $Y_{m|n}(\sigma)$ listed in Definition~\ref{drshift}, we deduce the following result.
\begin{proposition}\label{drtroub}\cite[(2.21)]{BK2}\cite[(51)]{Go}
For all $1\leq a,b,c,d\leq m+n$, $r\geq s_{a,b}$, $t\geq s_{c,d}$,
the following identity holds in $\gr^LY_{m|n}(\sigma)$:
\begin{equation}\label{drs-inj}
[e_{a,b}^{(r)},e_{c,d}^{(t)}]=(-1)^{\pa{b}}\delta_{b,c} e_{a,d}^{(r+t)}
-(-1)^{\pa{a}\pa{b}+\pa{a}\pa{c}+\pa{b}\pa{c}}\delta_{a,d}e_{c,b}^{(r+t)}
\end{equation}
\end{proposition}

Let $\gl_{m|n}[x](\sigma)$ be the subalgebra of the loop superalgebra $\gl_{m|n}[x]$ generated by the following elements
$$\lbrace e_{ij}x^r \,|\, 1\leq i,j\leq m+n, r\geq s_{i,j}\rbrace.$$ 
By (\ref{sijk}), $\gl_{m|n}[x](\sigma)$ is indeed a subalgebra of $\gl_{m|n}[x]$.
Let the universal enveloping algebra $U\big(\gl_{m|n}[x](\sigma)\big)$ be equipped with the natural grading induced by the grading on $\gl_{m|n}[x]$.

\begin{theorem}\cite[Theorem 2.1]{BK2}\label{drgrisoLS}
The map 
$$\gamma:U\big(\gl_{m|n}[x](\sigma)\big)\longrightarrow \gr^LY_{m|n}(\sigma)$$ defined by
$$\gamma(e_{a,b}x^r)=(-1)^{\pa{a}}e_{a,b}^{(r)},$$ for all $1\leq a,b\leq m+n$, $r\geq s_{a,b}$,
is an isomorphism of graded superalgebras.
\end{theorem}
\begin{proof}
$\gamma$ is a homomorphism by (\ref{drs-inj}).
Since the image of $\gamma$ contains the image of $\mathcal{P}_{D,\sigma}$ in $\gr^LY_{m|n}(\sigma)$, $\gamma$ is surjective.

It remains to show the injectivity. Consider firstly the special case when $\sigma=0$, where we can identify $Y_{m|n}(0)=Y_{m|n}$. By \cite[Proposition 7.9]{Pe4}, the ordered supermonomials in the elements $\{e_{a,b}^{(r)} \, | \, 1\leq a,b\leq m+n, \, r\geq 0\}$ are linearly independent in $\gr^L Y_{m|n}$. It follows that $\gamma$ is injective.

For the general case,
observe that the canonical map $\Gamma: Y_{m|n}(\sigma)\rightarrow Y_{m|n}$ is a homomorphism of filtered superalgebras.
It induces a map $\gr^L Y_{m|n}(\sigma)\rightarrow \gr^L Y_{m|n}$, sending $e_{a,b}^{(r)}\in\gr^L Y_{m|n}(\sigma)$ to $e_{a,b}^{(r)}\in \gr^L Y_{m|n}$. 
By the previous paragraph, the ordered supermonomials in the elements $\{e_{a,b}^{(r)} \, | \, 1\leq a,b\leq m+n, \, r\geq s_{a,b}\}$ are linearly independent in $\gr^L Y_{m|n}(\sigma)$ as well, which implies that $\gamma$ is injective by the PBW theorem for $U\big(\gl_{m|n}[x](\sigma)\big)$.
\end{proof}

\begin{corollary}\label{drsbpr}
The canonical map $\Gamma:Y_{m|n}(\sigma)\rightarrow Y_{m|n}$ is injective.
As a consequence, the structure $Y_{m|n}(\sigma)$ defined in Definition~{\em \ref{drshift}} can be identified as a subalgebra of $Y_{m|n}$. 
\end{corollary}

\section{Shifted super Yangian: Parabolic presentations}\label{shiftY2}
In this section, we provide a more sophisticated definition for $Y_{m|n}(\sigma)$ together with corresponding results mentioned in \textsection 4. For the sake of the purpose, we introduce some terminologies and notations.

Let $\sigma=(s_{i,j})$ be a shift matrix of size $m+n$. We say a composition $\mu=(\mu_1,\ldots,\mu_z)$ of $m+n$ of length $z$ is {\em admissible to $\sigma$} if 
$$s_{\mu_1+\mu_2+\cdots+\mu_{a-1}+i,\mu_1+\mu_2+\cdots+\mu_{a-1}+j}=0$$
for all $1\leq a\leq z$, $1\leq i,j\leq \mu_{a}$. In addition, $\mu$ is called {\em minimal admissible} if it is admissible to $\sigma$ and its length is minimal among all compositions admissible to $\sigma$.
Clearly, for a shift matrix $\sigma$, its minimal admissible shape uniquely exists. 
Moreover, $(1^{m+n})$ is admissible for any $\sigma$ of size $m+n$.

\begin{remark}
The notion of admissibility can be intuitively explained in terms of pyramid. Note that one can decompose a pyramid horizontally into a number of rectangles. An admissible shape $\mu$ records the heights of these rectangles from top to bottom, while the minimal admissible shape records such a decomposition with the least number of rectangles.
\end{remark}

When $\mu=(\mu_1,\mu_2,\ldots,\mu_z)$ is admissible to $\sigma$, we will use a shorthand notation
\begin{equation}\label{sabmu}
s_{a,b}^{\mu}:=s_{\mu_1+\ldots+\mu_a,\mu_1+\ldots+\mu_b}, \quad\forall  \,\,1\leq a,b\leq z.
\end{equation}
Note that one can recover the original matrix $\sigma$ if an admissible shape $\mu$ and the numbers $\{s_{a,b}^{\mu}|1\leq a,b\leq z\}$ are known. Moreover, the admissible condition (\ref{sijk}) implies that for any $1\leq a,b\leq z$, we have
\begin{equation}\label{admcon}
s_{\mu_1+\cdots+\mu_{a-1}+i,\mu_1+\cdots+\mu_{b-1}+j}=s_{a,b}^\mu, \qquad \forall 1\leq i\leq \mu_a, 1\leq j\leq\mu_b.
\end{equation}
Let $\bo$ be a fixed $0^m1^n$-sequence. We decompose $\bo$ into $z$ subsequences 
according to $\mu$
\[
\bo=\bo_1\bo_2\cdots\bo_z,
\]
and define the restricted parity $\pa{i}_a$ as in (\ref{respa}). 
Now we give the following presentation for $Y_{m|n}(\sigma)$, a super analogue of shifted Yangian given in \cite[\textsection 3]{BK2}.

\begin{definition}\label{parashift}
Let $\sigma=(s_{i,j})$ be a shift matrix of size $m+n$ with a fixed $0^m1^n$-sequence $\bo$. 
Let $\mu=(\mu_1,\ldots,\mu_z)$ be an admissible shape to $\sigma$.
The shifted super Yangian of $\gl_{m|n}$ associated to $\sigma$ and $\mu$, denoted by $Y_{\mu}(\sigma)$, is the superalgebra over $\mathbb{C}$ generated by the following symbols
\begin{align*}
&\big\lbrace D_{a;i,j}^{(r)}, D_{a;i,j}^{\prime(r)} \,|\, {1\leq a\leq z,\; 1\leq i,j\leq \mu_a,\; r\geq 0}\big\rbrace,\\
&\big\lbrace E_{b;h,k}^{(r)} \,|\, {1\leq b< z,\; 1\leq h\leq \mu_b, 1\leq k\leq\mu_{b+1},\; r> s_{b,b+1}^\mu}\big\rbrace,\\
&\big\lbrace F_{b;k,h}^{(r)} \,|\, {1\leq b< z,\; 1\leq h\leq \mu_b, 1\leq k\leq\mu_{b+1}, \; r> s_{b+1,b}^\mu}\big\rbrace,
\end{align*}
where their parities are defined by {\em(\ref{pad})--(\ref{paf})},
subject to the following relations:
\begin{eqnarray}
\label{p701}D_{a;i,j}^{(0)}=D_{a;i,j}^{\prime(0)}&=&\delta_{ij}\,,\\
\label{p702}\sum_{p=1}^{\mu_a}\sum_{t=0}^{r}D_{a;i,p}^{(t)}D_{a;p,j}^{\prime (r-t)}&=&\delta_{r0}\delta_{ij}\,,\\
\label{p703}\big[D_{a;i,j}^{(r)},D_{b;h,k}^{(s)}\big]&=&
    \delta_{ab}(-1)^{\pa{i}_a\pa{j}_a+\pa{i}_a\pa{h}_a+\pa{j}_a\pa{h}_a}\times \notag\\
    &&\sum_{t=0}^{min(r,s)-1}\big(D_{a;h,j}^{(t)}D_{a;i,k}^{(r+s-1-t)}-D_{a;h,j}^{(r+s-1-t)}D_{a;i,k}^{(t)}\big),
\end{eqnarray}
{\allowdisplaybreaks
\begin{multline}\label{p704}
 [D_{a;i,j}^{(r)}, E_{b;h,k}^{(s)}]
        =\delta_{a,b}\delta_{hj}(-1)^{\pa{h}_a\pa{j}_a}\sum_{p=1}^{\mu_a}\sum_{t=0}^{r-1} D_{a;i,p}^{(t)} E_{b;p,k}^{(r+s-1-t)}\\
        -\delta_{a,b+1}(-1)^{\pa{h}_b\pa{k}_a+\pa{h}_b\pa{j}_a+\pa{j}_a\pa{k}_a} \sum_{t=0}^{r-1} D_{a;i,k}^{(t)} E_{b;h,j}^{(r+s-1-t)},
        \end{multline}
\begin{multline}\label{p705}
 [D_{a;i,j}^{(r)}, F_{b;h,k}^{(s)}]
        =-\delta_{a,b}(-1)^{\pa{i}_a\pa{j}_a+\pa{h}_{a+1}\pa{i}_a+\pa{h}_{a+1}\pa{j}_a}\sum_{p=1}^{\mu_a}\sum_{t=0}^{r-1} F_{b;h,p}^{(r+s-1-t)}D_{a;p,j}^{(t)}\\
        +\delta_{a,b+1}(-1)^{\pa{h}_a\pa{k}_b+\pa{h}_a\pa{j}_a+\pa{j}_a\pa{k}_b} \sum_{t=0}^{r-1} F_{b;i,k}^{(r+s-1-t)}D_{a;h,j}^{(t)},
        \end{multline}        
\begin{multline}\label{p706}
 [E_{a;i,j}^{(r)} , F_{b;h,k}^{(s)}]
          =\delta_{a,b}(-1)^{\pa{h}_{a+1}\pa{k}_a+\pa{j}_{a+1}\pa{k}_a+\pa{h}_{a+1}\pa{j}_{a+1}+1}
          \sum_{t=0}^{r+s-1} D_{a;i,k}^{\prime (r+s-1-t)} D_{a+1;h,j}^{(t)},       
          \end{multline}            
\begin{multline}\label{p707}
 [E_{a;i,j}^{(r)} , E_{a;h,k}^{(s)}]
          =(-1)^{\pa{h}_{a}\pa{j}_{a+1}+\pa{j}_{a+1}\pa{k}_{a+1}+\pa{h}_{a}\pa{k}_{a+1}}\times\\
          \big( \sum_{t=s_{a,a+1}^\mu+1}^{s-1} E_{a;i,k}^{(r+s-1-t)} E_{a;h,j}^{(t)} 
          -\sum_{t=s_{a,a+1}^\mu+1}^{r-1} E_{a;i,k}^{(r+s-1-t)} E_{a;h,j}^{(t)}  \big),       
\end{multline} 
\begin{multline}\label{p708}
 [F_{a;i,j}^{(r)} , F_{a;h,k}^{(s)}]
          =(-1)^{\pa{h}_{a+1}\pa{j}_{a}+\pa{j}_{a}\pa{k}_{a}+\pa{h}_{a+1}\pa{k}_{a}}\times\\
          \big( \sum_{t=s_{a+1,a}^\mu+1}^{r-1} F_{a;i,k}^{(r+s-1-t)} F_{a;h,j}^{(t)} 
          -\sum_{t=s_{a+1,a}^\mu+1}^{s-1} F_{a;i,k}^{(r+s-1-t)} F_{a;h,j}^{(t)}  \big),         
 \end{multline}

\begin{equation}
\label{p709}[E_{a;i,j}^{(r+1)}, E_{a+1;h,k}^{(s)}]-[E_{a;i,j}^{(r)}, E_{a+1;h,k}^{(s+1)}]
=(-1)^{\pa{j}_{a+1}\pa{h}_{a+1}}\delta_{h,j}\sum_{q=1}^{\mu_{a+1}}E_{a;i,q}^{(r)}E_{a+1;q,k}^{(s)}\,,
\end{equation}

\begin{multline}
\label{p710}[F_{a;i,j}^{(r+1)}, F_{a+1;h,k}^{(s)}]-[F_{a;i,j}^{(r)}, F_{a+1;h,k}^{(s+1)}]=\\
(-1)^{\pa{i}_{a+1}(\pa{j}_{a}+\pa{h}_{a+2})+\pa{j}_a\pa{h}_{a+2}+1}\delta_{i,k}\sum_{q=1}^{\mu_{a+1}}F_{a+1;h,q}^{(s)}F_{a;q,j}^{(r)}\,,
\end{multline}

\begin{align}
\label{p711}&[E_{a;i,j}^{(r)}, E_{b;h,k}^{(s)}] = 0
\qquad\qquad\text{\;\;if\;\; $|b-a|>1$ \;\;or\;\; \;if\;\;$b=a+1$ and $h \neq j$},\\[3mm]
\label{p712}&[F_{a;i,j}^{(r)}, F_{b;h,k}^{(s)}] = 0
\qquad\qquad\text{\;\;if\;\; $|b-a|>1$ \;\;or\;\; \;if\;\;$b=a+1$ and $i \neq k$},\\[3mm]
\label{p713}&\big[E_{a;i,j}^{(r)},[E_{a;h,k}^{(s)},E_{b;f,g}^{(t)}]\big]+
\big[E_{a;i,j}^{(s)},[E_{a;h,k}^{(r)},E_{b;f,g}^{(t)}]\big]=0 \quad \text{if}\,\,\, |a-b|\geq 1,\\[3mm]
\label{p714}&\big[F_{a;i,j}^{(r)},[F_{a;h,k}^{(s)},F_{b;f,g}^{(t)}]\big]+
\big[F_{a;i,j}^{(s)},[F_{a;h,k}^{(r)},F_{b;f,g}^{(t)}]\big]=0 \quad \text{if}\,\,\, |a-b|\geq 1,\\[3mm]
\label{p715}&\big[\,[E_{a-1;i,f_1}^{(r)},E_{a;f_2,j}^{(s_{a,a+1}^{\mu}+1)}]\,,\,[E_{a;h,g_1}^{(s_{a,a+1}^{\mu}+1)},E_{a+1;g_2,k}^{(s)}]\,\big]=0 \;\;\text{if\;\;}  z\geq 4\, \text{and\;\;} \pa{h}_{a}+\pa{j}_{a+1}=1,\\[3mm]
\label{p716}&\big[\,[F_{a-1;i,f_1}^{(r)},F_{a;f_2,j}^{(s_{a+1,a}^{\mu}+1)}]\,,\,[F_{a;h,g_1}^{(s_{a+1,a}^{\mu}+1)},F_{a+1;g_2,k}^{(s)}]\,\big]=0 \;\;\text{if\;\;} z\geq 4\, \text{and\;\;} \pa{j}_{a}+\pa{h}_{a+1}=1, 
\end{align}}
for all indices $a, b, f, f_1, f_2, g, g_1, g_2, h, i, j, k, r, s, t$ that make sense.
For example, {\em (\ref{p709})} is supposed to hold for all $1\leq a\leq z-1$, $1\leq i\leq \mu_a$, $1\leq h,j\leq \mu_{a+1}$, $1\leq k\leq \mu_{a+2}$, $r\geq s_{a,a+1}^{\mu}+1$, $s\geq s_{a+1,a+2}^{\mu}+1$.
\end{definition}

In the special case where $\sigma$ is the zero matrix, the above relations are
precisely the defining relations of $Y_{m|n}$ with respect to the parabolic generators $\mathcal{P}_\mu$ given in \textsection 3.
We shall write $Y_\mu$ instead of $Y_{m|n}$ to emphasize that we are using the parabolic presentation in \cite{Pe4} to define $Y_{m|n}$.
The generators of $Y_{\mu}(\sigma)$, denoted by $\mathcal{P}_{\mu,\sigma}$, will be called the {\em parabolic generators} of $Y_\mu(\sigma)$. Later we will identify $\mathcal{P}_{\mu,\sigma}$ as as subset of $\mathcal{P}_\mu$.

\begin{remark}
As noticed in \cite{Pe2, Ts}, the definition of $Y_\mu$ is independent from the choice of the 01-sequence $\bo$ since the RTT presentation of $Y_{m|n}$ is. For $Y_{\mu}(\sigma)$, we have a similar but slightly weaker phenomenon. 
Write $Y_\mu(\sigma, \bo)$ for the shifted super Yangian to emphasize the choice of $\bo$.
Let $S_{m+n}$ be the symmetric group on $m+n$ objects, which acts on $\bo$ by permutation, and let $S_\mu$ denote its Young subgroup associated to $\mu$. Then we have 
$$Y_{\mu}(\sigma, \bo)\cong Y_{\mu}(\sigma, \rho\cdot\bo)  \,\, \,\,  \forall \rho\in S_\mu.$$
\end{remark}

Fix an admissible shape $\mu$. 
Similar to \textsection 3, we will show that $Y_{\mu}(\sigma)$ can be identified as a subalgebra of $Y_{\mu}$.
Let $\Gamma:Y_{\mu}(\sigma)\rightarrow Y_{\mu}$ be the map sending elements in $\mathcal{P}_{\mu,\sigma}$ to elements  (\ref{goodd}) and (\ref{paraef}) with the same name in $Y_{\mu}$ obtained by Gauss decomposition with respect to $\mu$.
\begin{proposition}\label{parahom}
The canonical map $\Gamma:Y_{\mu}(\sigma)\rightarrow Y_{\mu}$ is a homomorphism.
\end{proposition}
\begin{proof}
By \cite{Pe4}, the relations (\ref{p701})--(\ref{p714}) hold in $Y_{\mu}$ whenever the indices make sense. It remains to show that (\ref{p715}) and (\ref{p716}) also hold in $Y_{\mu}$. These relations are crucial differences from the non-super case in \cite{BK2} and checking them turns out to be very technical and involved. 
As a result, we postpone this part to the end of this section; see Proposition~\ref{missing}.
\end{proof}

For $1\leq a<b\leq z$, $1\leq i\leq \mu_a$, $1\leq j\leq \mu_b$, $r>s_{a,b}^{\mu}$ and a fixed $1\leq k\leq \mu_{b-1}$, we define the higher root elements $E_{a,b;i,j}^{(r)}\in Y_{\mu}(\sigma)$ recursively by
\begin{equation}\label{eparag}
E_{a,a+1;i,j}^{(r)}:= E_{a;i,j}^{(r)}, \qquad E_{a,b;i,j}^{(r)}:=(-1)^{\pa{k}_{b-1}}[E_{a,b-1;i,k}^{(r-s_{b-1,b}^{\mu})}, E_{b-1;k,j}^{(s_{b-1,b}^{\mu}+1)}].
\end{equation}
Similarly, using the same indices except that for $r>s_{b,a}^{\mu}$, we define $F_{b,a;j,i}^{(r)}\in Y_{\mu}(\sigma)$ by
\begin{equation}\label{fparag}
F_{a+1,a;j,i}^{(r)}:= F_{a;j,i}^{(r)}, \qquad F_{b,a;j,i}^{(r)}:=(-1)^{\pa{k}_{b-1}}[F_{b-1;j,k}^{(s_{b,b-1}^{\mu}+1)}, F_{b-1,a;k,i}^{(r-s_{b,b-1}^{\mu})}].
\end{equation}
It turns out that the above definitions are independent of the choice of $k$; see Remark~\ref{4.8}.

We introduce the $loop$ $filtration$ on $Y_{\mu}(\sigma)$
\begin{equation}\notag
L_0 Y_{\mu}(\sigma) \subseteq L_1 Y_{\mu}(\sigma) \subseteq L_2 Y_{\mu}(\sigma) \subseteq \cdots
\end{equation}
by setting the degrees of the generators $D_{a;i,j}^{(r)}$, $E_{a;i,j}^{(r)}$, and $F_{a;i,j}^{(r)}$ to be $r-1$ and setting $L_k Y_{\mu}(\sigma)$ to be the span of all supermonomials in the generators of total degree not greater than $k$. We let $\gr^LY_{\mu}(\sigma)$ denote the associated graded superalgebra and 
define the elements $\{ e_{a,b;i,j}^{(r)} \,|\, 1\leq a,b\leq z, 1\leq i\leq \mu_a, 1\leq j\leq \mu_b, r\geq s_{a,b}^\mu \}\subseteq \gr^LY_{\mu}(\sigma)$ by
\begin{equation}\notag
e_{a,b;i,j}^{(r)} := \left\{
\begin{array}{ll}
\gr^L_{r} D_{a;i,j}^{(r+1)} &\hbox{if $a = b$,}\\[3mm]
\gr^L_{r} E_{a,b;i,j}^{(r+1)} &\hbox{if $a < b$,}\\[3mm]
\gr^L_{r} F_{a,b;i,j}^{(r+1)} &\hbox{if $a > b$.}
\end{array}
\right.
\end{equation}
The following is a parabolic version of Proposition \ref{drtroub}, which can be derived by the same argument in \cite[Lemma 7.5]{Pe4} with the defining relations in Definition~\ref{parashift}.
\begin{proposition}\cite[Lemma 6.7]{BK1}\cite[Lemma 7.5]{Pe4}
For all $1\leq a,b,c,d\leq z$, $1\leq i\leq \mu_a$, $1\leq j\leq \mu_b$, $r\geq s_{a,b}^\mu$, $t\geq s_{c,d}^\mu$,
the following identity holds in $\gr^LY_{\mu}(\sigma)$:
\begin{equation}\label{s-inj}
[e_{a,b;i,j}^{(r)},e_{c,d;h,k}^{(t)}]=(-1)^{\pa{j}_b\pa{h}_c}\delta_{b,c}\delta_{h,j} e_{a,d;i,k}^{(r+t)}
-(-1)^{\pa{i}_a\pa{j}_b+\pa{i}_a\pa{h}_c+\pa{j}_b\pa{h}_c}\delta_{a,d}\delta_{i,k}e_{c,b;h,j}^{(r+t)}.
\end{equation}
\end{proposition}

\begin{theorem}\label{grisoLS}
The map 
$$\gamma:U\big(\gl_{m|n}[x](\sigma)\big)\longrightarrow \gr^LY_{\mu}(\sigma)$$ defined by
$$\gamma(e_{\mu_1+\dots+\mu_{a-1}+i,\mu_1+\dots+\mu_{b-1}+j}x^r)=(-1)^{\pa{i}_a}e_{a,b;i,j}^{(r)},$$ for all $1\leq a,b\leq z$, $1\leq i\leq \mu_a$, $1\leq j\leq \mu_b$, $r\geq s_{a,b}^\mu$,
is an isomorphism of graded superalgebras.
\end{theorem}
\begin{proof}
$\gamma$ is a surjective homomorphism by (\ref{s-inj}). For injectivity, we start with the case $\sigma=0$, where we already know that $Y_{\mu}(0)=Y_\mu$, and the statement follows from Corollary~\ref{Yloop}.
For the general case, observe that the canonical map $\Gamma:Y_{\mu}(\sigma)\rightarrow Y_{\mu}$ is a homomorphism of filtered superalgebras (under loop filtration), 
and its induced map $\gr^L Y_{\mu}(\sigma)\rightarrow \gr^L Y_{\mu}$ sends $e_{a,b;i,j}^{(r)}\in\gr^L Y_{\mu}(\sigma)$ to $e_{a,b;i,j}^{(r)}\in \gr^L Y_{\mu}$.
By the previous paragraph, the ordered supermonomials in the elements $\{e_{a,b;i,j}^{(r)} \, | \, 1\leq a,b\leq m+n, \, r\geq s_{a,b}^{\mu}\}$ are linearly independent in $\gr^L Y_{\mu}(\sigma)$, hence $\gamma$ is injective by the PBW theorem for $U\big(\gl_{m|n}[x](\sigma)\big)$.
\end{proof}

\begin{theorem}\label{sbpr}
Let $Y_{\mu}(\sigma)$ be the subalgebra of $Y_{\mu}$ generated by the union of the following subsets of $\mathcal{P}_{\mu}$:
\begin{align*}
&\big\lbrace D_{a;i,j}^{(r)}, D_{a;i,j}^{\prime(r)} \,|\, {1\leq a\leq z,\; 1\leq i,j\leq \mu_a,\; r\geq 0}\big\rbrace,\\
&\big\lbrace E_{b;h,k}^{(r)} \,|\, {1\leq b< z,\; 1\leq h\leq \mu_b, 1\leq k\leq\mu_{b+1},\; r> s_{b,b+1}^\mu}\big\rbrace,\\
&\big\lbrace F_{b;k,h}^{(r)} \,|\, {1\leq b< z,\; 1\leq k\leq\mu_{b+1}, 1\leq h\leq \mu_b,\; r> s_{b+1,b}^\mu}\big\rbrace.
\end{align*}
Then the relations {\em (\ref{p701})--(\ref{p716})} form a set of defining relations for $Y_\mu(\sigma)$.
In other words, $Y_\mu(\sigma)$ defined in Definition~{\em\ref{parashift}} can be realized as a subalgebra of the super Yangian $Y_{\mu}$. 
\end{theorem}
\begin{proof}
We slightly change the notation in this proof to avoid possible confusion.
Let $\tilde{Y}_\mu(\sigma)$ denote the abstract superalgebra generated by elements in $\mathcal{P}_{\mu,\sigma}$ with defining relations given in Definition~\ref{parashift} and let $Y_\mu(\sigma)$ denote the concrete subalgebra of $Y_{\mu}$ as stated in the theorem.

Let $\Gamma:\tilde{Y}_\mu(\sigma)\rightarrow Y_{\mu}(\sigma)$ be the map sending elements of $\mathcal{P}_{\mu,\sigma}$ to the corresponding elements of $Y_{\mu}$ denoted by the same notations. 
By Proposition \ref{parahom}, $\Gamma$ is a surjective homomorphism.
Its injectivity follows from Theorem~\ref{grisoLS}.
\end{proof}

\begin{remark}\label{4.8}
By Theorem {\em\ref{sbpr}}, $E_{b;h,k}^{(r)}$ and $F_{b;k,h}^{(r)}$ are now concrete elements in $Y_\mu$.
Using the same argument as in \cite[(6.9)]{BK1} together with the admissible condition {\em (\ref{admcon})}, one can show that the higher root elements defined recursively by {\em (\ref{eparag})} and {\em (\ref{fparag})} are independent of the choices of $k$.

\end{remark}

Let $Y_{\mu}^0$ denote the subalgebra of $Y_{\mu}(\sigma)$ generated by all of the $D_{a;i,j}^{(r)}$'s , $Y_{\mu}^+(\sigma)$ denote the subalgebra generated by all of the $E_{b;h,k}^{(r)}$'s and $Y_{\mu}^-(\sigma)$ denote the subalgebra generated by all of the $F_{b;k,h}^{(r)}$'s. The following corollary give PBW bases for these subalgebras.

\begin{corollary}\label{pbw2}\cite[Theorem 3.2]{BK2}
\begin{enumerate}
\item The set of supermonomials in the elements
$$\{D_{a;i,j}^{(r)} \,|\, 1\leq a\leq z, 1 \leq i,j \leq \mu_a, r>0\}$$
taken in some
fixed order forms a basis for $Y_\mu^0$.\\
\item The set of supermonomials in the elements
$$\{E_{a,b;h,k}^{(r)} \, |\, 1 \leq a < b \leq z, 1 \leq h \leq \mu_a, 1 \leq k \leq \mu_b, r >s_{a,b}^{\mu} \}$$ 
taken in some fixed
order forms a basis for $Y_\mu^+(\sigma)$.\\
\item The set of supermonomials in the elements
$$\{F_{b,a;k,h}^{(r)} \,|\, 1 \leq a < b \leq z,
1 \leq k \leq \mu_b, 1 \leq h \leq \mu_a, r > s_{b,a}^{\mu}\}$$ 
taken in some fixed
order forms a basis for $Y_\mu^-(\sigma)$.\\
\item The set of supermonomials in the union of
the elements listed in {\em (1)--(3)}
taken in some fixed order forms a basis for $Y_{\mu}(\sigma)$.
\end{enumerate}
\end{corollary}
\begin{proof}
(4) follows from Theorem \ref{grisoLS} and the PBW theorem for $U\big(\gl_{m|n}[x](\sigma)\big)$, while the others follow from (\ref{s-inj}).
\end{proof}

\begin{corollary}\label{iso2}\cite[Corollary 3.4]{BK2}
The multiplicative map  $Y_\mu^-(\sigma) \otimes Y_\mu^0 \otimes Y_\mu^+(\sigma)
\longrightarrow Y_{\mu}(\sigma)$ is an isomorphism of superspaces.
\end{corollary}

Now we show that the definition of $Y_{\mu}(\sigma)$ is independent of the choice of the admissible shape $\mu$. 
It suffices to show that $Y_\mu(\sigma)= Y_{(1^{m+n})}(\sigma)$.
Assume that $\mu=(\mu_1,\ldots, \mu_z)$ is admissible to $\sigma$. If $\mu_j=1$ for all $j$, then we have done.
Otherwise, suppose that $\mu_p>1$ for some $1\leq p\leq z$ and we decompose $\mu_p=x+y$ for some positive integers $x,y$.

Define a finer composition $\nu$ of length $z+1$ by setting $\nu_i=\mu_i$ for all $1\leq i\leq p-1$,
$\nu_p=x$, $\nu_{p+1}=y$, $\nu_{j+1}=\mu_j$ for all $p+1\leq j\leq z$, that is, 
$$
\nu=(\mu_1,\ldots, \mu_{p-1}, x, y, \mu_{p+2}, \ldots, \mu_z),
$$
which is also admissible to $\sigma$. We claim that $$Y_\mu(\sigma)= Y_{\nu}(\sigma).$$

Consider the Gauss decomposition of the matrix $T(u)$ with respect to the two compositions $\mu$ and $\nu$, respectively:
\[
T(u)= {^\mu E(u)} {^\mu D(u)} {^\mu F(u)}= {^\nu E(u)} {^\nu D(u)} {^\nu F(u)},
\]
where the matrices are block matrices as described in \textsection 3.

Denote by $^\mu D_a$ and $^\nu D_a$ the $a$-th diagonal matrices in $^\mu D(u)$ and $^\nu D(u)$ with respect to the compositions $\mu$ and $\nu$, respectively. Similarly, let $^\mu E_a$ and $^\mu F_a$ denote the matrices in the $a$-th upper and the $a$-th lower diagonal of $^\mu E(u)$ and $^\mu F(u)$, respectively; $^\nu E_a$ and $^\nu F_a$ are defined to be the matrices in the $a$-th upper and the $a$-th lower diagonal of $^\nu E(u)$ and $^\nu F(u)$, respectively.
\begin{lemma}\label{split}
Using the above notation, define an $(x \times x)$-matrix
$A$, an $(x\times y)$-matrix $B$,
a $(y \times x)$-matrix $C$ and a $(y \times y)$-matrix $D$
from the equation
$$
{^\mu}D_p = \left(\begin{array}{ll}I_x&0\\
C&I_y\end{array}
\right)\left(\begin{array}{ll}A&0\\
0&D\end{array}
\right)\left(\begin{array}{ll}I_x&B\\
0&I_y\end{array}
\right).
$$
Then
\begin{itemize}
\item[(i)] ${^\nu}D_a = {^\mu}D_a$ for $a < p$, ${^\nu}D_p = A$,
${^\nu}D_{p+1} = D$, and ${^\nu}D_c = {^\mu}D_{c-1}$ for $c > p+1$;
\item[(ii)] ${^\nu}E_a = {^\mu}E_a$ for $a < p-1$,
${^\nu}E_{p-1}$ is the submatrix consisting
of the first $x$ columns
of
${^\mu}E_{p-1}$, ${^\nu}E_{p} = B$, ${^\nu}E_{p+1}$ is the submatrix consisting of the last $p$ rows
of ${^\mu}E_p$, and ${^\nu}E_c = {^\mu}E_{c-1}$ for $c > p+1$;
\item[(iii)] ${^\nu}F_a = {^\mu}F_a$ for $a < p-1$,
${^\nu}F_{p-1}$ is the submatrix consisting of the first
$x$ rows  of
${^\mu}F_{p-1}$, ${^\nu}F_{p} = C$, ${^\mu}F_{p+1}$ is the submatrix
consisting of the last $y$ columns of
${^\mu}F_p$, and ${^\nu}F_c = {^\mu}F_{c-1}$ for $c > p+1$.
\end{itemize}
\end{lemma}
\begin{proof}
Matrix multiplication.
\end{proof}

As a consequence of Lemma \ref{split}, one has that $Y_{\nu}(\sigma)\subseteq Y_{\mu}(\sigma)$. Now the equality follows from the fact that the isomorphism $U\big(\gl_{m|n}[x](\sigma)\big)\cong \gr^L Y_{\mu}(\sigma)$ is independent of the choice of $\mu$. 
Applying induction on the length of the admissible shape $\mu$, we have deduced the desired result. 
\begin{corollary}\label{indmu}
$Y_{\mu}(\sigma)$ is independent of the choice of the admissible shape $\mu$.
\end{corollary}

Let $\sigma$ be a shift matrix with an admissible shape $\mu$.
Note that the transpose matrix $\sigma^t$ is again a shift matrix while $\mu$ is still admissible for $\sigma^t$.
On the other hand, suppose that
$\vec\sigma = (\vec{s}_{i,j})_{1 \leq i,j \leq m+n}$ is another shift matrix satisfying (\ref{sijk}) and the condition 
$$\vec{s}_{i,i+1}+\vec s_{i+1,i}= s_{i,i+1}+s_{i+1,i}$$ holds for all $1\leq i\leq m+n-1$. 
As a result, if $\mu$ is an admissible shape for $\sigma$ then it is also admissible for $\vec\sigma$. 
Denote by $\vec{D}_{a;i,j}^{(r)}, \vec{E}_{b;h,k}^{(r)}$ and $\vec{F}_{b;k,h}^{(r)}$ the parabolic generators of $Y_{\mu}(\vec\sigma)$ to avoid confusion. 
The following results can be easily deduced from the presentation of $Y_{\mu}(\sigma)$.

\begin{proposition}
The map $\tau: Y_\mu(\sigma)\rightarrow Y_{\mu}(\sigma^t)$ defined by
\begin{equation}\label{taudef}
\tau(D_{a;i,j}^{(r)}) =
D_{a;j,i}^{(r)},\,\,\,
\tau(E_{b;h,k}^{(r)}) =
F_{b;k,h}^{(r)},\,\,\,
\tau(F_{b;k,h}^{(r)}) =
E_{b;h,k}^{(r)}.
\end{equation}
is a superalgebra anti-isomorphism of order 2.
\end{proposition}

\begin{proposition}
The map $\iota:Y_{\mu}(\sigma) \rightarrow Y_{\mu}({\vec{\sigma}})$ defined by
\begin{equation}\label{iotadef}
\iota(D_{a;i,j}^{(r)}) = \vec{D}_{a;i,j}^{(r)},\quad
\iota(E_{b;h,k}^{(r)}) = \vec{E}_{b;h,k}^{(r-s_{b,b+1}^\mu+\vec s_{b,b+1}^\mu)},\quad
\iota(F_{b;k,h}^{(r)}) = \vec{F}_{b;k,h}^{(r-s_{b+1,b}^\mu+\vec s_{b+1,b}^\mu)},
\end{equation}
is a superalgebra isomorphism.
\end{proposition}

Now we prove the missing piece in the proof of Proposition~\ref{parahom}.
\begin{proposition}\label{missing}
The relations {\em (\ref{p715})} and {\em (\ref{p716})} hold in $Y_{\mu}$, where $E_{b;h,k}^{(r)}$ and $F_{b;k,h}^{(r)}$ are the elements in $Y_{\mu}$ defined by {\em (\ref{paraef})}.
\end{proposition}
\begin{proof}
We prove (\ref{p715}) where (\ref{p716}) is similar.
Inspired by \cite[\textsection 2.4]{BK3}, the proof is given by downward induction on the length of the admissible shape $\mu$. Our initial step is the case $\mu=(1^{m+n})$, where (\ref{p715}) reduces to (\ref{d415}), which was proved in Proposition~\ref{drinj}.

Assume now the length of $\mu=(\mu_1,\ldots, \mu_z)$ is strictly less than $m+n$. Following the same notations given in the proof of Corollary \ref{indmu}, we may choose some $1\leq p\leq z$ and decompose $\mu_p=x+y$ to obtain a new composition $
\nu=(\mu_1,\ldots, \mu_{p-1}, x, y, \mu_{p+2}, \ldots, \mu_z).$
The key idea is to describe the relations between the elements $^\mu E_{b;h,k}^{(r)}$ and $^\nu E_{b;h,k}^{(r)}$.

Recall the set $\mathcal{P}_{\mu,\sigma}$ consisting of the following elements in $Y_{\mu}$
\begin{align*}
&\big\lbrace ^\mu D_{a;i,j}^{(r)}  \,|\, {1\leq a\leq z,\; 1\leq i,j\leq \mu_a,\; r\geq 0}\big\rbrace\\
&\big\lbrace ^\mu E_{b;h,k}^{(r)} \,|\, {1\leq b < z,\; 1\leq h\leq \mu_b, 1\leq k\leq\mu_{b+1},\; r> s_{b,b+1}^\mu}\big\rbrace\\
&\big\lbrace ^\mu F_{b;k,h}^{(r)} \,|\, {1\leq b < z,\; 1\leq h\leq \mu_b, 1\leq k\leq\mu_{b+1}, \; r> s_{b+1,b}^\mu}\big\rbrace
\end{align*}
obtained by the Gauss decomposition of $T(u)$ with respect to $\mu$. 
Similarly, replacing $\mu$ by $\nu$, we have the following elements in $Y_{\nu}$ as well
\begin{align*}
&\big\lbrace ^\nu D_{a;i,j}^{(r)}  \,|\, {1\leq a\leq z+1,\; 1\leq i,j\leq \nu_a,\; r\geq 0}\big\rbrace\\
&\big\lbrace ^\nu E_{b;h,k}^{(r)} \,|\, {1\leq b \leq z,\; 1\leq h\leq \nu_b, 1\leq k\leq\nu_{b+1},\; r> s_{b,b+1}^\nu}\big\rbrace\\
&\big\lbrace ^\nu F_{b;k,h}^{(r)} \,|\, {1\leq b \leq z,\; 1\leq h\leq \nu_b, 1\leq k\leq\nu_{b+1}, \; r> s_{b+1,b}^\nu}\big\rbrace
\end{align*}

For every $1\leq a<b\leq z+1$, $1\leq i\leq \nu_a$, $1\leq j\leq \nu_b$, 
we inductively define higher root elements $^\nu E_{a,b;i,j}^{(r)}$ for $r>s_{a,b}^{\nu}$ by equation (\ref{eparag}) and similarly define $^\nu F_{b,a;j,i}^{(r)}$ for $r>s_{b,a}^{\nu}$ by equation (\ref{fparag}).
We further define the formal series in $Y_\nu(\sigma)[[u^{-1}]]$:
\begin{equation}\label{hrne}
^\nu E_{a,b;i,j}(u):= \sum_{r > s_{a,b}^{\nu}} \,^\nu E_{a,b;i,j}^{(r)} u^{-r}, \qquad ^\nu F_{b,a;j,i}(u):= \sum_{r > s_{b,a}^{\nu}} \,^\nu F_{b,a;j,i}^{(r)} u^{-r}.
\end{equation}
Note that the value of $k$ in (\ref{eparag}) and (\ref{fparag}) can be arbitrarily chosen between $1$ and $\nu_{b-1}$ due to Remark \ref{4.8}. 
Moreover, one should be careful that they are in general different from those series in $Y_\nu[[u^{-1}]]$ given by (\ref{badef}) so that we have to slightly modify the argument in the proof of Corollary~\ref{indmu}.
Finally, let $^\nu D_{a;i,j}(u)$ be given as in (\ref{3.7}) with respect to $\nu$.

Using these series, one defines the following matrices 
\begin{align*}
 {^\nu} D_a(u) &=\big({^\nu} D_{a;i,j}(u)\big)_{1 \leq i,j \leq \nu_a}\\
 {^\nu} E_{a,b}(u)&=\big({^\nu} E_{a,b;h,k}(u)\big)_{1 \leq h \leq \nu_a, 1 \leq k \leq \nu_b}\\
 {^\nu} F_{b,a}(u)&=\big({^\nu} F_{b,a;k,h}(u)\big)_{1 \leq k \leq \nu_b, 1 \leq h \leq \nu_a}
\end{align*}
One further defines the block matrices ${^\nu}D(u)$, ${^\nu}E(u)$ and ${^\nu}F(u)$ exactly the same way as (\ref{T=FDE})--(\ref{3.9}), except that we use their product to define the matrix $^\nu G(u)$:
$$ ^\nu G(u):= \, {^\nu}F(u) {^\nu}D(u) {^\nu}E(u) $$

By exactly the same way, one defines the higher root elements $^\mu E_{a,b;i,j}^{(r)}$, $^\mu F_{b,a;j,i}^{(r)}$, formal series $^\mu E_{a,b;i,j}(u)$, $^\mu F_{b,a;j,i}(u)$, $^\mu D_{a;i,j}(u)$, block matrices ${^\mu}D(u)$, ${^\mu}E(u)$ and ${^\mu}F(u)$ and hence their product $^\mu G(u):= \, {^\mu}F(u) {^\mu}D(u) {^\mu}E(u)$.
A key observation from \cite[\textsection 2.4]{BK3} is that these two matrices are in fact the same $^\nu G(u)= {^\mu} G(u)$ and hence we have
$$  {^\nu}F(u) {^\nu}D(u) {^\nu}E(u) =  \, {^\mu}F(u) {^\mu}D(u) {^\mu}E(u) $$
As a consequence of Lemma \ref{split},  
for each $1 \leq a < b \leq z$,
$1 \leq i \leq \mu_a$ and $1 \leq j \leq \mu_b$, we have the following relation 
\begin{align}\label{redu}
{^\mu}E_{a,b;i,j}(u) &=
\left\{
\begin{array}{ll}
{^\nu}E_{a,b;i,j}(u)&\hbox{if $b < p$};\\
{^\nu}E_{a,b;i,j}(u)&\hbox{if $b=p, j \leq x$};\\
{^\nu}E_{a,b+1;i,j-x}(u)&\hbox{if $b=p, j > x$};\\
{^\nu}E_{a,b+1;i,j}(u)&\hbox{if $a < p<b $};\\
{^\nu}E_{a,b+1;i,j}(u)
\\
\qquad -\:\sum_{q=1}^{y}
{^\nu} E_{a,a+1;i,q}(u) {^\nu} E_{a+1,b+1;q,j}(u)\qquad
&\hbox{if $a=p, i \leq x$};\\
{^\nu}E_{a+1,b+1;i-x,j}(u)&\hbox{if $a=p, i > 
x$};\\
{^\nu}E_{a+1,b+1;i,j}(u)&\hbox{if $a > p$}.
\end{array}\right.
\end{align}

Back to (\ref{p715}), we may assume that $f_1=f_2=f$ and $g_1=g_2=g$ by (\ref{p709}).
Moreover, by (\ref{redu}), $^\mu E_{a;i,j}^{(r)}=\,^\nu E_{a;i,j}^{(r)}$ except for $a\in\{p-1,p,p+1\}$ so the general case is further reduced to the special case $\mu=(\mu_1,\mu_2,\mu_3,\mu_4)$ since (\ref{p715}) holds for $\nu$ by induction.
Therefore, it suffices to check the following relation holds in $Y_{\mu}$ for any $t>s_{2,3}^\mu$:
\begin{equation}\label{KY}
\big[\,[ ^{\mu}E_{1;i,f}^{(r)}, \,^{\mu}E_{2;f,j}^{(t)}]\,,\,[ ^{\mu}E_{2;h,g}^{(t)}, \,^{\mu}E_{3;g,k}^{(s)}]\,\big]=0
\end{equation}
This can be checked by a case-by-case discussion. We list all possibilities below:
{\allowdisplaybreaks
\begin{align}
\label{426} p=1,& \qquad 1\leq i\leq x \\
\label{427} p=1,& \qquad 1\leq i-x\leq y\\
\label{428} p=2,& \qquad 1\leq f\leq x, \qquad  1\leq h\leq x\\
\label{429} p=2,& \qquad  1\leq f-x\leq y, \qquad  1\leq h\leq x\\
\label{430} p=2,& \qquad  1\leq f\leq x, \qquad  1\leq h-x \leq y\\
\label{431} p=2,& \qquad  1\leq f-x\leq y, \qquad  1\leq h-x \leq y\\
\label{432} p=3,& \qquad  1\leq g\leq x, \qquad  1\leq j\leq x\\
\label{433} p=3,& \qquad  1\leq g-x\leq y, \qquad  1\leq j\leq x\\
\label{434} p=3,& \qquad  1\leq g\leq x, \qquad  1\leq j-x\leq y\\
\label{435} p=3,& \qquad  1\leq g-x\leq y, \qquad  1\leq j-x\leq y\\
\label{436} p=4,& \qquad  1\leq k\leq x\\
\label{437} p=4,& \qquad  1\leq k-x \leq y
\end{align}
}
We will check some of them in detail here and the remaining ones can be deduced similarly.

Suppose that (\ref{426}) holds. By (\ref{redu}), we have
$${^\mu}E_{1;i,f}^{(r)} ={^\nu}E_{1,3;i,f}^{(r)}-\:\sum_{s^\nu_{2,3}<q<r}\sum_{\ell=1}^{y}
{^\nu} E_{1;i,\ell}^{(r-q)} \, {^\nu} E_{2;\ell,f}^{(q)}$$
Note that the admissible condition implies 
$s^\mu_{1,2}=s^\nu_{2,3}$ so the indices $q$ and $r-q$ make sense.
Then relation (\ref{KY}) becomes
\begin{multline*}
\big[\,[ {^\nu}E_{1,3;i,f}^{(r)}-\:\sum_{s^\nu_{2,3}<q<r}\sum_{\ell=1}^{y}
{^\nu} E_{1;i,\ell}^{(r-q)} \, {^\nu} E_{2;\ell,f}^{(q)}, \,{^\nu}E_{3;f,j}^{(t)}]\,,\,[ {^\nu}E_{3;h,g}^{(t)}, \,{^\nu}E_{4;g,k}^{(s)}]\,\big]=\\
\big[\,[ {^\nu}E_{1,3;i,f}^{(r)}, \,{^\nu}E_{3;f,j}^{(t)}]\,,\,[ {^\nu}E_{3;h,g}^{(t)}, \,{^\nu}E_{4;g,k}^{(s)}]\,\big]
-
\big[\,[ \:\sum_{s^\nu_{2,3}<q<r}\sum_{\ell=1}^{y}
{^\nu} E_{1;i,\ell}^{(r-q)} \, {^\nu} E_{2;\ell,f}^{(q)}, \,{^\nu}E_{3;f,j}^{(t)}]\,,\,[ {^\nu}E_{3;h,g}^{(t)}, \,{^\nu}E_{4;g,k}^{(s)}]\,\big]
\end{multline*}
We first use the relation (\ref{eparag}) to rewrite ${^\nu}E_{1,3;i,f}^{(r)}=(-1)^{\pa{\ell}}[{^\nu} E_{1;i,\ell}^{(r-s^\nu_{2,3})} ,\, {^\nu} E_{2;\ell,f}^{(s^\nu_{2,3}+1)}]$.
Then we use super Jacobi identity twice together with the fact that
${^\nu} E_{1;i,\ell}^{(r-s^\nu_{2,3})}$ and $^{\nu}E_{3;f,j}^{(t)}$ supercommute to rewrite the first term into
$$
(-1)^{\pa{\ell}}\Big[\,{^\nu} E_{1;i,\ell}^{(r-s^\nu_{2,3})}, \big[ \,[ {^\nu} E_{2;\ell,f}^{(s^\nu_{2,3}+1)}, \,{^\nu}E_{3;f,j}^{(t)}]\,,\,[ {^\nu}E_{3;h,g}^{(t)}, \,{^\nu}E_{4;g,k}^{(s)}]\,\big] \,\Big]
$$
Similarly, up to an irrelevant sign factor, one can rewrite the second term as
$$
\sum_{s^\nu_{2,3}<q<r}\sum_{\ell=1}^{y}
{^\nu} E_{1;i,\ell}^{(r-q)} \,\big[\,[ \:
 {^\nu} E_{2;\ell,f}^{(q)}, \,{^\nu}E_{3;f,j}^{(t)}]\,,\,[ {^\nu}E_{3;h,g}^{(t)}, \,{^\nu}E_{4;g,k}^{(s)}]\,\big]
$$
Now both of them are zero since (\ref{p715}) holds for $\nu$ by induction and the case (\ref{426}) is proved.

Suppose that (\ref{433}) holds. Using (\ref{redu}), we rewrite (\ref{KY}) into
$$
\Big[ \,
 [ {^\nu} E_{1;i,f}^{(r)}, {^\nu} E_{2;f,j}^{(t)}]  ,
 [{^\nu} E_{2,4;h,g}^{(t)}, {^\nu} E_{4;g,k}^{(s)}]
 \, \Big]
$$
By relation (\ref{eparag}), we have 
\begin{equation}\label{e24=}
{^\nu}E_{2,4;h,g}^{(t)}=(-1)^{\pa{\ell}}[{^\nu} E_{2;h,\ell}^{(t)} ,\, {^\nu} E_{3;\ell,g}^{(1)}], 
\end{equation}
where it is crucial to use the fact that $s^\nu_{3,4}=0$ due to the admissible condition. Following the same argument given in the case (\ref{426}), one easily deduces that (\ref{KY}) is indeed zero in the case (\ref{433}).

Now we prove the case (\ref{434}). By (\ref{redu}), equation (\ref{KY}) becomes 
\begin{equation}\label{439}
\Big[  \big[ {^\nu} E_{1;i,f}^{(r)}, {^\nu} E_{2,4;f,j}^{(t)} \big] \, , \, 
\big[ \,{^\nu}E_{2;h,g}^{(t)},  \,{^\nu}E_{3,5;g,k}^{(s)}- \sum_{s^\nu_{4,5}<q<s}\sum_{\ell=1}^{\nu_4}\,{^\nu}E_{3;g,\ell}^{(s-q)} \,{^\nu}E_{4;\ell,k}^{(q)} \big] \Big]
\end{equation}
For convenience, write 
$$B={^\nu}E_{3,5;g,k}^{(s)}- \sum_{s^\nu_{4,5}<q<s}\sum_{\ell=1}^{\nu_4}\,{^\nu}E_{3;g,\ell}^{(s-q)} \,{^\nu}E_{4;\ell,k}^{(q)}.$$

We need an extra relation before moving on.
Applying the shift map $\psi_{\nu_1}$ in \cite[Lemma~4.2]{Pe4} to the equation \cite[(6.31)]{Pe4}, one deduces the following relation in $Y_{\nu}[[u^{-1},v^{-1}]]$
\begin{equation}\label{541}
\big[ E_{2,4;f,j}(u), \, E_{3,5;g,k}(v) - \sum_{\ell=1}^{\nu_4} E_{3;g,\ell}(v) E_{4;\ell,k}(v)\big] =0
\end{equation}
We emphasize again that the series $E_{2,4;f,j}(u)$ and $E_{3,5;g,k}(v)$ in (\ref{541}) are given by (\ref{3.8}) and they are in general different from ${^\nu} E_{2,4;f,j}(u)$ and ${^\nu}E_{3,5;g,k}(v)$ defined by (\ref{hrne}).
Fortunately, $s^\nu_{3,4}=0$ due to the admissible condition so that we do have ${^\nu} E_{2,4;f,j}(u)=E_{2,4;f,j}(u)$.
By using (\ref{p709}) in the case $\sigma=0$ multiple times, one deduces that
\[
E_{3,5;g,k}^{(s)}={^\nu} E_{3,5;g,k}^{(s)} + \sum_{j=1}^{s_{4,5}^\nu} \sum_{\ell=1}^{\nu_4} E_{3;g,\ell}^{(s+j-1)} E_{4;\ell,k}^{(j)}.
\]
As a result, we may rewrite (\ref{541}) into the following identity in $Y_\nu(\sigma)$
\begin{equation}\label{she}
\big[ {^\nu}E_{2,4;f,j}^{(t)}, \, {^\nu}E_{3,5;g,k}^{(s)} - \sum_{s^\nu_{4,5}<q<s}\sum_{\ell=1}^{\nu_4} {^\nu}E_{3;g,\ell}^{(s-q)}\,{^\nu}E_{4;\ell,k}^{(q)}\big] = \big[ {^\nu}E_{2,4;f,j}^{(t)}, \, B \big] = 0
\end{equation}
By super Jacobi identity and (\ref{e24=}), we rewrite (\ref{439}) into
\begin{align*}
\Big[  \big[  & {^\nu} E_{1;i,f}^{(r)}, {^\nu} E_{2,4;f,j}^{(t)} \big] \, , \, 
\big[ \,{^\nu}E_{2;h,g}^{(t)},  \,B \big] \Big] \\
= &  \Big[  \big[ \, [ \, {^\nu} E_{1;i,f}^{(r)}, {^\nu} E_{2,4;f,j}^{(t)} \, ] \, , \, {^\nu}E_{2;h,g}^{(t)} \big]  \, , \, B \Big] 
 \pm \Big[ {^\nu}E_{2;h,g}^{(t)} \, , \, \big[ \, [ {^\nu} E_{1;i,f}^{(r)}, {^\nu} E_{2,4;f,j}^{(t)} \, ]  \, , \, B \big] \Big]
\end{align*}
The second term is zero due to (\ref{she}) and the fact that ${^\nu} E_{1;i,f}^{(r)}$ supercommute with $B$, which is a consequence of equation (\ref{p711}). Using (\ref{e24=}) and super Jacobi identity, we rewrite the term inside the bracket of the first term as follows 
\begin{align*}
\Big[ \, [ \, &{^\nu} E_{1;i,f}^{(r)}, {^\nu} E_{2,4;f,j}^{(t)} \, ] \, , \, {^\nu}E_{2;h,g}^{(t)} \Big] = 
\Big[ \, \big[ \, {^\nu} E_{1;i,f}^{(r)}, (-1)^{\pa{\ell}}[{^\nu} E_{2;f,\ell}^{(t)} ,\, {^\nu} E_{3;\ell,j}^{(1)}] \, \big] \, , \, {^\nu}E_{2;h,g}^{(t)} \Big] \\
& = \pm \Big[ \, {^\nu} E_{1;i,f}^{(r)} \, , \,  \big[ {^\nu}E_{2;h,g}^{(t)},   [{^\nu} E_{2;f,\ell}^{(t)} ,\, {^\nu} E_{3;\ell,j}^{(1)}]  \big]  \, \Big] \pm \Big[ \,   \big[ {^\nu} E_{1;i,f}^{(r)} \, , \, {^\nu}E_{2;h,g}^{(t)}\big] \, , \big[{^\nu} E_{2;f,\ell}^{(t)} ,\, {^\nu} E_{3;\ell,j}^{(1)}\big]  \, \Big]
\end{align*}
The first term is zero due to equation (\ref{p713}) while the second term is zero since (\ref{p715}) holds for $\nu$ by induction. This completes the proof of (\ref{KY}) in the case (\ref{434}).

The cases (\ref{431}) and (\ref{435}) are similar to (\ref{433});
the cases (\ref{427}), (\ref{436}) and (\ref{437}) are immediate results of the induction hypothesis; the cases (\ref{430}) and (\ref{432}) are similar to (\ref{426});
the cases (\ref{428}) and (\ref{429}) are similar to (\ref{434}).
\end{proof}


\section{Baby comultiplications}\label{babyco}
Although $Y_{m|n}$ is a hopf-superalgebra,  
the shifted super Yangian $Y_\mu(\sigma)$ is not closed under the comultiplication defined by (\ref{Del}) in general; that is, 
$$\Delta(Y_\mu(\sigma))\nsubseteq Y_\mu(\sigma)\otimes Y_\mu(\sigma).$$
However, one can define some comultiplication-like maps on $Y_{\mu}(\sigma)$ as in \cite[\textsection 4]{BK2}.

We first set up our assumptions and notations throughout this section. Let $\sigma$ be a non-zero shift matrix of size $m+n$ with minimal admissible shape $\mu=(\mu_1,\ldots,\mu_z)$. Let $\bo$ be a fixed $0^m1^n$-sequence and let $Y_\mu(\sigma)$ be the shifted super Yangian defined in \textsection \ref{shiftY2}. Suppose that there are $p$ 0's and $q$ 1's in the very last $\mu_z$ digits of $\bo$; that is, $\bo_z$ is a $0^p1^q$-sequence and $\mu_z=p+q$.
Since $\mu$ is minimal admissible and $\sigma\neq 0$, we have that $1\leq \mu_z< m+n$ and either $s_{m+n-\mu_z,m+n+1-\mu_z}\neq 0$ or $s_{m+n+1-\mu_z,m+n-\mu_z}\neq 0$.

\begin{theorem}\label{baby1}
Let $\mu=(\mu_1,\mu_2,\ldots,\mu_{z})$ be minimal admissible to $\sigma$. For $1\leq i,j\leq \mu_z$, define
 \[
 \tilde e_{i,j}:=e_{i,j}+\delta_{i,j}((m-p)-(n-q))\in U(\gl_{p|q}).
 \]
 Here $e_{i,j}$ is the elementary matrix identified with the element in $\gl_{p|q}$ and its parity is determined by the $0^p1^q$-sequence $\bo_z$.
\begin{enumerate}
\item[(1)] Suppose that $s_{m+n-\mu_z,m+n+1-\mu_z}\neq 0$. Define $\dot\sigma = (\dot s_{i,j})_{1 \leq i,j \leq m+n}$ by
     \begin{equation}\label{babyr1}
     \dot s_{i,j} = \left\{
     \begin{array}{ll}
      s_{i,j}-1&\hbox{if\, $i \leq m+n-\mu_z < j$,}\\
      s_{i,j}&\hbox{otherwise.}
      \end{array}\right.
      \end{equation}
     Then the map $\Delta_{R}:Y_{m|n}(\sigma) \rightarrow Y_{m|n}(\dot\sigma) \otimes U(\mathfrak{gl}_{p|q})$
     defined by
     \begin{align*}
       D_{a;i,j}^{(r)} &\mapsto \dot D_{a;i,j}^{(r)} \otimes 1
       + \delta_{a,z} \sum_{f=1}^{\mu_z} (-1)^{\pa{f}_z}\dot D_{a;i,f}^{(r-1)} \otimes \tilde e_{f,j},\\
       E_{b;h,k}^{(r)} &\mapsto \dot E_{b;h,k}^{(r)} \otimes 1
       + \delta_{b,z-1} \sum_{f=1}^{\mu_z} (-1)^{\pa{f}_z}\dot E_{b;h,f}^{(r-1)} \otimes \tilde e_{f,k},\\
       F_{b;k,h}^{(r)} & \mapsto \dot F_{b;k,h}^{(r)} \otimes 1,
      \end{align*}
       is a superalgebra homomorphism.\\[4mm]
\item[(2)] Suppose that $s_{m+n+1-\mu_z,m+n-\mu_z}\neq 0$. Define $\dot\sigma = (\dot s_{i,j})_{1 \leq i,j \leq m+n}$ by
     \begin{equation}\label{babyl1}
     \dot s_{i,j} = \left\{
     \begin{array}{ll}
      s_{i,j}-1&\hbox{if\, $j \leq m+n-\mu_z < i$,}\\
      s_{i,j}&\hbox{otherwise.}
      \end{array}\right.
      \end{equation}
       Then the map $\Delta_{L}:Y_{m|n}(\sigma) \rightarrow U(\mathfrak{gl}_{p|q})\otimes Y_{m|n}(\dot\sigma)$
       defined by 
       \begin{align*}
       D_{a;i,j}^{(r)} &\mapsto 1\otimes\dot D_{a;i,j}^{(r)}
       + \delta_{a,z} (-1)^{\pa{i}_z}\sum_{k=1}^{\mu_z} \tilde e_{i,k}\otimes \dot D_{a;k,j}^{(r-1)},\\
       E_{b;h,k}^{(r)} &\mapsto 1\otimes \dot E_{b;h,k}^{(r)} ,\\
       F_{b;k,h}^{(r)} & \mapsto 1\otimes \dot F_{b;k,h}^{(r)}+
       \delta_{b,z-1} (-1)^{\pa{i}_z}\sum_{f=1}^{\mu_z} \tilde e_{k,f}\otimes \dot F_{b;f,h}^{(r-1)},
      \end{align*}
       is a superalgebra homomorphism.
       \end{enumerate}
To avoid possible confusion, in the above description and hereafter, the parabolic generators of $Y_{m|n}(\dot{\sigma})$ are denoted by $\dot D_{a;i,j}^{(r)}$, $\dot E_{a;i,j}^{(r)}$, and $\dot F_{a;i,j}^{(r)}$, where $\dot{\sigma}$ is the shift matrix defined by either {\em (\ref{babyr1})} or {\em (\ref{babyl1})}, with respect to the same shape $\mu$ which is also admissible to $\dot\sigma$.   
\end{theorem}
\begin{proof}
Check that $\Delta_{R}$ and $\Delta_{L}$ preserve the defining relations in Definition \ref{parashift}. Similar to \cite[Theorem 4.2]{BK2}, to check (\ref{p713}) and (\ref{p714}), one needs to use (\ref{p707}), (\ref{p708}), (\ref{p709}) and (\ref{p710}) multiple times.
Note that it suffices to check the special case when $z=4$ since the non-trivial situations only happen in the very last block.

We check (\ref{p716}) here as an illustrating example since it is a super phenomenon which does not appear in \cite{BK2}.
Assume $z=4$ and (\ref{babyl1}) holds. Applying $\Delta_L$ to the left-hand-side of (\ref{p716}), we have 
$$\big[\,[1\otimes \dot F_{1;i,f_1}^{(r)},1\otimes \dot F_{2;f_2,j}^{(t)}]\,,\,[1\otimes \dot F_{2;h,g_1}^{(t)},1\otimes \dot F_{3;g_2,k}^{(s)}+(-1)^{\pa{g_2}_4} \sum_{x=1}^{\mu_4} \tilde{e}_{g_2,x} \otimes \dot F_{3;x,k}^{(s-1)}]\,\big].$$ 
It equals to
\begin{multline*}
1\otimes \big[\,[\dot F_{1;i,f_1}^{(r)},\dot F_{2;f_2,j}^{(t)}]\,,\,[\dot F_{2;h,g_1}^{(t)},\dot F_{3;g_2,k}^{(s)}]\,\big]\\
+ \theta \sum_{x=1}^{\mu_4} \tilde{e}_{g_2,x}\otimes \big[\,[\dot F_{1;i,f_1}^{(r)},\dot F_{2;f_2,j}^{(t)}]\,,\,[\dot F_{2;h,g_1}^{(t)},\dot F_{3;x,k}^{(s-1)}]\big],
\end{multline*}
where $\theta=\pm1$ is an irrelevant sign. 
It vanishes due to (\ref{p716}) in $Y_{m|n}(\dot\sigma)$.
\end{proof}

The next lemma computes the images of higher root elements $E_{a,b;i,j}^{(r)}$ and $F_{b,a;i,j}^{(r)}$ under $\Delta_R$ and $\Delta_L$.

\begin{lemma}\label{baby2}
\begin{itemize}
\item[(\it{1})] Suppose the assumption of Theorem~{\em \ref{baby1}(1)} holds. For all admissible indices $i,j,r$ and $1\leq a< b-1<z$, we have
\begin{align*}
\Delta_{R}(F_{b,a;i,j}^{(r)}) &= \dot F_{b,a;i,j}^{(r)} \otimes 1,\\
\Delta_{R}(E_{a,b;i,j}^{(r)}) &= \!\dot E_{a,b;i,j}^{(r)} \otimes 1\qquad \hbox{  if $\,\,b < z$},
\end{align*} 
and
\begin{multline*}
\Delta_{R}(E_{a,z;i,j}^{(r)}) = (-1)^{\pa{h}_{z-1}} [\dot E_{a,z-1;i,h}^{(r-s_{z-1,z}^\mu)},\dot E_{z-1;h,j}^{(s_{z-1,z}^\mu +1)}]\otimes 1 
+\sum_{k=1}^{\mu_z}(-1)^{\pa{k}_z}\dot E_{a,z;i,k}^{(r-1)} \otimes \tilde e_{k,j},
\end{multline*}
for any $1 \leq h \leq \mu_{z-1}$.
\item[(\it{2})] Suppose the assumption of Theorem~{\em \ref{baby1}(2)} holds. For all admissible indices $i,j,r$ and $1\leq a<b-1<z$, we have
\begin{align*}
\Delta_{L}(E_{a,b;i,j}^{(r)}) &= 1 \otimes \dot E_{a,b;i,j}^{(r)} ,\\
\Delta_{L}(F_{b,a;i,j}^{(r)}) &= 1 \otimes \!\dot F_{b,a;i,j}^{(r)}\qquad \hbox{  if $\,\,b < z$},
\end{align*} 
and
\begin{multline*}
\Delta_{L}(F_{z,a;i,j}^{(r)}) = (-1)^{\pa{h}_{z-1}}
\big( 1 \otimes [\dot F_{z-1;i,h}^{(s_{z,z-1}^\mu+1)},\dot F_{z-1,a;h,j}^{(r-s_{z,z-1}^\mu)}] \big)
+(-1)^{\pa{i}_{z}}\sum_{k=1}^{\mu_z} \tilde e_{i,k} \otimes \dot F_{z-1,a;k,j}^{(r-1)},
\end{multline*}
for any $1 \leq h \leq \mu_{z-1}$.
\end{itemize}
\end{lemma}

\begin{proof}
We compute $\Delta_{R}(E_{a,z;i,j}^{(r)})$ for $1\leq a\leq z-1$ in detail here, while others are similar. By definition, for any $1\leq h\leq\mu_{z-1}$, we have
\[E_{a,z;i,j}^{(r)}=(-1)^{\pa{h}_{z-1}}
[E_{a,z-1;i,h}^{(r-s_{z-1,z}^\mu)},E_{z-1;h,j}^{(s_{z-1,z}^\mu+1)}].\]
Also, $\Delta_{R}(E_{a,z-1;i,h}^{(r-s_{z-1,z}^\mu)})= 
\dot E_{a,z-1;i,h}^{(r-s_{z-1,z}^\mu)} \otimes 1$.
Hence
\begin{align*}
\Delta_{R}(E_{a,z;i,j}^{(r)})&=
(-1)^{\pa{h}_{z-1}}\left[\dot E_{a,z-1;i,h}^{(r-s_{z-1,z}^\mu)} \otimes 1,
\dot E_{z-1;h,j}^{(s_{z-1,z}^\mu+1)} \otimes 1 \right]\\
&\qquad +(-1)^{\pa{h}_{z-1}}\left[ \dot E_{a,z-1;i,h}^{(r-s_{z-1,z}^\mu)} \otimes 1, 
\sum_{k=1}^{\beta}(-1)^{\pa{k}_z}\dot E_{z-1;h,k}^{(s_{z-1,z}^\mu)} \otimes \tilde e_{k,j}\right]\\
&=
(-1)^{\pa{h}_{z-1}}\left[\dot E_{a,z-1;i,h}^{(r-s_{z-1,z}^\mu)},
\dot E_{z-1;h,j}^{(s_{z-1,z}^\mu+1)}\right ] \otimes 1 +\sum_{k=1}^{\mu_z}(-1)^{\pa{k}_z}\dot E_{a,z;i,k}^{(r-1)} \otimes \tilde e_{k,j}.
\end{align*}
\end{proof}

\begin{proposition}\label{babyinj}
If the assumption of Theorem~{\em \ref{baby1}(1)} holds, then $\Delta_R$ is injective. 
Similarly, if the assumption of Theorem~{\em \ref{baby1}(2)} holds, then $\Delta_L$ is injective. 
\end{proposition}
\begin{proof}
Let $\epsilon:U(\mathfrak{gl}_{p|q})\rightarrow \C$ be the homomorphism such that 
$$\epsilon(\tilde e_{i,j})=0$$
for $1 \leq i,j \leq \mu_z$. By definition, $Y_{\mu}(\sigma) \subseteq Y_{\mu}(\dot\sigma)\subseteq Y_{\mu}$ is a chain of subalgebras. Note that the compositions $m\circ (\id \otimes \epsilon) \circ \Delta_{R}$ and $m\circ(\epsilon \otimes \id)\circ \Delta_{L}$ coincide with the natural embedding $Y_{\mu}(\sigma) \hookrightarrow Y_{\mu}(\dot\sigma)$, where $m(a\otimes b):=ab$ is the usual multiplication map. This deduces that the maps $\Delta_R$ and $\Delta_L$ are injective whenever they are defined.
\end{proof}

\section{Canonical filtration}\label{canfil}
There is another filtration on $Y_{m|n}$, called the {\em canonical filtration}
\[
F_0Y_{m|n}\subseteq F_1Y_{m|n}\subset F_2Y_{m|n}\subseteq \cdots
\]
defined by deg $t_{ij}^{(r)}:=r$ where $F_dY_{m|n}$ is defined to be the span of all supermonomials in $t_{ij}^{(r)}$ of total degree not greater than $d$. Let $\gr Y_{m|n}$ denote the associated superalgebra, which is supercommutative by (\ref{RTT}).

Now we describe the canonical filtration using parabolic presentations. Let $\mu=(\mu_1,\ldots,\mu_{z})$ be a composition of $m+n$. By \cite[Proposition 3.1]{Pe4}, the parabolic generators $D_{a;i,j}^{(r)}$ $E_{a,b;i,j}^{(r)}$ and $F_{b,a;i,j}^{(r)}$ of $Y_{\mu}=Y_{m|n}$ are linear combinations of supermonomials in $t_{i,j}^{(s)}$ of total degree $r$.

On the other hand, if we set $D_{a;i,j}^{(r)}$, $E_{a,b;i,j}^{(r)}$ and $F_{b,a;i,j}^{(r)}$ all to be of degree $r$, by multiplying the matrix equation $T(u)=F(u)D(u)E(u)$, each $t_{ij}^{(r)}$ is a linear combination of supermonomials in the parabolic generators of total degree $r$ as well.
Thus $F_dY_{m|n}$ can be alternatively defined as the span of all supermonomials in the parabolic generators $D_{a;i,j}^{(r)}$ $E_{a,b;i,j}^{(r)}$ and $F_{b,a;i,j}^{(r)}$ of total degree $\leq d$.

For $1\leq a,b\leq z$, $1\leq i\leq \mu_a$, $1\leq j\leq \mu_b$ and $r>0$, 
define the following elements in $\gr Y_{\mu}$ by
\begin{equation}\label{canorel}
e_{a,b;i,j}^{(r)} := \left\{
\begin{array}{ll}
\gr_r D_{a;i,j}^{(r)}&\hbox{if $a=b$,}\\
\gr_r E_{a,b;i,j}^{(r)}&\hbox{if $a < b$,}\\
\gr_r F_{a,b;i,j}^{(r)}&\hbox{if $a > b$.}
\end{array}
\right.
\end{equation}
Since $\gr Y_{\mu}$ is supercommutative, together with Corollary \ref{pbw2} ($4$), the following result can be deduced immediately.

\begin{proposition}\label{canopbw}\cite[Theorem 5.1]{BK2}
For any shape $\mu=(\mu_1,\ldots,\mu_{z})$, $\gr Y_{\mu}$ is the free supercommutative superalgebra on generators $\lbrace e_{a,b;i,j}^{(r)}\,|\, {1\leq a,b\leq {z}, 1\leq i\leq \mu_a, 1\leq j\leq \mu_{b}, r>0}\rbrace$.
\end{proposition}

Suppose now $\sigma$ is a shift matrix of size $m+n$ and $\mu=(\mu_1,\ldots,\mu_z)$ is an admissible shape to $\sigma$.
We induce the canonical filtration of $Y_\mu$ to the subalgebra $Y_\mu(\sigma)$ by  
defining 
$$F_dY_{\mu}(\sigma):=F_dY_{\mu}\cap Y_{\mu}(\sigma).$$ 
The natural embedding $Y_{\mu}(\sigma)\hookrightarrow Y_{\mu}$ is a filtered map 
and the induced map $\gr Y_{\mu}(\sigma)\rightarrow \gr Y_{\mu}$ is injective as well,
so that we may identify $\gr Y_{\mu}(\sigma)$ as a subalgebra of $\gr Y_{\mu}$. The next theorem gives a set of generators of $\gr Y_{\mu}(\sigma)$.

\begin{theorem}\label{scanopbw}\cite[Theorem 5.2]{BK2}
For an admissible shape $\mu=(\mu_1,\ldots,\mu_{z})$, $\gr Y_{\mu}(\sigma)$ is the subalgebra of $\gr Y_{\mu}$ generated by the elements 
$$\lbrace e_{a,b;i,j}^{(r)} \,|\, {1\leq a,b\leq z, 1\leq i\leq \mu_a, 1\leq j\leq \mu_b, r>s_{a,b}^\mu}\rbrace.$$
\end{theorem}
\begin{proof}
By relations (\ref{p709}) and (\ref{p710}), the elements $e_{a,b;i,j}^{(r)}$ of $\gr Y_{\mu}(\sigma)$ can be identified as the elements of the same notation in $\gr Y_{\mu}$ defined in (\ref{canorel}) by the embedding $\gr Y_{\mu}(\sigma)\rightarrow \gr Y_{\mu}$. Now the statement follows from Corollary \ref{pbw2} ($4$) and Proposition~\ref{canopbw}.
\end{proof}

One consequence of Theorem \ref{scanopbw} is that we may define the canonical filtration on $Y_{\mu}(\sigma)$ intrinsically by setting the degree of the elements $D_{a;i,j}^{(r)}$, $E_{a,b;i,j}^{(r)}$ and $F_{a,b;i,j}^{(r)}$ in $Y_{\mu}(\sigma)$ to be $r$.  
By Corollary \ref{indmu}, this definition is independent of the choice of admissible shape $\mu$.

By definition, the comultiplication $\Delta:Y_{\mu}\rightarrow Y_{\mu}\otimes Y_{\mu}$ is a filtered map with respect to the canonical filtration. If we extend the canonical filtration of $Y_{\mu}(\dot{\sigma})$ to $Y_{\mu}(\dot{\sigma})\otimes U(\gl_{p|q})$ by declaring the degree of the matrix unit $e_{ij}\in\gl_{p|q}$ to be 1, then the baby comultiplications $\Delta_R$ and $\Delta_L$ defined in Theorem \ref{baby1}, as long as they are defined, are filtered maps as well. 
Moreover, the same argument in Proposition \ref{babyinj} implies that the associated graded maps 
$$\gr\Delta_L: \gr Y_{\mu}(\dot{\sigma}) \rightarrow \gr \big(Y_{\mu}(\dot{\sigma})\otimes U(\gl_{p|q})\big)  \qquad \gr\Delta_R:Y_{\mu}(\dot{\sigma})\rightarrow \gr \big( U(\gl_{p|q})\otimes Y_{\mu}(\dot{\sigma})\big)$$ 
are injective as well. We state this fact as a proposition.

\begin{proposition}\label{injdelr}\cite[Remark 5.4]{BK2}
The induced maps $\gr\Delta_{R}$ and $\gr\Delta_{L}$ are injective whenever they are defined,
\end{proposition}

\section{Truncation}\label{truncation}
Let $\sigma$ be a fixed shift matrix of size $m+n$.
Choose an integer $\ell> s_{1,m+n}+s_{m+n,1}$, which will be called {\em level} later. For each $1\leq i\leq m+n$, set
\begin{equation}\label{defpi}
p_i:=\ell-s_{i,m+n}-s_{m+n,i}.
\end{equation}
This defines a tuple $(p_1,\ldots,p_{m+n})$ of integers such that $0< p_1\leq \cdots\leq p_{m+n}=\ell$.
Let $\mu=(\mu_1,\ldots,\mu_{z})$ be an admissible shape for $\sigma$. For each $1\leq a\leq z$, set
\begin{equation}\label{pamu}
 p_a^\mu:=p_{\mu_1+\ldots+\mu_a}.
\end{equation}
Since $\mu$ is admissible, together with (\ref{sijk}), for any $1\leq a\leq z$, we have
$p_i=p_a^\mu$ for any value of $i$ such that ${\displaystyle 1\leq  i- \sum_{k=1}^{a-1} \mu_k \leq  \mu_a}$.

Following \cite[\textsection 6]{BK2}, we define the
{\em shifted super Yangian of level} $\ell$, denoted by $Y_{\mu}^\ell(\sigma)$, to be the quotient of $Y_{\mu}(\sigma)$ by the two-side ideal of $Y_{\mu}(\sigma)$ generated by the elements 
$$\lbrace D_{1;i,j}^{(r)}\,|\, 1\leq i,j\leq \mu_1, r>p_1\rbrace.$$

We claim that the definition of $Y_{\mu}^\ell(\sigma)$ is independent of the choice of the admissible shape $\mu$ so that we may simply write $Y_{m|n}^\ell(\sigma)$ when appropriate. 
Let $I_\mu$ denote the two-sided ideal associated to $\mu$ as in the definition. Since $\nu=(1^{m+n})$ is admissible for any $\sigma$, it suffices to prove that $I_\mu=I_\nu$.

By definition, we have $^\nu D_1^{(r)}=t_{1,1}^{(r)}$. 
Assume $\mu$ is an arbitrary admissible shape. By \cite[(3.10)]{Pe4}, we have $^\mu D_{1;1,1}^{(r)}=t_{1,1}^{(r)}$ and hence  $I_\nu\subseteq I_\mu$.
On the other hand, by relation (\ref{p703}),
we have $^\mu D_{1;i,j}^{(r)}\in I_\nu$ for all $1\leq i,j\leq \mu_1$, $r>p_1$ and hence $I_\mu=I_\nu$.

When $\sigma=0$,
the two-sided ideal is generated by $\{t_{i,j}^{(r)} \, | \, 1\leq i,j\leq m+n, \, r>\ell \}$.
In this special case, the quotient is exactly the {\em truncated super Yangian} in \cite{BR, Pe2}, which is a
super analogy of {\em Yangian of level $\ell$} due to Cherednik \cite{C1,C2}.
It should be clear from the context that we are dealing with $Y_{\mu}(\sigma)$ or the quotient $Y_{\mu}^\ell(\sigma)$ and hence, by abusing notation, we will use the same symbols $D_{a;i,j}^{(r)}$, $E_{a,b;i,j}^{(r)}$ and $F_{b,a;i,j}^{(r)}$ to denote the elements in $Y_{\mu}(\sigma)$ and their images in the quotient $Y_{\mu}^\ell(\sigma)$.

It is obvious that the anti-isomorphism $\tau$ defined in (\ref{taudef}) factors through the quotient and induces an anti-isomorphism
\begin{equation}\label{tauiso}
\tau:Y_{\mu}^\ell(\sigma)\rightarrow Y_{\mu}^\ell(\sigma^t).
\end{equation}
Similarly, let $\vec{\sigma}$ be another shift matrix satisfying that $\vec{s}_{i,i+1}+\vec{s}_{i+1,i}=s_{i,i+1}+s_{i+1,i}$ for all $1\leq i\leq m+n-1$. Then the isomorphism $\iota$ defined by (\ref{iotadef}) also induces an isomorphism
\begin{equation}\label{iotaiso}
\iota:Y_{\mu}^\ell(\sigma)\rightarrow Y_{\mu}^\ell(\vec{\sigma}).
\end{equation}

Recall the canonical filtration defined in \textsection \ref{canfil}. We obtain a filtration
\[
F_0Y_{\mu}^\ell(\sigma)\subseteq F_1Y_{\mu}^\ell(\sigma)\subseteq\cdots
\]
induced from the quotient map $Y_{\mu}(\sigma)\rightarrow Y_{\mu}^\ell(\sigma)$,  where 
we define the elements $D_{a;i,j}^{(r)}$, $E_{a,b;i,j}^{(r)}$ and $F_{b,a;i,j}^{(r)}$ of $Y_{\mu}^\ell(\sigma)$ to be of degree $r$ and $F_dY_{\mu}^\ell(\sigma)$ is the span of all supermonomials in these elements of total degree $\leq d$.

For $1\leq a,b\leq z$, $1\leq i\leq \mu_a$, $1\leq j\leq \mu_b$ and $r>s_{a,b}^\mu$, define element $e_{a,b;i,j}^{(r)}$ (by abusing notation again) in the associative graded superalgebra $\gr Y_{\mu}^\ell(\sigma)$ according to exactly the same formula (\ref{canorel}), except that now our $D$'s, $E$'s and $F$'s here are in the quotient. 
By Proposition~\ref{canopbw} and Theorem~\ref{scanopbw}, $\gr Y_{\mu}^\ell(\sigma)$ is also supercommutative and is generated by the elements 
$$\lbrace e_{a,b;i,j}^{(r)}\in\gr Y_{\mu}^\ell(\sigma)  \,|\, 1\leq a,b\leq z, 1\leq i\leq\mu_a, 1\leq j\leq \mu_b, r>s_{a,b}^\mu\rbrace$$
Following the same argument in \cite[Lemma 6.1]{BK2}, one may deduce that $\gr Y_{\mu}^\ell(\sigma)$ is in fact {\em finitely generated}.
\begin{lemma}\label{lema6.2}
For any admissible shape $\mu=(\mu_1,\ldots,\mu_z)$, $\gr Y_{\mu}^\ell(\sigma)$ is generated only by the elements 
$$\lbrace e_{a,b;i,j}^{(r)} \,|\, 1\leq a,b\leq z, 1\leq i\leq\mu_a, 1\leq j\leq \mu_b, 
s_{a,b}^\mu<r\leq s_{a,b}^\mu+p_{min(a,b)}^\mu\rbrace.$$
\end{lemma}

Let $\sigma=(s_{ij})_{1\leq i,j\leq m+n}$ be a non-zero shift matrix with minimal admissible shape $\mu=(\mu_1,\ldots,\mu_z)$ and let $\bo$ be a $0^m1^n$-sequence.
Then $\mu_z$ equals to the size of the largest zero square matrix in the southeastern corner of $\sigma$. Hence we have $1\leq \mu < m+n$ and either $s_{m+n-\mu_z,m+n+1-\mu_z}\neq 0$ or $s_{m+n+1-\mu_z,m+n-\mu_z}\neq 0$. Let $p$ and $q$ denote the the number of 0's and 1's respectively in the last $\mu_z$ digits of the $0^m1^n$-sequence $\bo$. 

If $s_{m+n-\mu_z,m+n+1-\mu_z}\neq 0$, then the baby comultiplication $\Delta_R$ defined in Theorem~\ref{baby1} factors through the quotient and we obtain an induced map 
\begin{equation}\label{babylr}
\Delta_R:Y_{\mu}^\ell(\sigma)\rightarrow Y_{\mu}^{\ell-1}(\dot{\sigma})\otimes U(\gl_{p|q})
\end{equation}
where $\dot{\sigma}$ is given by (\ref{babyr1}). 

Similarly, if $s_{m+n+1-\mu_z,m+n-\mu_z}\neq 0$, then 
$\Delta_L$ induces a map 
 \begin{equation}\label{babyll}
\Delta_L:Y_{\mu}^\ell(\sigma)\rightarrow U(\gl_{p|q})\otimes Y_{\mu}^{\ell-1}(\dot{\sigma})
\end{equation}
where $\dot{\sigma}$ is given by (\ref{babyl1}).

Recall that $\Delta_R$ and $\Delta_L$ are filtered maps with respect to the canonical filtration, so they induce the following homomorphisms of graded superalgebras
\begin{align}\label{grdelr}
\gr \Delta_R: \gr Y_{\mu}^\ell(\sigma)\rightarrow \gr\big( Y_{\mu}^{\ell-1}(\dot{\sigma})\otimes U(\gl_{p|q})\big),\\ \label{grdell}
\gr \Delta_L: \gr Y_{\mu}^\ell(\sigma)\rightarrow \gr\big(U(\gl_{p|q})\otimes Y_{\mu}^{\ell-1}(\dot{\sigma})\big).
\end{align}

\begin{theorem}\label{PBWSYLpara}
For any admissible shape $\mu=(\mu_1,\ldots,\mu_z)$, $\gr Y_{\mu}^\ell(\sigma)$ is the free supercommutative superalgebra on generators 
\[
\lbrace e_{a,b;i,j}^{(r)} \,|\, 1\leq a,b\leq z, 1\leq i\leq\mu_a, 1\leq j\leq \mu_b, 
s_{a,b}^\mu<r\leq s_{a,b}^\mu+p_{min(a,b)}^\mu \rbrace.
\]
Also, the maps $\gr\Delta_{R}$ and $\gr\Delta_{L}$ in {\em (\ref{grdelr})} and  {\em (\ref{grdell})} are injective whenever they are defined, and so are the maps $\Delta_R$ and $\Delta_L$ in  {\em (\ref{babylr})} and  {\em (\ref{babyll})}.
\end{theorem}
\begin{proof}
Similar to the argument in \cite[Theorem 6.2]{BK2}, except that our induction starts from $\ell=1$.
In that case, the assertion follows from \cite[Proposition~2.3]{Pe2}.
\end{proof}

As a corollary, we obtain a PBW basis for $Y_{m|n}^\ell(\sigma)$.
\begin{corollary}\label{pbwbasis}
For any admissible shape $\mu=(\mu_1,\ldots,\mu_{z})$, the supermonomials in the elements
\[\lbrace D_{a;i,j}^{(r)}| 1\leq a\leq z, 1\leq i,j\leq \mu_{a}, 0<r\leq p_a^\mu\rbrace,\]
\[\lbrace E_{a,b;i,j}^{(r)}| 1\leq a<b\leq z, 1\leq i\leq \mu_a, 1\leq j\leq \mu_b , s_{a,b}^\mu<r\leq s_{a,b}^{\mu}+p_a^\mu\rbrace,\]
\[\lbrace F_{b,a;i,j}^{(r)}| 1\leq a<b\leq z, 1\leq i\leq \mu_b, 1\leq j\leq \mu_{a} , s_{b,a}^\mu<r\leq s_{b,a}^{\mu}+p_a^\mu\rbrace,\]
taken in any fixed order forms a basis for $Y_{m|n}^\ell(\sigma)$.
\end{corollary}

Another corollary is obtained by counting.
\begin{corollary}\label{dimcoro}
Consider $Y_{m|n}^\ell(\sigma)$ together with the canonical filtration and some fixed $\bo$. Let $S(\g^e)$ be the supersymmetric superalgebra of $\g^e$ with the Kazhdan filtration, where $e$ is the nilpotent element corresponding to the triple $(\sigma,\ell, \bo)$ as explained in \textsection {\em\ref{preW}}.
Denote by $F_dY_{m|n}^\ell(\sigma)$ and $F_dS(\g^e)$ the superspaces with total degree not greater than $d$ in the associated filtered superalgebras respectively.
Then for each $d\geq 0$, we have $\dim F_dY_{m|n}^\ell(\sigma) = \dim F_dS(\g^e)$.
\end{corollary}
\begin{proof}
Take $\mu=(1^{m+n})$ in Theorem \ref{PBWSYLpara}. Then the statement follows from Proposition \ref{counting2} and induction on $d$.
\end{proof}

\begin{remark}
Consider the following inverse system
\[
Y_{m|n}^\ell(\sigma)\twoheadleftarrow Y_{m|n}^{\ell+1}(\sigma)\twoheadleftarrow Y_{m|n}^{\ell+2}(\sigma)\twoheadleftarrow \cdots
\]
where the maps are homomorphisms of filtered superalgebras with respect to the canonical filtration.
As an observation from Corollary~{\em \ref{pbw2}} {\em(4)} and Corollary~{\em\ref{pbwbasis}}, we have 
$$ Y_{m|n}(\sigma)= \lim_{\leftarrow} Y_{m|n}^\ell(\sigma)$$
where the inverse limit is taken in the category of filtered superalgebras.
Similar to \cite[Remark 6.4]{BK2}, we may view $Y_{m|n}(\sigma)$ as the inverse limit $\ell\rightarrow \infty$ of the shifted super Yangian of level $\ell$.
\end{remark}


\section{Invariants}\label{Inv}
Let $\pi$ be a given pyramid of height $m+n$ associated to a $0^m1^n$-sequence $\bo$. 
Let $M$ and $N$ be the number of boxes in $\pi$ labeled by $``+"$ and $``-"$, respectively.
Let $\mathfrak{p}$ and $\mathfrak{m}$ be the subalgebras of $\gl_{M|N}$ associated to the good pair $(e_\pi, h_{\pi})$.
Generalizing \cite[\textsection 9]{BK2}, we will define some distinguished $\mathfrak{m}$-invariant (under the $\chi$-twisted action) elements in $U(\mathfrak{p})$; that is, some elements in $\W_\pi$.

We number the columns of $\pi$ from left to right by $1,\ldots,\ell$.
Let $h=m-n$ and let $(\check{q}_1,\dots,\check{q}_{\ell})$ denote the {\em super column heights} of $\pi$, where each $\check{q}_i$ is defined to be the number of boxes in the $i$-th column of $\pi$ labeled with $``+"$ subtract the number of boxes labeled with $``-"$ in the same column.

Define $\rho = (\rho_1,\dots,\rho_\ell)$, where $\rho_r$ is given by
\begin{equation}\label{rhodef}
\rho_r := h-\check{q}_{r} - \check{q}_{r+1} -\cdots-\check{q}_\ell
\end{equation}
for each $r=1,\dots,\ell$.

Give an order on the index set $I:=\lbrace 1<\ldots<M < \ovl{1}<\ldots<\ovl{N}\rbrace$. For all $i,j\in I$, define
\begin{equation}\label{etil}
\tilde e_{i,j} := (-1)^{\col(j)-\col(i)} 
(e_{i,j} + \delta_{i,j} (-1)^{\tp{(i)}}\rho_{\col(i)}),
\end{equation}
where $\tp{(i)}:= 0$ if $i\in\{ 1 ,\ldots, M \}$ and $\tp{(i)}:= 1$ otherwise. 
Note that
the parity notation $\tp{(i)}$ used here is for $\gl_{M|N}$, while another parity notation $\pa i$ 
defined in \textsection \ref{preY} is used for $Y_{m|n}$.

Calculation shows that \begin{multline}\label{etilrel}
[\tilde e_{i,j}, \tilde e_{h,k}]
=(\tilde{e}_{i,k} - \delta_{i,k} (-1)^{\tp(i)} \rho_{\col(i)})\delta_{h,j}\\
- (-1)^{(\tp(i)+\tp(j))(\tp(h)+\tp(k))}
\delta_{i,k} (\tilde e_{h,j} - \delta_{h,j} (-1)^{\tp(j)}\rho_{\col(j)}).
\end{multline}

The effect of the homomorphism $U(\mathfrak{m}) \rightarrow \C$ induced by the character $\chi$ can be obtained easily by definition. 
We explicitly give the result here since it will be frequently use later.
For any $i,j\in I$, we have
\begin{equation}\label{chidef}
\chi (\tilde e_{i,j}) = \left\{
\begin{array}{ll}
(-1)^{\tp(i)+1}&\hbox{if $\row(i)=\row(j)$ and $\col(i) = \col(j)+1$;}\\[3mm]
0&\hbox{otherwise.}
\end{array}\right.
\end{equation}

Now we are going to define certain crucial elements in the universal enveloping algebra $U(\gl_{M|N})$.
For $1\leq i,j\leq m+n$ and signs
$\sigma_i\in \{\pm\}$,
we firstly set
$$T_{ i,j;\sigma_{1},\dots,\sigma_{n+1}}^{(0)} := \delta_{i,j} \sigma_i$$
and then for $r \geq 1$ we define
\begin{equation}\label{tdef}
T_{i,j;\sigma_{1},\ldots,\sigma_{m+n}}^{(r)}
:=
\sum_{s = 1}^r
\sum_{\substack{i_1,\dots,i_s\\j_1,\dots,j_s}}
\sigma_{\row(j_1)} \cdots \sigma_{\row(j_{s-1})}
(-1)^{\tp(i_1)+\cdots+\tp(i_s)}
 \tilde e_{i_1,j_1} \cdots \tilde e_{i_s,j_s}
\end{equation}
where the second sum is taken over all $i_1,\dots,i_s,j_1,\dots,j_s\in I$
such that
\begin{itemize}
\item[(1)] $\deg(e_{i_1,j_1})+\cdots+\deg(e_{i_s,j_s}) = r$;
\item[(2)] $\col(i_t) \leq \col(j_t)$ for each $t=1,\dots,s$;
\item[(3)] if $\sigma_{\row(j_t)} = +$, then
$\col(j_t) < \col(i_{t+1})$ for each
$t=1,\dots,s-1$;
\item[(4)]
if $\sigma_{\row(j_t)} = - $, then $\col(j_t) \geq \col(i_{t+1})$
for each
$t=1,\dots,s-1$;
\item[(5)] $\row(i_1)=i$, $\row(j_s) = j$;
\item[(6)]
$\row(j_t)=\row(i_{t+1})$ for each $t=1,\dots,s-1$.
\end{itemize}
Due to conditions (1) and (2), $T_{i,j;\sigma_{1},\dots,\sigma_{m+n}}^{(r)}$ belongs to $\mathrm{F}_r U(\mathfrak p)$.

For an integer $0\leq x\leq m+n$, we set the shorthand notation 
$$T_{i,j;x}^{(r)}:= T_{i,j;\sigma_{1},\dots,\sigma_{m+n}}^{(r)}$$ 
where
\[
\sigma_{i}= \left\{
\begin{array}{ll}
- &\hbox{if $i\leq x$,}\\[3mm]
+ &\hbox{if $x<i$.}\end{array}
\right.
\]
We further define the following series for all $1\leq i,j\leq m+n$:
\begin{equation}\label{tseries}
T_{i,j;x}(u) := \sum_{r \geq 0} T_{i,j;x}^{(r)} u^{-r}
\in U(\mathfrak p) [[u^{-1}]].
\end{equation}

The following lemma can be established by exactly the same approach as \cite[Lemma~9.2]{BK2}, where the use of  super column height perfectly solves the subtle sign issue.
We omit the detail since the argument there is quite formal and does not depend on the underlying associative superalgebra in which the calculations are performed.
\begin{lemma}\cite[Lemma~9.2]{BK2}\label{ttodef}
Let $0\leq i,j,x,y \leq m+n$ be integers with $x<y$.
\begin{itemize}
\item[(1)] If $x < i \leq y < j \leq m+n$ then
$$
T_{i,j;x}(u) = \sum_{k=x+1}^y T_{i,k;x}(u) \, T_{k,j;y}(u).
$$
\item[(2)] If $x < j \leq y<i\leq m+n$ then
$$
T_{i,j;x}(u) = \sum_{k=x+1}^y  T_{i,k;y}(u) \, T_{k,j;x}(u).
$$
\item[(3)] If $x< y < i \leq m+n$ and $y < j \leq m+n$, then
$$
T_{i,j;x}(u) = T_{i,j;y}(u)
+ \sum_{k,\ell=x+1}^y  T_{i,k;y}(u) \, T_{k,\ell;x}(u) \, T_{\ell,j;y}(u).
$$
\item[(4)]
If $x < i \leq y\leq m+n$ and $x < j \leq y$, then
$$
\sum_{k=x+1}^y  T_{i,k;x}(u) \, T_{k,j;y}(u) = -\delta_{i,j}.
$$
\end{itemize}
\end{lemma}

Define an invertible $(m+n)\times (m+n)$ matrix with entries in $U(\mathfrak{p})[[u^{-1}]]$ by
$$T(u) := \big( T_{i,j;0}(u) \big)_{1\leq i,j\leq m+n}$$ 
Fix a composition $\mu = (\mu_1,\mu_{2},\ldots,\mu_{z})$ of $m+n$. 
Applying the Gauss decomposition in \textsection \ref{preY}, we have 
$$T(u) = F(u)D(u)E(u)$$ 
where $D(u)$ is a diagonal block matrix, $E(u)$ is an upper unitriangular block matrix, and $F(u)$ is a lower unitriangular block matrix, with respect to $\mu$.

The diagonal blocks of $D(u)$ define matrices $D_1(u),\ldots,D_{z}(u)$, the upper diagonal blocks of $E(u)$ define matrices $E_{1,2}(u),\ldots,E_{z-1,z}(u)$, and the lower diagonal matrices of $F(u)$ define matrices $F_{2,1}(u),\dots,F_{z,z-1}(u)$, respectively. Set $E_{b}(u)=E_{b,b+1}(u)$, $F_{b}(u)=F_{b+1,b}(u)$ for $1\leq b\leq z-1$ and $D^{\prime}_a(u) := D_a(u)^{-1}$ for all $1\leq a\leq z$.
The entries of these matrices in turn define the following series:
\begin{align*}
D_{a;i,j}(u) &= \sum_{r \geq 0} D_{a;i,j}^{(r)} u^{-r},\quad
& D^{\prime}_{a;i,j}(u) &= \sum_{r \geq 0} 
D_{a;i,j}^{\prime(r)} u^{-r},\\
E_{b;h,k}(u) &= \sum_{r \geq 1} E_{b;h,k}^{(r)} u^{-r},\quad
&F_{b;k,h}(u) &= \sum_{r \geq 1} F_{b;k,h}^{(r)} u^{-r},
\end{align*}
for all $1\leq a\leq z$, $1\leq b\leq z-1$, $1\leq i,j \leq \mu_a$, $1\leq h\leq \mu_b$, $1\leq k\leq \mu_{b+1}$. 

Nevertheless, all of these elements, depending on the fixed choice of $\mu$, are parallel to the elements in $Y_{m|n}$ 
with the same notations given in \textsection \ref{preY}, except that
the elements defined here belong to $U(\mathfrak{p})$.

\begin{theorem}\label{ttodefthm}
With $\mu = (\mu_1,\ldots,\mu_{z})$ be fixed as above.  For any admissible indices
$a,b,i,j,h,k$, we have
\begin{align*}
D_{a;i,j}(u) &= T_{\mu_1+\cdots+\mu_{a-1}+i,\mu_1+\cdots+\mu_{a-1}+j;
\mu_1+\cdots+\mu_{a-1}}(u),\\
D^{\prime}_{a;i,j}(u) &=
-T_{\mu_1+\cdots+\mu_{a-1}+i,\mu_1+\cdots+\mu_{a-1}+j;
\mu_1+\cdots+\mu_a}(u),\\
E_{b;h,k}(u) &= T_{\mu_1+\cdots+\mu_{b-1}+h,\mu_1+\cdots+\mu_{b}+k;
\mu_1+\cdots+\mu_{b}}(u),\\
F_{b;k,h}(u) &= T_{\mu_1+\cdots+\mu_{b}+k,\mu_1+\cdots+\mu_{b-1}+h;
\mu_1+\cdots+\mu_{b}}(u).
\end{align*}
\end{theorem}

\begin{proof}
Note that it suffices to show the identities for $D, E$ and $F$, since the one for $D^\prime$ follows from the one for $D$ and Lemma \ref{ttodef}(4). We prove our statement by induction on the length of $\mu$. 
The initial case is $\mu=(m+n)$, which is trivial since $T(u)=D_1(u)$.

Now let $\mu=(\mu_1,\ldots,\mu_{z})$ be a composition of length $z\geq 2$.
Define a new composition $\nu=(\nu_1,\ldots,\nu_{z-1})$ of length $z-1$ by setting $\nu_i=\mu_i$ for all $1\leq i\leq z-2$ and $\nu_{z-1} = \mu_{z-1}+\mu_{z}$; that is, merge the last two parts of $\mu$. By the induction hypothesis, we have
\begin{align*}
{^\nu}D_{a}(u) &= \left( T_{\nu_1+\cdots+\nu_{a-1}+i,\nu_1+\cdots+\nu_{a-1}+j;
\nu_1+\cdots+\nu_{a-1}}(u)\right)_{1 \leq i,j \leq \nu_a},\, \forall \,1\leq a \leq z-1,\\
{^\nu}E_{b}(u) &= \left( T_{\nu_1+\cdots+\nu_{b-1}+h,\nu_1+\cdots+\nu_{b}+k;
\nu_1+\cdots+\nu_{b}}(u)\right)_{1 \leq h \leq \nu_b, 1 \leq k \leq \nu_{b+1}},\,\forall \,1\leq b\leq z-2,\\
{^\nu}F_{b}(u) &= \left( T_{\nu_1+\cdots+\nu_{b}+k,\nu_1+\cdots+\nu_{b-1}+h;
\nu_1+\cdots+\nu_{b}}(u)\right)_{1 \leq k \leq \nu_{b+1}, 1 \leq h \leq \nu_{b}},\,\forall \,1\leq b\leq z-2,
\end{align*}
where we add a superscript $\nu$ to emphasize that these elements are defined with respect to $\nu$. Note that 
${^\nu}D_{a}(u) = {^\mu}D_{a}(u)$ for all $1\leq a\leq z-2$ and ${^\nu}E_{b}(u) = {^\mu}E_{b}(u)$, ${^\nu}F_{b}(u) = {^\mu}F_{b}(u)$ for all $1\leq b \leq z-3$.

Moreover, by Lemma \ref{split}, $^\mu E_{z-2}(u)$ equals to the submatrix consisting of the first $\mu_{z-1}$ columns of $^\nu E_{z-2}(u)$, while $^\mu F_{z-2}(u)$ equals to the submatrix consisting of the top $\mu_{z-1}$ rows of $^\nu F_{z-2}(u)$. Both of them are of the form described in the theorem.
It remains to check the identities for ${^\mu}D_{z-1}(u)$, ${^\mu}D_{z}(u)$, ${^\mu}E_{z-1}(u)$ and ${^\mu}F_{z-1}(u)$.

Define matrices $P,Q,R$ and $S$ by
\begin{align*}
P &= \left( T_{\mu_1+\cdots+\mu_{z-2}+i,\mu_1+\cdots+\mu_{z-2}+j;\mu_1+\cdots+\mu_{z-2}}(u)\right)_{1 \leq i,j \leq \mu_{z-1}},\\
Q &= \left( T_{\mu_1+\cdots+\mu_{z-2}+i,\mu_1+\cdots+\mu_{z-2}+\mu_{z-1}+j;\mu_1+\cdots+\mu_{z-2}+\mu_{z-1}}(u)\right)_{1 \leq i \leq \mu_{z-1}, 1 \leq j \leq \mu_{z}},\\
R &= \left( T_{\mu_1+\cdots+\mu_{z-2}+\mu_{z-1}+i,\mu_1+\cdots+\mu_{z-2}+j;\mu_1+\cdots+\mu_{z-2}+\mu_{z-1}}(u)\right)_{1 \leq i \leq \mu_{z}, 1 \leq j \leq \mu_{z-1}},\\
S &= \left( T_{\mu_1+\cdots+\mu_{z-2}+\mu_{z-1}+i,\mu_1+\cdots+\mu_{z-2}+\mu_{z-1}+j;\mu_1+\cdots+\mu_{z-2}+\mu_{z-1}} (u)\right)_{1 \leq i,j \leq \mu_{z}}.
\end{align*}
By Lemma \ref{ttodef} with $x=\mu_1+\ldots+\mu_{z-2}$ and $y=\mu_1+\ldots+\mu_{z-1}$, we have
$$
^{\nu}D_{z-1}(u)
=
\left(\begin{array}{ll}I_{\mu_{z-1}}&0\\ R&I_{\mu_{z}}\end{array}\right)
\left(\begin{array}{ll} P&0\\0& S\end{array}\right)
\left(\begin{array}{ll}I_{\mu_{z-1}}& Q\\0&I_{\mu_{z}}\end{array}\right)
= 
\left(
\begin{array}{cc}
 P & PQ\\
 RP & S+RPQ
\end{array}\right).
$$
Now the explicit descriptions of the matrices $^{\mu}D_{z-1}(u)$, $^\mu D_{z}(u)$, $^\mu E_{z-1}(u)$ and $^\mu F_{z-1}(u)$ follows from Lemma \ref{split}, which completes the induction argument.
\end{proof}

In the extreme case that $\mu = (1^{m+n})$, we write simply $D_i^{(r)}, D_i^{\prime(r)},E_j^{(r)}$ and $F_j^{(r)}$ for the elements $D_{i;1,1}^{(r)}$, $D_{i;1,1}^{\prime(r)}$, $E_{j;1,1}^{(r)}$ and $F_{j;1,1}^{(r)}$ of $U(\mathfrak{p})$ for all $1\leq i\leq m+n$, $1\leq j\leq m+n-1$, $r\geq 1$, respectively.

\begin{corollary}\label{transcor}
$D_i^{(r)} = T_{i,i;i-1}^{(r)}$, $E_j^{(r)} = T_{j,j+1;j}^{(r)}$, $F_j^{(r)} = T_{j+1,j;j}^{(r)}$ and $D_i^{\prime(r)} = -T_{i,i;i}^{(r)}$.
\end{corollary}

\section{Main theorem}\label{mainsec}
 Let $\pi$ be a pyramid associated with a $0^m1^n$-sequence $\bo$ which corresponds to a good pair in $\gl_{M|N}$ and let $(\sigma,\ell,\bo)$ be the triple associated to $\pi$ given by Proposition~\ref{SPbij}.
Let $Y_{m|n}^\ell(\sigma)$ denote the shifted super Yangian of level $\ell$ associated to $\pi$ equipped with the canonical filtration and let $\W_\pi$ denote the finite $W$-superalgebra associated to $\pi$ equipped with the Kazhdan filtration .

Suppose also that $\mu = (\mu_1, \dots, \mu_{z})$ is an admissible shape for $\sigma$, and recall the shorthand notations $s_{a,b}^\mu$ and $p_a^\mu$ from (\ref{sabmu}) and (\ref{pamu}). We have the elements $D_{a;i,j}^{(r)}$, $D_{a;i,j}^{\prime(r)}$, $E_{b;h,k}^{(r)}$ and $F_{b;k,h}^{(r)}$ of $U(\mathfrak{p})$ defined by Theorem~\ref{ttodefthm} according to this fixed shape $\mu$. On the other hand, we also have the parabolic generators $D_{a;i,j}^{(r)}$, $D_{a;i,j}^{\prime(r)}$, $E_{b;h,k}^{(r)}$ and $F_{b;k,h}^{(r)}$ in $Y_{\mu}^\ell(\sigma)$ as defined in \textsection~\ref{truncation}. We are ready to present the main result of this article.

\begin{theorem}\label{main}
Let $\pi$ be a pyramid and let $(\sigma,\ell,\bo)$ be the 
corresponding triple given by Proposition {\em\ref{SPbij}}. 
For any shape $\mu = (\mu_1,\dots,\mu_{z})$ admissible to $\sigma$, 
there exists is a unique isomorphism 
$Y_{m|n}^\ell(\sigma) \stackrel{\sim}{\rightarrow} \W_\pi$ 
of filtered superalgebras such that the generators
\begin{align*}
&\{D_{a;i,j}^{(r)} \, | \, 1 \leq a \leq z,1 \leq i,j \leq \mu_a, r > 0\},\\
&\{E_{b;h,k}^{(r)} \, | \, 1 \leq b < z, 1 \leq h \leq \mu_b, 1 \leq k \leq \mu_{b+1}, r > s_{b,b+1}^\mu \},\\
&\{F_{b;k,h}^{(r)} \, | \, 1 \leq a < z, 1 \leq k \leq \mu_{b+1}, 1 \leq h \leq \mu_{b}, r > s_{b+1,b}^\mu\}
\end{align*}
of $Y_{\mu}^\ell(\sigma)$ are mapped to corresponding elements of $U(\mathfrak{p})$ denoted by the same symbols. In particular, these elements of $U(\mathfrak{p})$ are $\mathfrak{m}$-invariants and they form a generating set for $\W_\pi$.
\end{theorem}

Similar to the argument in \cite{BK2}, the proof of Theorem \ref{main} is processed by induction on the number $\ell-t$, where $\ell$ is the length of the bottom row and $t$ is the length of the top row of $\pi$. 
Our initial case is $\ell=t$.
In this case, the pyramid is of rectangular shape so the associated shift matrix is the zero matrix. Hence the shifted super Yangian is the whole $Y_{m|n}$ itself, and its quotient is exactly the truncated super Yangian $Y_{m|n}^\ell$.
As mentioned in \textsection \ref{introd} , the statement of the theorem in this special case was firstly established in \cite{BR}; see also \cite{Pe2} for an approach similar to our setting here.

Assume that our pyramid $\pi$ is not of rectangular shape so that $\ell\geq 2$ and $\ell-t>0$.
By induction on the length of the shape and Lemma~\ref{split}, it suffices to prove the special case when $\mu$ is the
minimal admissible shape for $\sigma$.

Let $H$ denote the absolute height of the shortest column of $\pi$.
Since $\pi$ is a pyramid, either $H=|q_1|$ or $H=|q_\ell|$. There are two cases:
\begin{itemize}
\item Case R: $H=|q_\ell|\leq |q_1|$.
\item Case L: $H=|q_1|< |q_\ell|$.
\end{itemize}
We will explain the proof of Case R in detail and only sketch the proof of Case L, which can be obtained by a very similar argument with mild modifications.

From now on we assume that Case R holds.
Recall that we numbered the boxes of $\pi$ using the index set 
$$I:=\lbrace 1<\cdots<M<\ovl{1}<\ldots<\ovl{N} \rbrace$$ 
in the standard way: down columns from left to right, where $i$ (respectively, $\ovl{i}$) stands for the boxes labeled with $+$ (respectively, $-$). Suppose that there are $p$ (respectively, $q$) boxes labeled with $+$ (respectively, $-$) in the right-most column of $\pi$. Since $\mu$ is minimal admissible, we have $H=p+q=\mu_z$.

Let $\dot\pi$ be the pyramid obtained by removing the right-most column of $\pi$.
We know that the removed boxes of $\pi$ are numbered with 
$$M-p+1,M-p+2,\ldots,M,\ovl{N-q+1}, \ovl{N-q+2}, \ldots, \ovl{N},$$
and their order in the right-most column is determined by $\bo_z$, the last $H$ digits of the $0^m1^n$-sequence $\bo$.

By our assumption, the bottom $H$ rows of $\pi$ forms a rectangle, call it $\pi_H$.
A key observation \cite[Remark~3.5]{Pe2} is that permuting the rows of the rectangle $\pi_H$ will not change the corresponding even good pair $(e_\pi, h_\pi)$; see also Remark \ref{swap}. 
Although our argument in fact works in general, for convenience, we assume that the last $H$ digits of $\bo$ is the standard one: 
$$\bo_z=\stackrel{p}{\overbrace{0\cdots 0}} \stackrel{q}{\overbrace{1\cdots 1}}.$$
As a result, the right-most two columns of $\pi$ is of the form\\
\begin{center}
\newcolumntype{K}{\columncolor[gray]{0.8}\raggedright}
\begin{tabular}{|c|c|} 
  \multicolumn{1}{|c|}{ $\vdots$ } & \multicolumn{1}{c}{ } \\
  \hline
  $M-2p+1$ & $M-p+1$ \\  
 \hline
  $M-2p+2$ & $M-p+2$ \\  
 \hline
  $\vdots$ & $\vdots$ \\  
 \hline
  $M-p$ & $M$ \\  
 \hline
  \cellcolor[gray]{0.8}$\ovl{N-2q+1}$ & \cellcolor[gray]{0.8}$\ovl{N-q+1}$ \\  
 \hline
  \cellcolor[gray]{0.8}$\ovl{N-2q+2}$ & \cellcolor[gray]{0.8}$\ovl{N-q+2}$ \\  
 \hline
  \cellcolor[gray]{0.8}$\vdots$ & \cellcolor[gray]{0.8}$\vdots$ \\  
 \hline
  \cellcolor[gray]{0.8}$\ovl{N-q}$ & \cellcolor[gray]{0.8}$\ovl{N}$ \\  
 \hline
    \end{tabular}\\[1cm]
\end{center}

Let $\dot\sigma = (\dot s_{i,j})_{1 \leq i,j \leq m+n}$ be the shift matrix defined by ($\ref{babyr1}$) where its associated pyramid is $\dot\pi$. Define $\dot{\mathfrak{p}}, \dot{\mathfrak{m}}$ and $\dot e$ in $\dot{\mathfrak{g}} = \mathfrak{gl}_{M-p|N-q}$ according to (\ref{mpdef}) and (\ref{edef}) and let $\dot\chi:\dot{\mathfrak{m}}\rightarrow \C$ be the character $x \mapsto (x,\dot e)$.

Let $\dot D_{a;i,j}^{(r)}, \dot{D}_{a;i,j}^{\prime(r)}$, $\dot E_{b;h,k}^{(r)}$ and $\dot F_{b;k,h}^{(r)}$ denote the elements of $U(\dot{\mathfrak{p}})$ as defined in \textsection \ref{Inv} associated to the same shape $\mu$, which is admissible for both of $\sigma$ and $\dot \sigma$. By the induction hypothesis, Theorem~\ref{main} holds for $\dot\pi$, so the following elements of $U(\dot{\mathfrak{p}})$ are invariant under the $\dot{\chi}$-twisted action of $\dot{\mathfrak{m}}$; in other words, they belong to the finite $W$-superalgebra $\W_{\dot\pi}$:
\begin{align*}
&\{\dot D_{a;i,j}^{(r)}, \, \dot{D}_{a;i,j}^{\prime(r)}\} \text{ for } 1 \leq a \leq z, \,\, 1 \leq i,j \leq \mu_a \text{ and } r > 0;\\
&\{\dot E_{b;h,k}^{(r)}\} \text{ for } 1 \leq b \leq z-1,\,\, 1 \leq h \leq \mu_a,\,\, 1 \leq k \leq \mu_{a+1} \text{ and } r > s_{b,b+1}^\mu - \delta_{b+1,z};\\
&\{\dot F_{b;k,h}^{(r)} \}\text{ for } 1 \leq b \leq z-1,\,\, 1 \leq k \leq \mu_{a+1},\,\, 1 \leq h \leq \mu_{a} \text{ and } r > s_{b+1,b}^\mu.
\end{align*}

We embed $U(\dot\g)$ into $U(\g)$ in the following manner: 
for all $i,j$ in the index set 
$$\dot I:=\lbrace  1,\ldots,M-p,\ovl{1},\ldots, \ovl{N-q} \rbrace,$$ 
the generators $\tilde e_{ij}$ of $U(\dot \g)$ defined by (\ref{etil}) with respect to the pyramid $\dot\pi$ are assigned to the generators $\tilde e_{ij}$ of $U(\g)$ defined with respect to $\pi$.

This embedding in turns embeds $U(\dot{\mathfrak{p}})$ into $U(\mathfrak{p})$ and $\dot{\mathfrak{m}}$ into $\mathfrak{m}$, respectively. Moreover, the character $\dot\chi$ of $\dot{\mathfrak{m}}$ is precisely the restriction of the character $\chi$ of $\mathfrak{m}$. As a consequence, the $\dot{\chi}$-twisted action of $\dot{\mathfrak{m}}$ on $U(\dot{\mathfrak{p}})$ equals to the restriction of the $\chi$-twist action of $\mathfrak{m}$ on $U(\mathfrak{p})$.

For convenience, define the index sets 
$$\mathtt{J_1}=\{M-p+i\,|\, 1\leq i\leq p\}\cup \{\ovl{N-q+j}\,|\,1\leq j\leq q\},$$
$$\mathtt{J_2}=\{M-2p+i\,|\, 1\leq i\leq p\}\cup \{\ovl{N-2q+j}\,|\,1\leq j\leq q\}.$$ 
Note that they are the numbers appearing in the right-most and the second right-most columns of the rectangle $\pi_H$, respectively.

Define the bijection $R_1: \{1,2,\ldots, p+q\}\rightarrow \mathtt{J_1}$ by setting $R_1(f)$ to be the number assigned to the $f$-th box in the right-most column of the rectangle $\pi_H$.
Similarly, define the bijection $R_2: \{1,2,\ldots, p+q\}\rightarrow \mathtt{J_2}$ which assigns $R_2(f)$ to be the number appearing in the left of $R_1(f)$.
For example, $R_1(1)=M-p+1$, $R_1(p+q)=\ovl{N}$ and $R_2(p+q)=\ovl{N-q}$.
In particular, define 
\begin{equation}\label{defeta}
\eta:\mathtt{J_1}\rightarrow \{1,2,\ldots, p+q\}
\end{equation}
to be the inverse map of $R_1$.

The relations between the elements $D_{a;i,j}^{(r)}$, $E_{b;h,k}^{(r)}$, $F_{b;k,h}^{(r)}$ of $U(\mathfrak{p})$ given by $\pi$ and the elements $\dot D_{a;i,j}^{(r)}$, $\dot E_{b;h,k}^{(r)}$, $\dot F_{b;k,h}^{(r)}$ of $U(\dot{\mathfrak{p}})$ given by $\dot{\pi}$ are described in the following lemma, which is probably the most crucial step in the proof of our main theorem.

\begin{lemma}\label{sbabyr1}
The following equations hold for all  
$1\leq a \leq z$, $1\leq b\leq z-1$,
$1\leq i,j\leq \mu_a$,
$1\leq h\leq \mu_b$,
$1\leq k\leq \mu_{b+1}$,
all $r>0$ that makes sense,
and any fixed $1 \leq g \leq H$:
\begin{align}\label{sbr11}\notag
D_{a;i,j}^{(r)} & = \dot D_{a;i,j}^{(r)}\\
 &+\delta_{a,z}
\left(  \sum_{f=1}^H (-1)^{\pa{f}_z}
\dot D_{a;i,f}^{(r-1)} \tilde e_{R_1(f),R_1(j)}
+
[\dot D_{a;i,g}^{(r-1)}, \tilde e_{R_2(g),R_1(j)}]\right),\\
E_{b;h,k}^{(r)} & = \dot E_{b;h,k}^{(r)} + \delta_{b+1,z}\left(
 \sum_{f=1}^{H} (-1)^{\pa{f}_z} \dot E_{b;h,f}^{(r-1)} \tilde e_{R_1(f),R_1(k)}
+ [\dot E_{b;h,g}^{(r-1)}, \tilde e_{R_2(g),R_1(k)}]\right),\label{sbr12}\\
F_{b;k,h}^{(r)} & = \dot F_{b;k,h}^{(r)},
\end{align}
where for  {\em (\ref{sbr12})} we are assuming that $r > s_{z-1,z}^\mu$ if $b+1=z$.
\end{lemma}

\begin{proof}
It can be observed from the explicit description of the elements $T_{i,j;x}^{(r)}$ in (\ref{tdef}) with the help from Theorem~\ref{ttodefthm} together with our assumption on the right-most two columns of the rectangle $\pi_H$.
\end{proof}

The inductive descriptions provided in Lemma~\ref{sbabyr1}, together with the induction hypothesis, allow us to deduce the following several lemmas and eventually to show that the elements $\Daij$, $\Ebhk$ and $\Fbkh$ of $U(\mathfrak{p})$ 
are $\mathfrak{m}$-invariants when the indices are appropriate.

\begin{lemma}\label{demint1}
The following elements of $U(\mathfrak{p})$ are $\mathfrak{m}$-invariant:
\begin{itemize}
\item[(i)] $\Daij$ and $\Dpaij$
for  $1 \leq a \leq z-1$, $1 \leq i,j \leq \mu_a$ and $r > 0$;
\item[(ii)] $\Ebhk$ for 
$1 \leq a \leq z-2$, $1 \leq h \leq \mu_b$, $1 \leq k \leq \mu_{b+1}$ and $r>s_{b,b+1}^\mu$;
\item[(iii)] $\Fbkh$ for $1 \leq a \leq z-1$, $1 \leq k \leq \mu_{b+1}$, $1 \leq h \leq \mu_{b}$ 
and $r > s_{b+1,b}^\mu$.
\end{itemize}
\end{lemma}
\begin{proof}
All of these elements in $U(\mathfrak{p})$ coincide with the elements with the same name in $U(\dot{ \mathfrak{p}})$ by Lemma \ref{sbabyr1}. Hence they are 
$\dot{\mathfrak{m}}$-invariant by the induction hypothesis.
Define $$\mathfrak{\dot{m}^c}:=\mathfrak{m}\backslash\mathfrak{\dot{m}}.$$
It remains to show that these elements are invariant under the $\chi$-twisted action 
for all $\widetilde e_{f,g}$ in $\mathfrak{\dot{m}^c}$ only.
Note that $\widetilde e_{f,g}\in \mathfrak{\dot{m}^c}$ if and only if $g\in\dot{I}$ and $f\in\mathtt{J_1}$.

By Theorem \ref{ttodefthm} and (\ref{tdef}) again, all elements in the description of the lemma are linear combinations of supermonomials of the form $\tilde e_{i_1,j_1} \cdots \tilde e_{i_r,j_r}$ in $U(\dot{\mathfrak{p}})$ with $i_s\in \dot I$ and $j_s \in \dot{I}\backslash\mathtt{J_2}$ for all $1\leq s \leq r$.

By (\ref{chidef}), $\chi(\tilde{e}_{f,g}) = 0$ for all $g \in \dot{I}\backslash\mathtt{J_2}$ and $f\in\mathtt{J_1}$.
This implies that all such supermonomials are invariant under the $\chi$-twisted action of all $\widetilde e_{f,g}\in \mathfrak{\dot{m}^c}$ and our lemma follows.
\end{proof}

\begin{lemma}\label{demdot}
The following elements of $U(\mathfrak{p})$ are $\dot{\mathfrak{m}}$-invariant:
 \begin{enumerate}
  \item  $D_{z;i,j}^{(r)}$ for $1\leq i,j\leq \mu_{z}$ and $r>0$.
  \item  $E_{z-1;i,j}^{(r)}$ for $1\leq i\leq \mu_{z-1}$, $1\leq j\leq \mu_{z}$ and $r>s_{z-1,z}^\mu$.
 \end{enumerate}
\end{lemma}

\begin{proof}
(1) By (\ref{sbr11}), we obtain 
\[
D_{z;i,j}^{(r)}  = \dot D_{z;i,j}^{(r)} 
+\sum_{f=1}^H (-1)^{\pa{f}_z}
\dot D_{z;i,f}^{(r-1)} \tilde e_{R_1(f),R_1(j)}+
[\dot D_{z;i,g}^{(r-1)}, \tilde e_{R_2(g),R_1(j)}].
\]
For any $x\in\dot{\mathfrak{m}}$, we have
$[x,\widetilde e_{R_1(f),R_1(j)}]=0=[x,\widetilde e_{R_2(g),R_1(j)}]$. 
Using this result together with the induction hypothesis, one deduces that $\pr_\chi([x,D_{z;i,j}^{(r)}])=0$. 
The proof of (2) is similar by starting with (\ref{sbr12}).
\end{proof}

\begin{lemma}\label{de12dotminv}
 \begin{enumerate}
  \item  $D_{z;i,j}^{(1)}$ is $\mathfrak{\dot{m}^c}$-invariant for all $1\leq i,j\leq \mu_{z}$.
  \item  Suppose $s_{z-1,z}^{\mu}=1$. 
  Then $D_{z;i,j}^{(2)}$ is $\mathfrak{\dot{m}^c}$-invariant 
  for all $1\leq i,j\leq \mu_{z}$. 
  \item  Suppose $s_{z-1,z}^{\mu}=1$. 
  Then $E_{z-1;h,k}^{(2)}$ is $\mathfrak{\dot{m}^c}$-invariant 
  for all $1\leq h\leq \mu_{z-1}$ and $1\leq k\leq \mu_{z}$.
 \end{enumerate}
\end{lemma}
\begin{proof}
We only give the detail of the proof of (1) here, where (2) and (3) can be deduced in a similar fashion.

By Theorem \ref{ttodefthm}, (\ref{tdef}) and (\ref{sbr11}), we have
\begin{equation*}
D_{z;i,j}^{(1)}=\dot{D}_{z;i,j}^{(1)}+(-1)^{\pa{i}_z}\tilde{e}_{R_1(i),R_1(j)}
=\sum_{1\leq k\leq \ell-1} \big( \sum_{p_k,q_k} (-1)^{\pa{i}_z}\tilde{e}_{p_k,q_k}\big)
+(-1)^{\pa{i}_z}\tilde{e}_{R_1(i),R_1(j)},
\end{equation*}
where the second sum is taken over all $p_k$, $q_k\in\dot{I}$ satisfying the following conditions
\begin{enumerate}
\item[(i)] $\col(p_k)=\col(q_k)=k$, 
\item[(ii)] $\row(p_k)=\mu_1+\cdots+\mu_{z-1}+i$, 
\item[(iii)] $\row(q_k)=\mu_1+\cdots+\mu_{z-1}+j$.
\end{enumerate}
Let $\tilde{e}_{f,g}\in\mathfrak{\dot{m}^c}$ be arbitrary given so that we have $g\in\dot{I}$ and $f\in \mathtt{J_1}$. 

Suppose first that $\row(g) \neq\mu_1+\ldots+\mu_{z-1}+i$.
Then we have
$[ \tilde{e}_{f,g}, \tilde{e}_{p_k,q_k} ]=0$ 
for any $p_k$, $q_k$ appearing in the sum. 
Moreover, $[\tilde{e}_{f,g},\tilde{e}_{R_1(i),R_1(j)}]=\pm\del_{f,R_1(j)}\tilde{e}_{R_1(i),g}$, which belongs to the kernel of $\chi$ by (\ref{chidef}). It follows that $\pr_\chi([\tilde{e}_{f,g}, D_{z;i,j}^{(1)}])=0$.

Assume now that $\row(g)=\mu_1+\ldots+\mu_{z-1}+i$. 
Then $g$ equals exactly one $p_k$ appearing in the sum and hence 
$$
\big[\tilde{e}_{f,g}, \sum_{1\leq k\leq \ell-1} 
\big( \sum_{p_k,q_k\in\dot{I}} (-1)^{\pa{i}_z}\tilde{e}_{p_k,q_k}\big)\big]
=(-1)^{\pa{i}_z}\tilde{e}_{f,q_k}
$$
for a certain $1\leq k\leq \ell-1$.

Suppose in addition that $\col(q_k)\neq \ell-1$.
Then $\tilde{e}_{f,q_k}$ belongs to $\ker\chi$ by (\ref{chidef}).
Also, since $g=p_k$ and $\col(q_k)=\col(p_k)=\ell-1$, the term 
$$[\tilde{e}_{f,g},\tilde{e}_{R_1(i),R_1(j)}]=\pm\del_{f,R_1(j)}\tilde{e}_{R_1(i),g}$$ 
belongs to the $\ker\chi$. Then we have $\pr_\chi[\tilde{e}_{f,g}, D_{z;i,j}^{(1)}]=0$.

Finally, assume that $\row(g)=\mu_1+\ldots+\mu_{z-1}+i$ and $\col(q_k)=\ell-1$.
It implies that $g=p_k=R_2(i)$. By definition, we have
\[[\tilde{e}_{f,R_2(i)},D_{z;i,j}^{(1)}]=(-1)^{\pa{i}_z}\tilde{e}_{f,R_2(j)}+\del_{f,R_1(j)}(-1)^{1+\pa{j}_z}\tilde{e}_{R_1(i),R_2(i)},\]
which belongs to the kernel of $\chi$ by (\ref{chidef}). This completes the proof of (1). 

\end{proof}

\begin{lemma}\label{derdotminv}
Suppose that $s_{z-1,z}^{\mu}=1$. Then the following identities hold in $U(\mathfrak{p})$ for $r>1$:
 \begin{enumerate}
   \item  \[E_{z-1;h,k}^{(r+1)}=(-1)^{\pa{g}_{z-1}}[D_{z-1;h,g}^{(2)},E_{z-1;g,k}^{(r)}]-\sum_{f=1}^{\mu_{z-1}} D_{z-1;h,f}^{(1)}E_{z-1;f,k}^{(r)}\,,\]
   \item  \[D_{z;i,j}^{(r+1)}=(-1)^{\pa{g}_{z-1}}[F_{z-1;i,g}^{(2)},E_{z-1;g,j}^{(r)}]-\sum_{t=1}^{r+1}D_{z;i,j}^{(r+1-t)}D_{z-1;g,g}^{\prime (r)}\,.\]
 \end{enumerate}
\end{lemma}
\begin{proof}
By the induction hypothesis and (\ref{p704}), for any $r>0$ and any $1\leq g\leq \mu_{z-1}$, we have
\begin{equation}\label{l961}
[\dot{D}_{z-1;h,g}^{(2)}, \dot{E}_{z-1;g,k}^{(r)}]=
(-1)^{\pa{g}_{z-1}} \dot{E}_{z-1;h,k}^{(r+1)}+
(-1)^{\pa{g}_{z-1}}\sum_{p=1}^{\mu_z-1} \dot{D}_{z;h,p}^{(1)}\dot{E}_{z-1;p,k}^{(r)}.
\end{equation}
Also, (\ref{sbr12}) implies that
\begin{equation}\label{l962}
E_{z-1;g,k}^{(r)}=\dot{E}_{z-1;g,k}^{(r)}+\sum_{f=1}^H (-1)^{\pa{f}_{z}}\dot{E}_{z-1;g,f}^{(r-1)}\tilde{e}_{R_1(f),R_1(k)}+[\dot{E}_{z-1;g,j}^{(r-1)},\tilde{e}_{R_2(j),R_1(k)}]
\end{equation}
It is clear that $[\dot{D}_{z-1;h,g}^{(2)}, \tilde{e}_{R_1(f),R_1(k)}]=0$. Also, due to (\ref{tdef}) and Theorem \ref{ttodefthm}, the expansion of $\dot{D}_{z-1;h,g}^{(2)}$ into supermonomials will never involve any matrix unit of the form $\tilde{e}_{?,R_2(j)}$ and it follows that $[\dot{D}_{z-1;h,g}^{(2)}, \tilde{e}_{R_2(j),R_1(k)}]=0$. Computing the supercommutator of (\ref{l962}) with $D_{z-1;h,g}^{(2)}=\dot{D}_{z-1;h,g}^{(2)}$ and using (\ref{l961}), we have
\begin{align*}
[D_{z-1;h,g}^{(2)}&, E_{z-1;g,k}^{(r)}]  = [\dot{D}_{z-1;h,g}^{(2)}, \dot{E}_{z-1;g,k}^{(r)}]+\sum_{f=1}^{H}(-1)^{\pa{f}_{z}}[\dot{D}_{z-1;h,g}^{(2)}, \dot{E}_{z-1;g,f}^{(r-1)}]\tilde{e}_{R_1(f),R_1(k)}\\
& +\big[[D_{z-1;h,g}^{(2)},E_{z-1;g,j}^{(r-1)}],\tilde{e}_{R_2(j),R_1(k)}\big]\\
& = (-1)^{\pa{g}_{z-1}} \dot{E}_{z;h,k}^{(r+1)}+(-1)^{\pa{g}_{z-1}}\sum_{p=1}^{\mu_{z-1}}\dot{D}_{z-1;h,p}^{(1)}\dot{E}_{z-1;p,k}^{(r)}\\
& +\sum_{f=1}^{H}(-1)^{\pa{f}_z} \Big( (-1)^{\pa{g}_{z-1}} \dot{E}_{z;h,f}^{(r+1)}
+(-1)^{\pa{g}_{z-1}}\sum_{p=1}^{\mu_{z-1}}\dot{D}_{z-1;h,p}^{(1)}\dot{E}_{z-1;p,f}^{(r)}\Big)
\tilde{e}_{R_1(f),R_1(k)}\\
& +\Big[ (-1)^{\pa{g}_{z-1}} \dot{E}_{z;h,j}^{(r+1)}
+(-1)^{\pa{g}_{z-1}}\sum_{p=1}^{\mu_{z-1}}\dot{D}_{z-1;h,p}^{(1)}\dot{E}_{z-1;p,j}^{(r)}, \tilde{e}_{R_2(j),R_1(k)}\Big].
\end{align*}
Using (\ref{sbr12}) a few times, one shows that the above equals to 
$$(-1)^{\pa{g}_{z-1}}\big(E_{z-1;h,k}^{(r+1)}+\sum_{p=1}^{\mu_{z-1}}D_{z-1;h,p}^{(1)}E_{z-1;p,k}^{(r)}\big)$$
and the equality (1) is established.

Now we deal with (2).
By the induction hypothesis and (\ref{p706}), we have
\begin{multline}\label{1221}
[\dot{F}_{z-1;i,g}^{(2)}, \dot{E}_{z-1;g,j}^{(r)}]=(-1)^{\pa{g}_{z-1}}(\sum_{t=0}^{r+1}\dot{D}_{z;i,j}^{(r+1-t)}\dot{D}_{z-1;g,g}^{\prime(t)})\\
=(-1)^{\pa{g}_{z-1}}\dot{D}_{z;i,j}^{(r+1)}+(-1)^{\pa{g}_{z-1}}\sum_{t=1}^{r+1}\dot{D}_{z;i,j}^{(r+1-t)}\dot{D}_{z-1;g,g}^{\prime(t)}.
\end{multline} 
Changing the indices in equation (\ref{l962}), we have
\begin{equation}\label{l963}
E_{z-1;g,j}^{(r)}=\dot{E}_{z-1;g,j}^{(r)}+\sum_{f=1}^H (-1)^{\pa{f}_{z}}\dot{E}_{z-1;g,f}^{(r-1)}\tilde{e}_{R_1(f),R_1(j)}+[\dot{E}_{z-1;g,h}^{(r-1)},\tilde{e}_{R_2(h),R_1(j)}]
\end{equation}

Note that the expansion of $\dot{F}_{z-1;i,g}^{(2)}$ into supermonomials will never involve any matrix unit of the forms $\tilde{e}_{?,R_1(h)}$, $\tilde{e}_{R_1(h),?}$ or $\tilde{e}_{R_2(h),?}$, and hence 
$[\dot{F}_{z-1;i,g}^{(2)},\tilde{e}_{R_1(f),R_1(j)}]=[\dot{F}_{z-1;i,g}^{(2)},\tilde{e}_{R_2(h),R_1(j)}]=0$.
As a consequence, we perform the following calculation using the fact that $F_{z-1;i,g}^{(2)}=\dot{F}_{z-1;i,g}^{(2)}$ together with (\ref{l963}):
\begin{align*}
[F_{z-1;i,g}^{(2)}, &E_{z-1;g,j}^{(r)}]  = [\dot{F}_{z-1;i,g}^{(2)}, \dot{E}_{z-1;g,j}^{(r)}]+\sum_{f=1}^{H}(-1)^{\pa{f}_z} [\dot{F}_{z-1;i,g}^{(2)}, \dot{E}_{z-1;g,f}^{(r-1)}]\tilde{e}_{R_1(f),R_1(j)}\\
& +\big[[\dot{F}_{z-1;i,g}^{(2)},\dot{E}_{z-1;g,h}^{(r-1)}],\tilde{e}_{R_2(h),R_1(j)}\big]\\
& = (-1)^{\pa{g}_{z-1}}\dot{D}_{z;i,j}^{(r+1)}
+(-1)^{\pa{g}_{z-1}}\sum_{t=1}^{r+1}\dot{D}_{z;i,j}^{(r+1-t)}\dot{D}_{z-1;g,g}^{\prime(t)}\\
& +\sum_{f=1}^{H}(-1)^{\pa{f}_z}\Big( (-1)^{\pa{g}_{z-1}}\dot{D}_{z;i,f}^{(r)}+(-1)^{\pa{g}_{z-1}}\sum_{t=1}^{r}\dot{D}_{z;i,f}^{(r-t)}\dot{D}_{z-1;g,g}^{\prime(t)} \Big)\tilde{e}_{R_1(f),R_1(j)}\\
& + \Big[ (-1)^{\pa{g}_{z-1}}\dot{D}_{z;i,h}^{(r)}+(-1)^{\pa{g}_{z-1}}\sum_{t=1}^{r}\dot{D}_{z;i,h}^{(r-t)}\dot{D}_{z-1;g,g}^{\prime(t)} ,\tilde{e}_{R_2(h),R_1(j)} \Big]
\end{align*}
Using (\ref{sbr11}) a few times, the above can be rewritten as 
$$(-1)^{\pa{g}_{z-1}}D_{z;i,j}^{(r+1)}+(-1)^{\pa{g}_{z-1}}\sum_{t=1}^{r+1} D_{z;i,j}^{(r+1-t)}\dot{D}_{z-1;g,g}^{\prime(t)}$$
and our assertion (2) follows.
\end{proof}

\begin{lemma}
Suppose $s^\mu_{z-1,z}=1$. Then 
\begin{enumerate}
  \item $D_{z;i,j}^{(r)}$ are $\mathfrak{m}$-invariant for all $r\geq 0$ and $1\leq i,j\leq \mu_{z}$. 
  \item $E_{z-1;h,k}^{(r)}$ are $\mathfrak{m}$-invariant for all $r>1$ and $1\leq h\leq \mu_{z-1}$, $1\leq k\leq \mu_{z}$.
  \end{enumerate}
\end{lemma}
\begin{proof}
By Lemma \ref{demdot}, these elements are $\dot{\mathfrak{m}}$-invariant. It remains to check that they are $\mathfrak{\dot{m}^c}$-invariant, but that follows from Lemma \ref{de12dotminv}, Lemma \ref{derdotminv} and induction on $r$.
\end{proof}

\begin{lemma}\label{des2minv}
Suppose that $s_{z-1,z}^\mu>1$. Then the following elements are invariant under the $\chi$-twisted action of $\tilde{e}_{R_1(x),R_2(y)}$ for all $1\leq x,y\leq H$.
\begin{enumerate}
\item $D_{z;i,j}^{(r)}$ for all $r\geq 2$ and $1\leq i,j\leq \mu_{z}$.
\item $E_{z-1;h,k}^{(r)}$ for all $r>s_{z-1,z}^\mu$ and $1\leq h\leq \mu_{z-1}$, $1\leq k\leq \mu_{z}$.
\end{enumerate}
\end{lemma}

\begin{proof}
Let $\ddot{\pi}$ be the pyramid obtained by deleting the right-most two columns of $\pi$. 
Define $\ddot{\mathfrak{p}}$, $\ddot{\mathfrak{m}}$ and $\ddot{e}\in\mathfrak{gl}_{M-2p|N-2q}$ as before, and embed $U(\ddot{\mathfrak{g}})$ into $U(\dot{\mathfrak{g}})$ as how we embed $U(\dot{\mathfrak{g}})$ into $U(\mathfrak{g})$. The induction hypothesis applies to the pyramid $\ddot{\pi}$ hence we know that the elements $\ddot{D}_{z;i,j}^{(r)}$ in $\W_{\ddot\pi}$ are $\ddot{\mathfrak{m}}$-invariant under the $\dot{\chi}$-twisted action.

Applying Lemma \ref{sbr11} to $\pi$ and $\dot{\pi}$, we have
\begin{equation}\label{des2minv1}
D_{z;i,j}^{(r)}  = \dot D_{z;i,j}^{(r)} 
+\sum_{f=1}^H (-1)^{\pa{f}_z}
\dot D_{z;i,f}^{(r-1)} \tilde e_{R_1(f),R_1(j)}+
[\dot D_{z;i,g}^{(r-1)}, \tilde e_{R_2(g),R_1(j)}]
\end{equation}
and
\begin{equation}\label{des2minv2}
\dot{D}_{z;i,j}^{(r)}  = \ddot{D}_{z;i,j}^{(r)} 
+\sum_{f=1}^H (-1)^{\pa{f}_z}
\ddot{D}_{z;i,f}^{(r-1)} \tilde e_{R_2(f),R_2(j)}+
[\ddot{D}_{z;i,g}^{(r-1)}, \tilde e_{R_3(g),R_2(j)}]
\end{equation}
where $R_3(g)$ is defined to be the number assigned to $g$-th box in the third right-most column of the rectangle $\pi_H$.

Substituting (\ref{des2minv2}) into (\ref{des2minv1}) and simplifying the result by (\ref{etilrel}), one deduces that for all $r\geq 2$,
$D_{z;i,j}^{(r)}=A+B+C+D+E+F+G+H$, where
\begin{align*}
A&=\ddot{D}_{z;i,j}^{(r)}, & 
B&=\sum_{k=1}^{H} (-1)^{\pa{k}_z}\ddot{D}_{z;i,k}^{(r-1)}\tilde{e}_{R_2(k),R_2(j)},\\
C&=[\ddot{D}_{z;i,g}^{(r-1)},\tilde{e}_{R_3(g),R_2(j)}], & 
D&=\sum_{k=1}^H (-1)^{\pa{k}_z}\ddot{D}_{z;i,k}^{(r-1)}\tilde{e}_{R_1(k),R_1(j)}\\
E&=\sum_{h,k=1}^H (-1)^{\pa{h}_z+\pa{k}_z}\ddot{D}_{z;i,h}^{(r-2)}\tilde{e}_{R_2(h),R_2(k)}\tilde{e}_{R_1(k),R_1(j)}, &
F&=\sum_{k=1}^H (-1)^{\pa{k}_z} \ddot{D}_{z;i,k}^{(r-2)}\tilde{e}_{R_2(k),R_1(j)},\\
G&=\sum_{k=1}^H [\ddot{D}_{z;i,g}^{(r-2)},\tilde{e}_{R_3(g),R_2(k)}]\tilde{e}_{R_1(k),R_1(j)},&
H&=[\ddot{D}_{z;i,g}^{(r-2)},\tilde{e}_{R_3(g),R_1(j)}].
\end{align*}

Let $X=\tilde{e}_{R_1(x),R_2(y)}$ for some $1\leq x,y\leq H$. Note that $X$ commutes with all elements in $U(\ddot{\mathfrak{p}})$. Using (\ref{rhodef}), (\ref{etilrel}) and (\ref{chidef}), we can explicitly compute their images under the composition $\text{pr}_\chi \circ \text{ad} X$ as follows:
\begin{align*}
\text{pr}_{\chi}([X,A])&=0,\qquad \text{pr}_{\chi}([X,B])=\del_{xj}(-1)^{1+\pa{x}_z+\pa{y}_z}\ddot{D}_{z;i,y}^{(r-1)},\\
\text{pr}_{\chi}([X,C])&=0,\qquad \text{pr}_{\chi}([X,D])=\del_{xj}(-1)^{\pa{x}_z+\pa{y}_z}\ddot{D}_{z;i,y}^{(r-1)},\\
\text{pr}_{\chi}([X,E])&=(-1)^{\pa{x}_z+\pa{y}_z}\del_{xj}(p-q)\ddot{D}_{z;i,y}^{(r-2)}
        +(-1)^{\pa{y}_z+1}\ddot{D}_{z;i,y}^{(r-2)}\tilde{e}_{R_1(x),R_1(j)}\\
&\qquad\qquad\qquad\qquad\qquad
+\del_{xj}\sum_{k=1}^{H} (-1)^{(\pa{x}_z+\pa{y}_z)(\pa{k}_z+\pa{j}_z)+\pa{k}_z} \ddot{D}_{z;i,k}^{(r-2)}\tilde{e}_{R_2(k),R_2(y)},\\
\text{pr}_{\chi}([X,F])&=-(-1)^{\pa{x}_z+\pa{y}_z}\del_{xj}(p-q)\ddot{D}_{z;i,y}^{(r-2)}
        +(-1)^{\pa{y}_z}\ddot{D}_{z;i,y}^{(r-2)}\tilde{e}_{R_1(x),R_1(j)}\\
&\qquad\qquad\qquad\qquad\qquad
-\del_{xj}\sum_{k=1}^{H} (-1)^{(\pa{x}_z+\pa{y}_z)(\pa{k}_z+\pa{j}_z)+\pa{k}_z} \ddot{D}_{z;i,k}^{(r-2)}\tilde{e}_{R_2(k),R_2(y)},\\
\text{pr}_{\chi}([X,G])&=(-1)^{\pa{y}_z+\pa{j}_z}\del_{xj}[\ddot{D}_{z;i,g_1}^{(r-2)},\tilde{e}_{R_3(g_1),R_2(f)}],\\
\text{pr}_{\chi}([X,H])&=-(-1)^{\pa{y}_z+\pa{j}_z}\del_{xj}[\ddot{D}_{z;i,g_1}^{(r-2)},\tilde{e}_{R_3(g_1),R_2(f)}].
\end{align*}
As a consequence, $\text{pr}_{\chi}([X,D_{z;i,j}^{(r)}])=0$. The proof of (2) is similar.
\end{proof}

\begin{proposition}\label{minvprop}
The following elements of $U({\mathfrak{p}})$ are $\mathfrak{m}$-invariant with respect to the $\chi$-twisted action:
\begin{align*}
&\{D_{a;i,j}^{(r)}\}_{1 \leq a \leq z,1 \leq i,j \leq \mu_a, r > 0},\\
&\{E_{b;h,k}^{(r)}\}_{1 \leq b < z, 1 \leq h \leq \mu_a, 1 \leq k \leq \mu_{a+1}, r > s_{a,b}^\mu},\\
&\{F_{b;k,h}^{(r)}\}_{1 \leq b < z, 1 \leq k \leq \mu_{a+1}, 1 \leq h \leq \mu_{a}, r > s_{b,a}^\mu}.
\end{align*}
\end{proposition}

\begin{proof}
If follows from the induction hypothesis and Lemma \ref{demint1}--Lemma \ref{des2minv}.
\end{proof}

A consequence of Proposition~\ref{minvprop} 
is that the elements in the description of Theorem~\ref{main} are actually elements of $\W_\pi$.
Moreover, by the induction hypothesis, we may identify $Y_{\mu}^{\ell-1}(\dot\sigma)=Y_{m|n}^{\ell-1}(\dot\sigma)$ with $\W_{\dot\pi}\subseteq U(\dot{\mathfrak{p}})$ and the generators $\dot D_{a:i,j}^{(r)}$, $\dot E_{b;h,k}^{(r)}$ and $\dot F_{b;k,h}^{(r)}$ in $Y_{\mu}^{\ell-1}(\dot\sigma)$ coincide with the elements of $\W_{\dot \pi}$ denoted by the same notations. Now we are going to make use of the useful monomorphism $\Delta_R:Y_{m|n}^\ell(\sigma)\rightarrow U(\dot{\mathfrak{p}})\otimes U(\gl_{p|q})$ obtained in Theorem \ref{PBWSYLpara}.

By Corollary \ref{dimcoro}, for each $d\geq 0$, we have
\begin{equation}\label{dim1}
\dim \Delta_R (F_dY_{m|n}^\ell(\sigma))=\dim F_d Y_{m|n}^\ell(\sigma)=\dim F_d S(\g^e),
\end{equation}
where $F_dS(\g^e)$ is the sum of all graded elements in $S(\g^e)$ of degree $\leq d$ with respect to the Kazhdan grading.

Define the general parabolic generators $E_{a,b;i,j}^{(r)}$ and $F_{b,a;j,i}^{(r)}$ in $F_r U(\mathfrak{p})$ by equations (\ref{eparag}) and (\ref{fparag}) recursively, where the index $k$ could be chosen arbitrarily there.
Let $X_d$ denote the subspace of $U(\mathfrak{p})$ spanned by all supermonomials in the elements
\begin{align*}
&\{D_{a;i,j}^{(r)}\}_{1\leq a\leq z, 1\leq i,j\leq \mu_a, 0\leq r\leq s_{a,a}^{\mu}},\\
&\{E_{a,b;h,k}^{(r)}\}_{1\leq a<b\leq z, 1\leq h\leq \mu_a, 1\leq k\leq \mu_b, s_{a,b}^\mu<r\leq s_{a,b}^{\mu}+p_a^\mu},\\
&\{F_{b,a;k,h}^{(r)}\}_{1\leq a<b\leq z, 1\leq k\leq \mu_b, 1\leq h\leq \mu_a, s_{b,a}^\mu<r\leq s_{b,a}^\mu+p_a^\mu} .
\end{align*}
taken in some fixed order with total degree $\leq d$. Proposition \ref{minvprop} implies that $X_d$ is a subspace of $F_d \W_\pi$.

Define a superalgebra homomorphism $\psi_R:U(\mathfrak{p})\rightarrow U(\dot{\mathfrak{p}})\otimes U(\gl_{p|q})$ by
\begin{equation}\notag
\psi_R(\tilde e_{i,j}):= \left\{
\begin{array}{ll}
\tilde e_{i,j}\otimes 1 &\hbox{if $\col(i)\leq \col (j)\leq \ell-1$,}\\
0 &\hbox{if $\col(i)\leq \ell-1, \col(j)=\ell$,}\\
1\otimes \tilde e_{\eta(i),\eta(j)} &\hbox{if $\col(i)=\col(j)=\ell$,}
\end{array}
\right.
\end{equation}
where the map $\eta$ is define in (\ref{defeta}).
By Lemma \ref{sbabyr1}, we have 
\begin{align*}
&\psi_R(D_{a;i,j}^{(r)})=\dot D_{a;i,j}^{(r)}\otimes 1+\delta_{a,z}\sum_{f=1}^H (-1)^{\pa{f}_z}\dot D_{a;i,f}^{(r-1)}\otimes \tilde e_{f,j},\\
&\psi_R(E_{b;h,k}^{(r)})=\dot E_{b;h,k}^{(r)}\otimes 1+\delta_{b+1,z}\sum_{f=1}^H (-1)^{\pa{f}_z} \dot E_{b;h,f}^{(r-1)}\otimes \tilde e_{f,k},\\
&\psi_R(F_{b;k,h}^{(r)})=\dot F_{b;k,h}^{(r)}\otimes 1.
\end{align*}

Comparing this with Theorem \ref{baby1}$( 1)$ and recalling the PBW basis for $Y_{m|n}^\ell(\sigma)$ obtained in Corollary \ref{pbwbasis}, we deduce that $\psi_R(X_d)=\Delta_R(F_dY_{m|n}^\ell(\sigma))$. Combining this with (\ref{dim1}) and Corollary \ref{dimcoro}, we obtain 
\begin{equation*}
\dim F_dS(\g^e)=\dim \psi_R(X_d)\leq \dim X_d \leq \dim F_d\W_{\pi}\leq \dim F_dS(\g^e). 
\end{equation*}
Hence equalities hold everywhere so we have $X_d=F_d\W_\pi$ for each $d\geq 0$. In particular, $\psi_R:\W_\pi\rightarrow U(\dot{\mathfrak{p}})\otimes \gl_{p|q}$ is an injective homomorphism. Comparing $\psi_R$ with the map $\Delta_R$ defined in Theorem~\ref{baby1}$( 1)$, we see that $\psi_R(D_{a;i,j}^{(r)})=\Delta_R(D_{a;i,j}^{(r)})$, where the elements $D_{a;i,j}^{(r)}$ on the left-hand side are the elements of $\W_\pi$ and the elements $D_{a;i,j}^{(r)}$ on the right-hand side are the generators of $Y_{m|n}^\ell(\sigma)$. Similarly, $\psi_R(E_{b;h,k}^{(r)})=\Delta_R(E_{b;h,k}^{(r)})$ and $\psi_R(F_{b;k,h}^{(r)})=\Delta_R(F_{b;k,h}^{(r)})$ for all admissible indices $b,h,k,r$.

Finally, the composition $\psi_R^{-1}\circ \Delta_R:Y_{m|n}^\ell(\sigma)\rightarrow \W_\pi$ is exactly the filtered superalgebra isomorphism described in Theorem \ref{main} and the elements listed in Theorem~\ref{main} indeed generate $\W_\pi$. This completes the induction step of our main theorem under the assumption of Case R.

Next we sketch how to complete the induction step under the assumption of Case L. In this case, we enumerate the bricks of $\pi$ down columns {\em from right to left}. Note that different ways of enumerating are just choosing different bases to describe $\gl_{M|N}\cong \End (\mathbb{C}^{M|N})$ so we may choose the way most suitable for our purpose.

Let $\dot \pi$ denote the pyramid obtained from $\pi$ by deleting the {\em left-most} column of $\pi$. 
Let $I$, $\dot{I}$, $\mathtt{J_1}$ and $\mathtt{J_2}$ be the same index sets as defined in Case R.
It is clear that the deleted bricks are still numbered with elements in $\mathtt{J_1}$. 
Moreover, we may again assume that the left-most two columns of $\pi$ is of the form
\begin{center}
\newcolumntype{K}{\columncolor[gray]{0.8}\raggedright}
\begin{tabular}{|c|c|}
  \multicolumn{1}{c}{ } & \multicolumn{1}{|c|}{ $\vdots$ }\\
   \hline
   $M-p+1$ & $M-2p+1$  \\  
 \hline
  $M-p+2$ & $M-2p+2$   \\  
 \hline
  $\vdots$ & $\vdots$ \\  
 \hline
  $M$   &   $M-p$   \\  
 \hline
  \cellcolor[gray]{0.8}$\ovl{N-q+1}$    & \cellcolor[gray]{0.8}$\ovl{N-2q+1}$ \\  
 \hline
  \cellcolor[gray]{0.8}$\ovl{N-q+2}$ & \cellcolor[gray]{0.8}$\ovl{N-2q+2}$ \\  
 \hline
  \cellcolor[gray]{0.8}$\vdots$ & \cellcolor[gray]{0.8}$\vdots$ \\  
 \hline
  \cellcolor[gray]{0.8}$\ovl{N}$ & \cellcolor[gray]{0.8}$\ovl{N-q}$ \\  
 \hline
    \end{tabular}\\[1cm]
\end{center}

Similarly we define the bijection $L_1: \{1,2,\ldots, p+q\}\rightarrow \mathtt{J_1}$ by setting $L_1(f)$ to be the number assigned to the $f$-th box in the {\em left-most} column of the rectangle $\pi_H$, and define the bijection $L_2: \{1,2,\ldots, p+q\}\rightarrow \mathtt{J_2}$ by assigning $L_2(f)$ to be the number appearing in the {\em right} of $L_1(f)$.
In particular, denote by 
\begin{equation}\label{defxi}
\xi:\mathtt{J_1}\rightarrow \{1,2,\ldots, p+q\}
\end{equation}
the inverse map of $L_1$.

Let $\dot\sigma$ be the shift matrix obtained from (\ref{babyl1}), where the corresponding pyramid is exactly $\dot\pi$, and define $\dot{\mathfrak{p}}, \dot{\mathfrak{m}}, \dot e \in \dot \g :=\mathfrak{gl}_{M-p|N-q}$ via (\ref{mpdef}) and (\ref{edef}) with respect to $\dot \pi$. Note that in Case L we embed $U(\dot{\g})$ into $U(\g)$ by the {\em natural embedding}, which already sends the elements $\tilde{e}_{ij}$ of $U(\dot{\g})$ to the elements $\tilde{e}_{ij}$ of $U(\g)$ for all $i,j\in\dot{I}$.

Under the natural embedding, the superalgebra $\W_{\dot{\pi}}=U(\dot{\mathfrak{p}})^{\dot{\mathfrak{m}}}$ is a subalgebra of $U(\dot{\mathfrak{p}})\subset U(\mathfrak{p})$ and the $\dot{\chi}$-twisted action of $\dot{\mathfrak{m}}$ on $U(\dot{\mathfrak{p}})$ is exactly the same with the restriction of the $\chi$-twisted action of $\mathfrak{m}$ on $U(\mathfrak{p})$. Let $\dot D_{a;i,j}^{(r)}, \dot{D}_{a;i,j}^{\prime(r)}$, $\dot E_{b;h,k}^{(r)}$ and $\dot F_{b;k,h}^{(r)}$ denote the elements of $U(\dot{\mathfrak{p}})$ as defined in \textsection \ref{Inv} associated to the shape $\mu$ which is the minimal admissible shape of $\sigma$ but also admissible for $\dot\sigma$. By the induction hypothesis, all of these elements are $\dot{\mathfrak{m}}$-invariant.

From now we follow exactly the same idea in Case R to complete the proof. By the following crucial lemma, which is the analogue of Lemma \ref{sbabyr1}, we may express the elements $D_{a;i,j}^{(r)}, {D}_{a;i,j}^{\prime(r)}$, $E_{b;h,k}^{(r)}$ and $F_{b;k,h}^{(r)}$ in $U(\mathfrak{p})$ in terms of $\dot D_{a;i,j}^{(r)}, \dot{D}_{a;i,j}^{\prime(r)}$, $\dot E_{b;h,k}^{(r)}$ and $\dot F_{b;k,h}^{(r)}$. Then by similar case-by-case discussions and computations as before, we can prove that all of the elements $D_{a;i,j}^{(r)}, {D}_{a;i,j}^{\prime(r)}$, $E_{b;h,k}^{(r)}$ and $F_{b;k,h}^{(r)}$ are indeed $\mathfrak{m}$-invariant under our current setting in Case L. 
We provide only the most crucial lemma below since its proof and other arguments are almost identical as in the earlier case.

\begin{lemma}\label{sbabyl1}
The following equalities hold for all all admissible $a,b,i,j,h,k,r$ and any fixed $1 \leq g \leq H$:
\begin{align}\label{sbl1}\notag
D_{a;i,j}^{(r)} & = \dot D_{a;i,j}^{(r)}\\
 &+\delta_{a,z}(-1)^{\pa{i}_z}
\left(\sum_{f=1}^H \tilde e_{L_1(i),L_1(f)}
\dot D_{z;f,j}^{(r-1)} 
+
[ \tilde e_{L_1(i),L_2(g)}, \dot D_{z;g,j}^{(r-1)}]\right),\\
E_{b;h,k}^{(r)} & = \dot E_{b;h,k}^{(r)},\\
F_{b;k,h}^{(r)} & = \dot F_{b;k,h}^{(r)} + \delta_{b,z-1}(-1)^{\pa{i}_z}\left(
\sum_{f=1}^H \tilde e_{L_1(k),L_1(f)}\dot F_{z-1;f,h}^{(r-1)} 
+ [\tilde e_{L_1(k),L_2(g)}, \dot F_{z-1;g,h}^{(r-1)} ]\right),\label{sbl2}
\end{align}
where for  {\em (\ref{sbl2})} we are assuming that $r > s_{z,z-1}^\mu$ if $b=z-1$.
\end{lemma}

With the help of Lemma \ref{sbabyl1}, one can deduce that the statement of Proposition~\ref{minvprop} still holds in Case L. Finally, define a superalgebra homomorphism $\psi_L:U(\mathfrak{p})\rightarrow  U(\gl_{p|q})\otimes U(\dot{\mathfrak{p}})$ by
\begin{equation}\notag
\psi_L(\tilde e_{i,j}):= \left\{
\begin{array}{ll}
\tilde e_{\xi(i),\xi(j)}\otimes 1 &\hbox{if $\col(i)=\col (j)=1$,}\\
0 &\hbox{if $\col(i)=1, \col(j)\geq 2$,}\\
1\otimes \tilde e_{i,j} &\hbox{if $2\leq \col(i)\leq \col(j)$,}
\end{array}
\right.
\end{equation}
where the function $\xi$ is defined by (\ref{defxi}).
Using Lemma \ref{sbabyl1} again, we have that
\begin{align*}
&\psi_L(D_{a;i,j}^{(r)})=1\otimes \dot D_{a;i,j}^{(r)}+\delta_{a,z}\sum_{f=1}^H (-1)^{\pa{f}_z}\tilde e_{i,f} \otimes \dot D_{a;f,j}^{(r-1)}\\
&\psi_L(E_{b;h,k}^{(r)})=1\otimes \dot E_{b;h,k}^{(r)},\\
&\psi_L(F_{b,k,h}^{(r)})=1\otimes \dot F_{b;k,h}^{(r)}+\delta_{b+1,z}\sum_{f=1}^H (-1)^{\pa{f}_z}\tilde e_{k,f} \otimes\dot F_{b;f,h}^{(r-1)}.
\end{align*}
Using exactly the same argument as in Case R, one shows that the map $\psi_L$ is injective and the composition $\psi_L^{-1}\circ \Delta_L:Y_{m|n}^\ell(\sigma)\rightarrow \W_\pi$ gives the required isomorphism of filtered superalgebras. This completes the proof of Theorem \ref{main}.

\begin{corollary}\label{ezg}
Let $\pi$ be a pyramid corresponding to an even good pair and $\vec\pi$ be a pyramid obtained by horizontally shifting rows of $\pi$. Let $\W_\pi$ and $\W_{\vec\pi}$ denote the associated finite $W$-superalgebras, respectively. Then there exists a superalgebra isomorphism $\iota:\W_\pi\rightarrow \W_{\vec\pi}$ defined on parabolic generators with respect to an admissible shape $\mu$ by  {\em (\ref{iotadef})}. In other words, the definition of a finite $W$-superalgebra associated to an even good pair depends only on $e$ up to isomorphism.
\end{corollary}
\begin{proof}
This is an immediate consequence of (\ref{iotaiso}) and the isomorphism in Theorem \ref{main}.
\end{proof}

\begin{remark}\label{912}
A more general result of Corollary  {\em \ref{ezg}} was obtained in \cite{Zh} by a very different approach. 
It is proved that the definition of type A finite $W$-superalgebra is independent of the choices of the good $\mathbb{Z}$-grading (which may not be even) up to isomorphism, generalizing the results of \cite{BG, GG}.
\end{remark}

\end{document}